\documentclass[english,10pt,oneside]{amsart}

\usepackage[utf8]{inputenc}
\usepackage[T1]{fontenc}
\usepackage{frcursive}
\usepackage{lmodern}
\usepackage[a4paper]{geometry}
\usepackage[french]{babel}
\usepackage{xcolor}
\usepackage{hyperref}
\usepackage{csquotes}
\usepackage[backend=biber]{biblatex}
\usepackage{listings}

\usepackage{amsmath,mathtools,amssymb,mathrsfs,eucal}

\usepackage{graphicx}
\usepackage{tikz-cd}
\usepackage[all]{xy}

\usepackage{amsthm}

\usepackage{todonotes}
\usetikzlibrary{matrix,arrows}

\swapnumbers

\definecolor{darkgreen}{rgb}{0,0.45,0}

\hypersetup{
colorlinks,
linktoc=all,
linkcolor={darkgreen},
citecolor={blue!50!black},
urlcolor={blue!80!black}
}

\theoremstyle{definition}

\theoremstyle{plain}
\newtheorem{theo}{Théorème}[section]

\newtheorem{prop}[theo]{Proposition}

\newtheorem{defi}[theo]{Définition}

\newtheorem{lemm}[theo]{Lemme}

\theoremstyle{remark}
\newtheorem{rema}[theo]{Remarque}

\newtheorem{exem}[theo]{Exemple}

\numberwithin{equation}{subsubsection}
\author{Mathieu Daylies}
\address{mathieu.daylies@imj-prg.fr}
\title[Descente fidèlement plate et algébrisation en géométrie de Berkovich]{Descente fidèlement plate et algébrisation en géométrie de Berkovich}

\usepackage{stmaryrd}

\def\Spec{\operatorname{Spec}}
\def\Hom{\operatorname{Hom}}

\theoremstyle{plain}
\newtheorem{theointro}{Théorème}

\begin{document}

\maketitle

Cet article étudie les questions de descente relative à la topologie sur les espaces de Berkovich dont les flèches couvrantes sont les morphismes plats et surjectifs. Nous donnons des conditions suffisantes pour qu'une catégorie fibrée donnée soit un champ pour cette topologie. Ensuite, nous utilisons ce résultat pour montrer que le foncteur tiré en arrière de la catégorie des $S$-espaces vers la catégorie des $S'$-espaces munis de données de descente est pleinement fidèle, et nous montrons l'effectivité de certaines données de descente au dessus de $S'$. Enfin, nous montrons que la propriété d'être algébrique pour un morphisme entre l'analytification de deux schémas est une propriété locale pour la topologie sus-citée.

%Ensuite nous utilisons ce résultat pour montrer la pleine fidélité du foncteur tiré en arrière sur les objets relatifs, et nous montrons l'effectivité de certaines données de descente. Enfin, nous montrons que la propriété d'être algébrique pour un morphisme entre l'analytification de deux schémas est une propriété locale pour la topologie sus-citée.

%Ensuite nous utilisons ce résultat pour montrer la pleine fidélité du foncteur tiré en arrière sur les objets relatifs, et nous montrons l'effectivité de certaines données de descente. Enfin, blabla

%Ensuite nous utilisons ce résultat pour montrer que la catégorie des objets relatifs est un préchamp, et nous montrons qu'une certaine catégorie de flèche est un champ. 

%Dégageons une sous-catégorie pleine des objets relatifs sur laquelle le foncteur tiré en arrière induit une équivalence de catégorie 
\section{Introduction}

\subsection{Motivation}

Les résultats de descente sont depuis les travaux de l'école de Grothendieck devenus indispensables et omniprésents en géométrie algébrique moderne. Le but de cet article est d'énoncer et de démontrer des résultats similaires aux résultats de l'exposé 8 de \autocite{grothendieck2002rev} dans le cadre des espaces analytiques au sens de Berkovich. 

Une notion de platitude dans les espaces de Berkovich a été dégagée et étudiée de manière approfondie par Ducros dans \autocite{ducros2017families}, mais jusqu'à ce jour, il n'existait pas d'analogue à des résultats fondamentaux de descente fidèlement plate existant dans le cadre des schémas, comme le théorème 5.2 ou bien le théorème 2.1 de l'exposé 8 de \autocite{grothendieck2002rev}, et l'utilisation de la platitude au sens de Ducros en général n'en est qu'à ses débuts. 

Des résultats de descente dans le cadre analytique ont bien été obtenus par Bosch, Görtz et Ducros dans les articles \autocite{Coherentmodulesandtheirdescentonrelativerigidspaces} et \autocite{ducros2020devisser} mais leurs résultats portaient uniquement sur la descente de modules cohérents (sur des espaces correspondant uniquement au cas $\textit{strict}$ en théorie de Berkovich pour Bosch et Görtz). De même, Conrad et Temkin obtiennent des résultats de descente. L'article \autocite{AIF_2006__56_4_1049_0} de Conrad établit ainsi la pleine fidélité de certains tirés en arrière dans le cas strict en utilisant les techniques de Raynaud, tandis que l'article de Conrad et Temkin \autocite{conrad2019descent} étudie de manière approfondie quelles sont les propriétés des morphismes d'espaces de Berkovich qui se descendent par des changements de base variés. %(citer original de EGA??) "Montre dans le cadre des espaces de Berkovich des résultats analogue à EGAIV : de nombreuses propriétés de morphismes d'espaces analytiques sont locale vis à vis de la topologie plate surjective. "
%Nous ferons usage de families de Ducros

%La descente c'est important parce que blablabla

%toutdebutintro
\subsection{Nos méthodes}

On peut identifier deux éléments qui distinguent la théorie de la descente dans les espaces analytiques de la théorie de la descente schématique et empêchent de transposer $\mathit{verbatim}$ les techniques schématiques. Premièrement, la présence dans les produits fibrés de produits tensoriels $\textit{complétés}$, et deuxièmement le fait que la notion de quasi-cohérence est une notion délicate en géométrie non archimédienne (voir \autocite{AIF_2006__56_4_1049_0} à ce sujet), ce qui empêche de recourir à la construction du spectre d'une algèbre quasi-cohérente générale comme en théorie des schémas.

La méthode que nous utiliserons tout au long de ce texte est d'appliquer la stratégie suivie par Ducros dans \autocite{ducros2020devisser} pour démontrer la descente fidèlement plate des modules cohérents. Le cadre est le suivant : supposons que nous voulions montrer qu'un certain morphisme plat et surjectif $p:S'\to S$ entre espaces $k$-affinoïdes vérifie une propriété de descente $P$ (e.g. être un morphisme de descente pour une certaine catégorie fibrée), qui possède de bonnes propriétés à la composition.

- On étudie d'abord le cas où le morphisme est de la forme $p:\mathcal{M}(\mathcal{A}_r)\to\mathcal{M(A)}$ avec $r$ un polyrayon $k$-libre. On peut alors se baser sur des calculs explicites faisant intervenir des séries convergentes. Cela nous permet de nous ramener au cas où $p$ est un morphisme entre espaces stricts.

- On utilise ensuite le théorème 9.1.3 de \autocite{ducros2017families}, qui est l'analogue d'un résultat schématique (voir la proposition 19.2.9 de \autocite{PMIHES_1967__32__5_0}) et qui exhibe dans le cas strict l'existence de multisections quasi-finies et plates au dessus de $S$ au morphisme $p$. Dans le cas où le morphisme $p$ possède une section, c'est un point clé et classique que les propriétés de descente sont alors automatiques, ce qui permet de nous ramener au cas où le morphisme $p$ est de dimension relative nulle.

- On utilise enfin un théorème de factorisation dû à Ducros en 8.4.6 de \autocite{ducros2017families}, qui raffine dans notre cas l'analogue non archimédien du théorème principal de Zariski, démontré aussi par Ducros dans \autocite{ducros_2007}, et qui fournit une factorisation agréable d'un morphisme plat de dimension relative nulle, pour se ramener d'une part au cas où $p$ est fini, plat et surjectif, et d'autre part au cas où $p$ est le morphisme plat et surjectif correspondant à un $G$-recouvrement. Ce dernier cas est assez facile pour les propriétés étudiées ici, qui sont toujours $G$-locales. Le cas fini, plat et surjectif est alors souvent traité en utilisant la descente fidèlement plate schématique puisque dans ce cas là, le produit tensoriel et le produit tensoriel complété coïncident.

Remarquons que cette méthode générale s'applique pour des morphismes plats et surjectifs entre espaces $k$-analytiques, et ne permet pas d'étudier la question de descente relative à une extension quelconque de corps valuées $\mathcal{M}(L)\to \mathcal{M}(k)$, qui semble être une question plus difficile en général (c'est déjà le cas pour l'étude de la descente des propriétés de morphismes, voir \autocite{conrad2019descent} pour plus de détails).

\subsection{Résultats principaux et vue d'ensemble de l'article} 
%Ducros
Considérons un corps  $k$ non-archimédien complet. Dans la deuxième section, nous établissons le théorème suivant $\ref{fibfi}$, qui fournit une condition générale pour qu'une catégorie fibrée fixée definie au dessus de la catégorie $\mathsf{C}$ des espaces $k$-affinoïdes soit un champ pour la topologie dont les flèches couvrantes sont les morphismes plats et surjectif. On rappelle comme dans le paragraphe 3.1.2 de \autocite{vistoli2005notes} que la donnée d'une catégorie fibrée $\mathsf{F}\to \mathsf{C}$ et de certaines flèches cartésiennes distinguées dans $\mathsf{F}$ est équivalente à la donnée d'un pseudo-foncteur au dessus de $\mathsf{C}$ par la correspondance qui à un objet $S$ de $\mathsf{C}$ associe la fibre $\mathsf{F}(S)$ de $\mathsf{F}$ au dessus de $S$, et qu'un morphisme $p:S'\to S$ de $\mathsf{C}$ est de descente effective (resp. de descente) si le foncteur tiré en arrière de la catégorie $\mathsf{F}(S)$ vers la sous catégorie des objets de $\mathsf{F}(S')$ munis de données de descente est une équivalence de catégories (resp. pleinement fidèle).

\begin{theointro} Considérons $\Psi$ un pseudo-foncteur au dessus de $\mathsf{C}$, avec $\mathsf{C}$ la catégorie des espaces analytiques $k$-affinoïdes. On suppose que :
\begin{enumerate}
\item Pour toute algèbre $k$-affinoïde $\mathcal{A}$, et tout polyrayon $k$-libre $r\in (\mathbb{R}_+ ^*)^n$, le morphisme $\mathcal{M}(\mathcal{A}_r)\to \mathcal{M(A)}$ est un morphisme de descente effective relativement à $\Psi$.
\item Tout morphisme plat, fini et surjectif $\mathcal{M(B)}\to \mathcal{M(A)}$ est un morphisme de descente effective relativement à $\Psi$.
\item Tout $G$-recouvrement fini $\amalg_{i=1}^n S_i\to S$ de $S$ par des domaines affinoïdes en nombre fini est un morphisme de descente effective relativement à $\Psi$.

\end{enumerate}

Alors tout morphisme plat et surjectif entre espaces affinoïdes est un morphisme de descente effective relativement à $\Psi$.
\end{theointro}
On donne de même un résultat similaire pour qu'un pseudo-foncteur au dessus de $\mathsf{C}$ soit un pré-champ en $\ref{fibfibis}$.

Dans la troisième section, on applique les théorèmes précédents à plusieurs pseudo-foncteurs particuliers. Ainsi, on obtient d'abord le théorème $\ref{pf}$ qui concerne les morphismes plats et proprement surjectifs, c'est à dire les morphismes d'espaces analytiques plats dont la base est $G$-recouverte par des domaines analytiques quasi-compacts qui soient l'image d'un domaine analytique quasi-compact de la source. Cette hypothèse couvre en particulier les morphismes plats, surjectifs, topologiquement propres (universellement fermés comme applications continues), et les morphismes sans bord, plats et surjectifs.

\begin{theointro}
Considérons le pseudo-foncteur $\Psi$ qui à un espace $k$-analytique $S$ associe la catégorie des espaces $k$-analytiques au dessus de $S$. Alors les morphismes plats et proprement surjectifs sont des $\textup{morphismes de descente}$ pour $\Psi$.

Autrement dit, les morphismes plats et proprement surjectifs sont des épimorphismes effectifs universels de la catégorie des espaces $k$-analytiques ; autrement dit, si l'on se donne un morphisme plat et proprement surjectif $p:S'\to S$ entre espaces $k$-analytiques, alors le foncteur de la catégorie des $S$-espaces analytiques vers la catégorie des $S'$-espaces munis de données de descente relativement au morphisme $p$ est pleinement fidèle.
\end{theointro}

On obtient ensuite le théorème $\ref{cool}$ qui constitue un résultat d'effectivité.
Un morphisme entre espaces $k$-analytiques $p:S'\to S$ est dit $\textit{presque affinoïde}$ s'il existe un $G$-recouvrement de la base $S$ par des domaines affinoïdes dont chaque image inverse par $p$ est un domaine affinoïde de $S'$. Cette notion est introduite ici pour remédier au fait qu'il n'existe pas en théorie de Berkovich de bonne notion de morphismes affinoïdes, qui soit l'analogue des morphismes affines en théorie des schémas. En effet, Liu a montré dans \autocite{liu1990} l'existence d'un espace rigide non affinoïde mais dotés d'un morphisme presque affinoïde vers un espace $k$-affinoïde.
\begin{theointro}
Considérons le pseudo-foncteur $\Psi$ qui à un espace $k$-affinoïde $S$ associe la catégorie des espaces $k$-analytiques au dessus de $S$ dont le morphisme structural est presque affinoïde. Alors les morphismes plat et surjectifs sont des $\textup{morphismes de descente effectif}$ pour ce pseudo-foncteur.

Autrement dit, le pseudo-foncteur qui à un espace $k$-affinoïde $S$ associe $\Psi S$ est un champ pour la topologie dont les flèches couvrantes sont plates surjectives ; autrement dit si l'on se donne un morphisme plat et surjectif $p:S'\to S$ entre espaces $k$-affinoïdes, alors le foncteur de la catégorie des $S$-espaces analytiques dont le morphisme structural est presque affinoïde vers la catégorie des $S'$-espaces analytiques dont le morphisme structural est presque affinoïde et munis de données de descente relativement au morphisme $p$ est une équivalence de catégories.
\end{theointro}
La quatrième partie est dévolue à l'application des résultats précédents. En particulier, on démontre en $\ref{genera}$ une généralisation de la proposition A-1 de \autocite{rémy2009bruhattits}, et on retrouve en $\ref{retrouv}$ des résultats obtenus par Conrad et Temkin dans \autocite{conrad2019descent}.

Dans la dernière section, nous nous intéressons à la descente fidèlement plate du caractère algébrique. Un espace $k$-analytique $X$ au dessus de $\mathcal{A}$ sera dit algébrisable si c'est l'analytifié d'un $\mathcal{A}$-schéma $\mathcal{X}$. Un morphisme d'espaces $k$-analytiques entre l'analytifié de deux $\mathcal{A}$-schémas est algébrique si c'est l'analytification d'un morphisme de $\mathcal{A}$-schémas. L'objectif serait d'avoir un résultat de descente d'objets d'algébrisation, qui pourrait faire l'objet d'un futur travail. La question serait la suivante : si l'on se donne un morphisme plat et surjectif $p:\mathcal{M(B)}\to\mathcal{M(A)}$ entre espaces $k$-affinoïdes, et un $\mathcal{A}$-espace analytique $X$, est-ce que le caractère algébrisable de $X_\mathcal{B}$ entraine le caractère algébrisable de $X$ ? 

Remarquons que le $\mathcal{A}$-schéma $\mathcal{X}$ dont l'analytification fournirait $X$ n'a aucune raison d'être unique dans le cas général, mais par GAGA, il l'est dans le cas où $X$ est propre au dessus de $\mathcal{A}$, ce qui serait sans doute une hypothèse naturelle pour commencer à étudier ce problème.

Dans cet article, nous étudions uniquement la question intermédiaire suivante plus accessible :  la propriété pour un morphisme d'être algébrique est-elle locale pour la topologie plate surjective ? Nous répondons par l'affirmative à cette question avec le théorème $\ref{algmorph}$ :

\begin{theointro}
Considérons $\mathcal{M(B)}\to \mathcal{M(A)}$ un morphisme plat et surjectif entre espaces $k$-affinoïdes. Considérons maintenant $\mathcal{X}$ et $\mathcal{Y}$ deux $\mathcal{A}$-schémas localement de type fini. Alors un morphisme $f:\mathcal{X}^\mathrm{an}\to\mathcal{Y}^\mathrm{an}$ est algébrique si et seulement si son changement de base $f_\mathcal{B}:\mathcal{X}_\mathcal{B}^\mathrm{an}\to \mathcal{Y}_\mathcal{B}^\mathrm{an} $ est algébrique.

\end{theointro}

\subsection{Remerciement}
Ce travail n'aurait vu le jour sans les idées, les suggestions, l'encouragements, et l'accompagnement bienveillant d'Antoine Ducros. Qu'il en soit remercié. 
Merci aussi à Marco Maculan pour ses remarques pertinentes et à Jean-Michel Fischer pour son accompagnement mathématique et Latexien.

\subsection{Notations et rappels techniques}
\subsubsection{Le cadre général} On fixe pour la suite du texte un corps $k$ muni d'une valeur absolue ultramétrique qui peut être triviale et pour laquelle le corps $k$ est complet. Nous travaillerons avec la notion d'espace $k$-analytique au sens de Berkovich et considérerons comme connues les bases de la théorie, exposée dans \autocite{berkovich1993etale}.

\subsubsection{Topologie, $G$-topologie et topologie de Zariski} Considérons un espace $k$-analytique $X$. Il est fourni avec une topologie au sens usuel, et une topologie de Grothendieck ensembliste plus fine, appelée $G$-topologie dont les éléments sont les domaines analytiques de $X$. Le site correspondant est noté $X_G$, et $X$ est muni d'un faisceau de $k$-algèbres que nous noterons $\mathcal{O}_X$, qui est cohérent (le résultat est énoncé et démontré dans \autocite{MR2566967}). Tous les faisceaux cohérents en jeu dans cet article seront toujours définis sur le site $X_G$, et s'il n'y a pas d'ambiguité, on notera $\mathcal{O}(X)$ l'algèbre des sections globales sur $X$. Par abus de langage, on appelera encore $G$-recouvrement tout morphisme d'espace $k$-analytique de la forme $\amalg_{i\in I} X_i \to X$ pour un $G$-recouvrement $(X_i)_{i\in I}$ de $X$ par des domaines analytiques.

Si $\mathcal{J}$ est un faisceau cohérent d'idéaux sur $X$, alors on note $V(\mathcal{J})$ l'ensemble des $x\in X$ tel que $f(x)=0$ pour toute section $f$ de $\mathcal{J}$ au voisinage de $x$. Les parties de la forme $V(\mathcal{J})$ forment les fermés d'une topologie plus grossière que la topologie usuelle sur $X$, et appelée topologie de Zariski sur $X$. Lorsque $X$ est $k$-affinoïde, la topologie de Zariski sur $X$ est l'image réciproque de la topologie de Zariski sur $X^\textrm{al}:=\Spec{A}$ par l'application naturelle $X\to X^\textrm{al}$.

\subsubsection{Topologie et morphismes}  Un morphisme d'espaces $k$-analytiques $f:Y\to X$ est topologiquement propre s'il est universellement fermé dans la catégorie des espaces topologiques, et cette propriété est équivalente à ce que l'image inverse de chaque domaine affinoïde de $X$ soit quasi-compacte dans $Y$. Un morphisme $f:Y\to X$ d'espaces $k$-analytiques sera dit proprement surjectif s'il existe un $G$-recouvrement de $X$ par des domaines analytiques quasi-compacts qui soient chacun d'entres eux l'image d'un domaine analytique quasi-compact de $Y$.

\subsubsection{Polyrayons et section de Shilov} Un polyrayon est une famille finie de nombres réels strictements positifs. Si l'on se donne un polyrayon $r=(r_1,..,r_n)$ et une famille d'indéterminées $\underline{T}=(T_1,..,T_n)$, alors on note $k_r$ la $k$-algèbre affinoïde des séries formelles $\sum_{I\in \mathbb{Z}^n}a_I \underline{T}^I$ à coefficient dans $k$ vérifiant $\vert a_I \vert r^I\to 0$ quand $\vert I \vert \to +\infty$. La flèche $\sum_{i\in \mathbb{Z}^n} a_I \underline{T}^I\mapsto \textrm{max}_{i\in \mathbb{Z}^n}\vert a_I \vert r^I\in \mathbb{R}$ est une norme multiplicative sur $k_r$. Si la famille $r$ est $k$-libre c'est à dire libre lorsqu'on la voit comme une famille d'éléments du $\mathbb{Q}$-espace vectoriel $\mathbb{Q}\otimes_\mathbb{Z} \mathbb{R}_+^*/\vert k^* \vert$, alors $k_r$ est un corps. 

Si $X$ est un espace $k$-analytique, on note $X_r:=X\times_k \mathcal{M}(k_r)$ et si $\mathcal{A}$ est une algèbre $k$-affinoïde, on note $\mathcal{A}_r:=\mathcal{A}\hat{\otimes}_k k_r$. Pour tout $x\in X$, si $r$ est $k$-libre, la fibre de $X_r\to X$ en $x$ s'identifie à l'espace $\mathcal{H}(x)$-affinoïde $\mathcal{M}(\mathcal{H}(x)\hat{\otimes}_k k_r)$ où l'on a noté $\mathcal{H}(x)$ le corps résiduel complété de $x$ (on renvoie à \autocite{berkovich2012spectral} pour la définition de $\mathcal{H}(x))$. On notera $\sigma (x)$ l'unique semi-norme sur $\mathcal{H}(x)\hat{\otimes}_k k_r$ qui envoie un élément $\sum_I a_I \underline{T}^I\in \mathcal{H}(x)\hat{\otimes}_k k_r$ sur $\max_I \vert a_i \vert (x) r^I$. L'application $x\mapsto \sigma (x)$ est alors une section continue de $X_r\to X$, que l'on appelera la section de Shilov (cf section 3.2.2 de \autocite{berkovich2012spectral}). Pour $X$ un espace $k$-analytique et $r$ un polyrayon $k$-libre, une extension de polyrayon sera un morphisme d'espaces $k$-analytiques de la forme $X_r\to X$.

\subsubsection{Finitude et complété} Considérons une $k$-algèbre de Banach noethérienne $\mathcal{A}$, $M$ un $\mathcal{A}$-module fini, et une $\mathcal{A}$-algèbre de Banach noethérienne $\mathcal{B}$. Alors l'application canonique $j:M\otimes_\mathcal{A}\mathcal{B}\to M\hat{\otimes}_\mathcal{A}\mathcal{B}$ est bijective par \autocite{bosch1984non}, corollaire 3.7.3.6.

\subsubsection{Dimension} On dispose en géométrie de Berkovich d'une théorie de la dimension, qui est exposée dans \autocite{berkovich2012spectral} ou bien \autocite{ducros_2007}. Nous la supposerons connue. Notre terminologie suivra celle de Ducros dans \autocite{ducros2017families}, et nous dirons qu'un morphisme d'espaces $k$-analytiques $f:Y\to X$ est quasi-fini en $y\in Y$ s'il est de dimension relative nulle en $y$, c'est à dire si la dimension $\textrm{dim}_y f^{-1}(f(y))$ est nulle. Le morphisme $f$ est alors fini en $y$ si et seulement si il est quasi-fini en $y$ et sans bord en $y$.

On dit que le morphisme $f$ est quasi-fini s'il est quasi-fini en tout point et topologiquement propre.

\subsubsection{GAGA} Considérons une algèbre $k$-affinoïde $\mathcal{A}$. Alors on dispose d'un foncteur d'analytification relative, qui à un $\mathcal{A}$-schéma localement de type fini $\mathcal{X}$ associe un bon espace $\mathcal{A}$-analytique $X=\mathcal{X}^{\mathrm{an}}$ ainsi qu'un morphisme d'espaces localement annelés $X\to \mathcal{X}$ qui est un objet final dans la catégorie des bons espaces $\mathcal{A}$-analytiques munis d'un $\Spec \mathcal{A}$-morphisme d'espaces localement annelés vers $\mathcal{X}$. Pour la construction de $\mathcal{X}^\textrm{an}$ et ses principales propriétés, on renvoie à \autocite{berkovich1993etale} et à l'annexe A de \autocite{Poineau_2010}. La flèche canonique $\mathcal{X}^\textrm{an}\to \mathcal{X}$ est plate et surjective en tant que morphisme d'espaces localement annelés. Cette construction est fonctorielle et commute au changement de base affinoïdes : si $Z\to \mathcal{M}(\mathcal{A})$ est un morphisme d'espaces affinoïdes, alors on a l'égalité $\mathcal{X}^{\textrm{an}}\times_\mathcal{A} Z=(\mathcal{X}\times_{\mathcal{A}}\mathcal{O} (Z))^{\textrm{an}}$.
Si $\mathcal{F}$ est un faisceau cohérent sur $\mathcal{X}$, alors son tiré en arrière sur $\mathcal{X}^{\mathrm{an}}$ est un faisceau cohérent au dessus de $X$, que l'on note $\mathcal{F}^{\mathrm{an}}$. 

Si $\mathcal{X}$ est propre, comme dans le cas complexe, on dispose de théorèmes de type GAGA : le foncteur $\mathcal{F}\mapsto \mathcal{F}^{\textrm{an}}$ induit une équivalence de catégories entre la catégorie des faisceaux cohérents au dessus de $\mathcal{X}$ et la catégorie des faisceaux cohérents au dessus de $\mathcal{X}^{\textrm{an}}$, et cette équivalence de catégories respecte la cohomologie cohérente sur les deux espaces. En particulier, $\mathcal{Z}\mapsto \mathcal{Z}^{\textrm{an}}$ induit une bijection entre l'ensemble des parties Zariski-fermées de $\mathcal{X}$ et l'ensemble des parties Zariski-fermées de $\mathcal{X}^{\textrm{an}}$. En corollaire, on voit que si l'on restreint le foncteur d'analytification aux schémas propres, on obtient un foncteur pleinement fidèle de la catégorie des $\mathcal{A}$-schémas localement de type finis et propres vers la catégorie des espaces $\mathcal{A}$-analytiques.

\subsubsection{Définition de la platitude} Nous rappelons ici la définition de la platitude de \autocite{ducros2017families}. Con\-si\-dé\-rons $f:Y\to X$ un morphisme d'espaces $k$-analytiques, $y$ un point de $Y$, $x$ son image dans $X$ et $\mathcal{F}$ un faisceau cohérent sur $Y$. 

Si $X$ et $Y$ sont des bons espaces, le faiseau $\mathcal{F}$ est dit $\textit{naïvement plat}$ en $y$ si $\mathcal{F}_y$ est un $\mathcal{O}_{X,x}$-module plat. Le problème de cette notion est qu'elle n'est pas stable par  changement de base (cf section 4.4 de \autocite{ducros2017families}), Ducros définit donc la platitude d'abord sur les bons espaces en $\textit{forçant}$ la stabilité par changement de base et extension des scalaires, et montre ensuite que la notion ainsi définie est la bonne notion de platitude.

Ainsi, si les espaces $X$ et $Y$ sont bons, alors $\mathcal{F}$ est dit $X$-plat en $y$ si pour tout bon espace analytique $X'$ définit sur une extension complète de $k$, tout $k$-morphisme $X'\to X$, et tout antécédent $y'$ de $y$ sur $Y':=Y\times_X X'$, le faisceau cohérent $\mathcal{F}_{Y'}$ tiré en arrière de $\mathcal{F}$ sur $Y'$ est naïvement $X'$-plat en $y'$. 

Dans le cas général où les espaces en jeu ne sont plus supposés bons, alors le faisceau $\mathcal{F}$ est dit $X$-plat en $y$ si pour tout bon domaine analytique $U$ de $X$ contenant $x$ et tout bon domaine analytique $V$ de $Y\times_X U$ contenant $y$, le faisceau cohérent $\mathcal{F}_V$ tiré en arrière de $\mathcal{F}$ à $V$ est {$U\mathrm{-plat}$~;~}et il suffit en fait de le vérifier pour un tel couple $(U,V)$ donné en particulier. Si $\mathcal{F}=\mathcal{O}_Y$, on dit que c'est le morphisme $f$ qui est plat en $y$. Le morphisme $f$ est dit plat s'il est plat en tout point de $Y$, et fidèlement plat s'il est plat et surjectif.
%intro polyrayon libre  G recouvremtn quasi étale plat compatibilité schématique +trois théorème utilisation intensive+dire ce qu'on appelle un G recouvrement par abus den otation+dire ce que c'est constructible=sens de EGA=loct constru pour stackproject+constructible image inverse, compatibilité ouvert algébrique
%intro polyrayon libre  G recouvremtn quasi étale plat compatibilité schématique 
\subsubsection{Platitude analytique et schématique} Considérons $f:\mathcal{M(B)}\to\mathcal{M(A)}$ un morphisme fidèlement plat d'espaces $k$-affinoïdes. Alors le morphisme induit $\Spec \mathcal{B}\to \Spec \mathcal{A}$ est fidèlement plat (c'est le lemme 4.2.1 de \autocite{ducros2017families} voir la section 4.2 de \autocite{ducros2017families} pour plus de résultats de ce type).
\subsubsection{Platitude et finitude} Condidérons $f:Y\to X$ un morphisme fini entre bons espaces $k$-analytiques et $\mathcal{F}$ un faisceau cohérent sur $Y$. Alors par la proposition 4.3.1 de \autocite{ducros2017families}, $\mathcal{F}$ est plat en un point si et seulement s'il est naïvement plat, et si $\mathcal{F}$ est plat en tout point alors pour tout $x\in X$, il existe un voisinage affinoïde $V$ de $x$ dans $X$ tel que $f_*(\mathcal{F}_{f^{-1}(V)})$ soit un $\mathcal{O}_V$-module libre.

\subsubsection{Image d'un espace compact par un morphisme plat} Si $X$ et $Y$ sont des espaces analytiques compacts et si $\mathcal{F}$ est $X$-plat, l'image $f(\textrm{Supp}(\mathcal{F}))$ est un domaine analytique compact par le théorème 9.2.1 de \autocite{ducros2017families}. Le cas strict avec $\mathcal{F}=\mathcal{O}_Y$ était déjà connu, et dû à Raynaud (cf corollaire 5.11 de \autocite{bosch1993formal}). 
\subsubsection{Le théorème de multisection} On cite ici une conséquence importante de l'énoncé de 9.1.3 de \autocite{ducros2017families} que nous utiliserons de manière cruciale à de nombreuse reprise dans le texte. On suppose uniquement dans ce paragraphe que le corps $k$ est non trivialement valué $\emph{i.e}$ que $\vert k^* \vert \neq \{1\}$. Considérons $Y$ un espace strictement $k$-analytique quasi-compact et $X$ un espace $k$-analytique séparé. Soit $\varphi:Y\to X$ un morphisme plat d'espaces $k$-analytiques. Alors il existe un espace strictement $k$-affinoïde $X'$, un $X$-morphisme d'espaces analytiques $\sigma:X'\to Y$  et un morphisme quasi-fini et plat d'espaces analytiques $\psi:X'\to X$ tel que l'on ait l'égalité ensembliste $\varphi(X)=\psi(X')$. De plus, si $\varphi$ est quasi-lisse, alors $\psi$ peut-être choisie quasi-étale.

\subsubsection{Constructibilité} Considérons un espace $k$-analytique $X$. On dit qu'une partie $E\subset X$ est constructible si elle s'écrit comme une union finie $E=\cup_i (U_i\cap F_i)$ avec $U_i$ (resp $F_i$) une partie Zariski-ouverte (resp. Zariski fermée) de $X$ pour tout $i\in I$.

Une partie $E$ est localement constructible (resp. $G$-localement constructible) s'il existe un recouvrement par des ouverts (resp. un $G$-recouvrement) $(X_i)_{i\in I}$ de $X$ tel que $E\cap X_i$ soit une partie constructible de $X_i$ pour tout $i\in I$. Ces deux notions coïncident en fait par la proposition 10.1.12 de \autocite{ducros2017families}, et dans le cas où l'espace analytique $X$ est de dimension finie, une partie est constructible si et seulement si elle est $G$-localement constructible si et seulement si elle est localement constructible.

Si l'on se donne  une algèbre $k$-affinoïde $\mathcal{A}$, un $\mathcal{A}$-schéma localement de type fini $\mathcal{X}$, et $E$ une partie constructible (resp. localement constructible) de $\mathcal{X}$, alors la préimage $E^\textrm{an}$ de $E$ sur $X^\textrm{an}$ est une partie constructible (resp. localement constructible) de $X^\textrm{an}$. Si $\mathcal{X}$ est propre au dessus de $\mathcal{A}$, il résulte de GAGA que $E\mapsto E^\textrm{an}$ est une bijection entre les parties constructibles de $X$ et les parties constructibles de $X^\textrm{an}$.

Par le corollaire 10.1.11 de \autocite{ducros2017families}, si une partie $E$ est $G$-localement constructible, alors son adhérence de Zariski coïncide avec son adhérence pour la topologie usuelle, donc une partie $G$-localement constructible est fermée (resp. ouverte) si et seulement si elle est fermée de Zariski (resp. ouverte de Zariski).

\section{Un premier théorème champêtre}

On utilisera librement dans tout cet article les notions de l'exposé \autocite{vistoli2005notes}. 
On considère $\mathsf{C}
$ la catégorie des espaces $k$-affinoïdes, et ${\Psi}$ un pseudo-foncteur, ou lax 2-foncteur sur $\mathsf{C}$. Autrement dit, on dispose pour tout objet $U$ de $\mathsf{C}$ d'une catégorie $\Psi U$, et pour tout morphisme $f:U\to V$ d'un foncteur tiré en arrière $f^*:\Psi V\to \Psi U$ satisfaisant certaines compatibilités à la composition que nous n'énoncerons pas ici.

Considérons $p:S'\to S$ un morphisme d'espaces $k$-affinoïdes. On utilisera dans la suite de ce texte les notations suivantes :  
$$S''=S'\times_S S',~~~~ S'''=S'\times_S S'\times_S S'$$et on notera : $$p_i:S''\to S',~i\in \{1,2\}~~~~~~p_{ij}:S'''\to S''~i,j\in\{1,2,3\},i<j$$ les projections sur le $i$-ème facteur et le facteur $(i,j)$ respectivement.
On a aussi des projections $q_i:S'''\to S'$ avec $q_i=p_1\circ p_{ij}$, et $q_j=p_2\circ p_{ij}$.

\setcounter{subsection}{1}

\begin{defi} Un objet avec données de descente relativement à $p$ est un élément $\xi'$ de $\Psi S'$ muni d'un isomorphisme $\varphi:p_1^*\xi'\to p_2^*\xi' $ vérifiant la condition de cocycle $p_{13}^*\varphi=p_{12}^*\varphi \circ p_{23}^*\varphi$. Un morphisme de données de descente de $\xi'$ vers $\eta'$ est un morphisme dans la catégorie $\Psi S'$ qui commute aux projections.
\end{defi}

On dispose donc de la catégories des objets de $\Psi S'$ munis de données de descente relativement à $p$, que l'on note $\Psi(S'\to S)$, et on a aussi un foncteur de $\Psi S$ vers $\Psi (S'\to S)$ qui envoie un objet $\xi\in \Psi S$ sur $p^*\xi$. On dit que le morphisme $p$ est un morphisme de descente relativement à $\Psi$ si ce foncteur est pleinement fidèle, et que c'est un morphisme de descente effective relativement à $\Psi$ si ce foncteur est une équivalence de catégories. Lorsque la situation est sans ambiguïté, on s'autorisera à parler de morphisme de descente (resp. de descente effective) sans faire référence au pseudo-foncteur que l'on considère. Un morphisme sera de descente universellement s'il est de descente et que n'importe quel changement de base de ce morphisme est un morphisme de descente.

\begin{exem}
On dispose d'un pseudo-foncteur qui à un espace $k$-affinoide $S\in \mathsf{C}$ associe la catégorie des modules cohérents au dessus de $S$. Le tiré en arrière est alors simplement le tiré en arrière des modules usuel. Ducros a démontré dans \autocite{ducros2020devisser} que les morphismes plats, surjectifs et topologiquement propres sont des morphismes de descentes effectifs pour ce pseudo-foncteur. Le résultat était déjà connu en géométrie rigide, et dû à Gabber sous des hypothèses restrictives, puis à Bosch et Görtz dans le cas général dans le texte \autocite{Coherentmodulesandtheirdescentonrelativerigidspaces}.
\end{exem}

\begin{exem} Considérons $\mathsf{D}$ une catégorie qui possède les produits fibrés. Dans une telle catégorie $\mathsf{D}$, un morphisme $p:S'\to S$ est un épimorphisme effectif si pour tout objet $X\in \mathsf{D}$, le diagramme d'ensemble suivant est exact : $X(S)\to X(S')\rightrightarrows X(S'\times_S S')$. Dans une telle catégorie, on dispose aussi du pseudo-foncteur $\Psi$ au dessus de $\mathsf{D}$ qui à un objet $S\in \mathsf{D}$ associe la catégorie des $S$-objets de $\mathsf{D}$. Le tiré en arrière est alors simplement donné par le produit fibré de la catégorie $\mathsf{D}$.

On rappelle qu'un morphisme $p:S'\to S$ dans $\mathsf{D}$ est un morphisme de descente pour le pseudo-foncteur $\Psi$ si et seulement si c'est un épimorphisme effectif universel dans $\mathsf{D}$ c'est à dire un épimorphisme effectif après tout changement de base dans $\mathsf{D}$ par l'exposé 4 de \autocite{grothendieck1963schemas}.
\end{exem}

\subsection{La descente effective}

Le résultat phare de cette section est le résultat suivant. Il donne une condition pour que les morphismes plats et surjectifs soient de descente effectif pour un pseudo-foncteur donné au dessus de la catégorie des espaces $k$-affinoïdes. Il permet donc de décider si un pseudo-foncteur en particulier est un champ pour la topologie de Grothendieck dont les flèches couvrantes sont les flèches plates et surjectives.

\begin{theo} \label{fibfi} Considérons $\Psi$ un pseudo-foncteur au dessus de $\mathsf{C}$, avec $\mathsf{C}$ la catégorie des espaces analytiques $k$-affinoïdes. On suppose que :
\begin{enumerate}
\item Pour toute algèbre $k$-affinoïde $\mathcal{A}$, et tout polyrayon $k$-libre $r\in (\mathbb{R}_+ ^*)^n$, le morphisme $\mathcal{M}(\mathcal{A}_r)\to \mathcal{M(A)}$ est un morphisme de descente effective relativement à $\Psi$.
\item Tout morphisme plat, fini et surjectif $\mathcal{M(B)}\to \mathcal{M(A)}$ est un morphisme de descente effective relativement à $\Psi$.
\item Tout $G$-recouvrement fini $\amalg_{i=1}^n S_i\to S$ de $S$ par des domaines affinoïdes en nombre fini est un morphisme de descente effective relativement à $\Psi$.

\end{enumerate}

Alors tout morphisme fidèlement plat entre espaces affinoïdes est un morphisme de descente effective relativement à $\Psi$.
\end{theo}

La technique de démonstration du théorème précédent est d'utiliser les deux premiers points du lemme suivant à répétition en suivant la stratégie de démonstration de la preuve du théorème 3.3 de \autocite{ducros2020devisser}.

\begin{lemm}\label{my} Considérons $\Psi$ un pseudo-foncteur au dessus d'une catégorie possédant des produits fibrés ($\emph{e.g}$ la catégorie des espaces affinoïdes). Considérons $R\stackrel{v}{\to} S\stackrel{u}{\to} T$ des morphismes d'une catégorie $\mathsf{C}$. Alors : 
\begin{enumerate}
\item Supposons que $u$ et $v$ sont des morphismes de descente effective, et que les morphismes canoniques $m:R\times_T R\to S\times_T S$ et $m':R\times_T R\times_T R\to S\times_T S \times_T S$ sont des morphismes de descente. Alors le morphisme $u\circ v$ est de descente effective.
\item Supposons que $v$ est de descente effective et que $u\circ v$ est de descente effective. Alors $u$ est de descente effective.
\end{enumerate}
\end{lemm}
\begin{proof}[Preuve]
On peut démontrer ces points de la même manière que les proposition 10.10 et 10.11 de \autocite{girauddescente1964}. On va par exemple démontrer le premier point. Supposons donc que $u$, $v$ sont des morphismes de descente effective et que $m$ et $m'$ sont des morphismes de descente. Notons $p_i:S\times_T S\to S$, $q_i:R\times_T R\to R$, et $r_i:R\times_S R\to R$ les projections canoniques pour $i\in\{1,2\}$, et notons $l:R\times_S R\to R\times_T R$ le morphisme canonique. La situation est résumée par le diagramme suivant : 
$$
\xymatrix{ 
    R\times_S R \ar@<-.5ex>[rd] \ar@<.5ex>[rd] \ar[r]^l & R\times_T R \ar@<-.5ex>[d] \ar@<.5ex>[d] \ar[r]^m & S\times_T S  \ar@<-.5ex>[d] \ar@<.5ex>[d] &     \\
    & R \ar[r]^v & S   \ar[r]^u & T 
  } $$

On veut montrer que $u\circ v$ est un morphisme de descente effective. On se donne donc une donnée de descente relativement au morphisme $u\circ v:R\to T$, c'est à dire un élément $\xi''\in \Psi R$ muni d'un isomorphisme $\varphi':q_1^*\xi''\to q_2^*\xi''$ vérifiant la condition de cocycle usuelle. On tire la donnée de descente $\varphi'$ en arrière par le morphisme $l$, ce qui nous fournit une donnée de descente $l^*\varphi':r_1^*\xi''\to r_2^*\xi''$ sur $\xi''$ relativement au morphisme $v$. Puisque $v$ est de descente effective, on en déduit qu'il existe un objet $\xi'$ dans $\Psi S$ muni d'un isomorphisme de données de descente $\lambda':v^*\xi'\to \xi''$ relativement à $v$, c'est à dire que $\lambda'$ vérifie la relation $l^*\varphi'\circ r_1^* \lambda'=r_2^* \lambda'$.

Maintenant, va équiper $\xi'$ d'une donnée de descente relativement à $u$, c'est à dire d'un isomorphisme $\varphi:p_1^*\xi'\to p_2^*\xi'$ vérifiant la condition de cocyle, et qui vérifie de plus la relation $q_2^*\lambda'\circ m^*\varphi=\varphi'\circ q_1^*\lambda'$. Pour l'existence de $\varphi$, on utilise l'hypothèse que $m$ est de descente : il suffit de vérifier que si l'on note $m_1$ et $m_2$ les deux projections canoniques de $(R\times_T R )\times_{S\times_T S} (R\times_T R) \to R\times_T R$, alors le morphisme $z:=q_2^*\lambda'^{-1}\circ \varphi'\circ q_1^*\lambda'$ vérifie la relation $m_1^*z=m_2^*z$. On vérifie aussi, en utilisant l'hypothèse sur $m'$ que $\varphi$ vérifie bien la relation de cocycle. Maintenant, puisque $u$ est de descente effective, on en déduit l'existence de $\xi\in \Psi T$ muni d'un isomorphisme $\lambda:u^*\xi\to \xi'$ de données de descente relativement à $u$, c'est à dire que $\lambda$ vérifie la relation $p_2^*\lambda=\varphi\circ p_1^*\lambda$. Si l'on tire cette relation en arrière par le morphisme $m$, on obtient la relation $q_2^*v^*\lambda=m^*\varphi \circ q_1^* v^*\lambda$, et par la relation vérifiée par $\varphi$, on en déduit $q_2^* v^* \lambda= q_2^*\lambda'^{-1}\circ \varphi'\circ q_1^*\lambda' \circ q_1^*v^*\lambda $, c'est à dire en réarrangeant les termes que $\varphi'\circ q_1^*(\lambda'\circ v^*\lambda)=q_2^*(\lambda'\circ v^*\lambda)$, donc $\lambda'\circ v^*\lambda$ définit bien un isomorphisme de données de descentes de $(u\circ v)^*\xi$ muni de ses données de descentes canoniques vers $\xi''$ muni de ses données de descente induites par $\varphi'$, et le morphisme $u\circ v$ est bien un morphisme de descente effective. \qedhere

\end{proof}
Dans la suite de cette section, on va démontrer le théorème. On se donne donc la catégorie $\mathsf{C}$ des espaces $k$-affinoïdes, un pseudo-foncteur $\Psi$ au dessus de $\mathsf{C}$ et un morphisme fidèlement plat $p:S'\to S$ dans la catégorie $\mathsf{C}$, dont on va montrer qu'il est de descente effective.

\subsubsection{Réduction au cas strict non trivialement valué}

Considérons $r$ un polyrayon $k$-libre tel que $S'_r$ et $S_r$ soient strictement $k_r$-affinoïdes et que la valeur absolue sur $k_r$ soit non triviale. 
On a le diagramme suivant : 

\begin{equation}\label{strict} \xymatrix{
    S'_r \ar[r] \ar[d]  &  S_r \ar[d] \\
    S' \ar[r] & S
  }
  \end{equation}

Quitte à remplacer $\Psi$ par le pseudo-foncteur induit sur les espaces $k_r$-affinoïdes, et à remplacer $k$ par $k_r$, on peut supposer que $S$ et $S'$ sont strictement affinoïdes, et que le corps $k$ est non trivialement valué. En effet, si l'on démontre le théorème dans ce cas là, tout morphisme plat surjectif entre espaces strictement $k_r$-affinoïdes est de descente effectif, et c'est en particulier le cas du morphisme $S'_r\to S_r$ mais aussi des morphismes ${S'_r}\times _S {S'_r}\to {S_r}\times_S {S_r}$ et ${S'_r}\times _S {S'_r}\times_S S'_r\to S_r\times_S {S_r}\times_S {S_r}$, et puisque $S_r\to S$ est effectif par hypothèse, par le lemme $\ref{my}$, on en déduit que $S'_r\to S$ est de descente effective. 

Maintenant, puisque $S'_r\to S'$ est de descente effectif par hypothèse, il suffit d'appliquer la deuxième partie du lemme $\ref{my}$ pour conclure que le morphisme $S'\to S$ est de descente effectif. Cela nous permet donc de réduire le problème à un morphisme $S'\to S$ plat surjectif entre espaces strictement $k$-affinoïdes, avec $k$ un corps non trivialement valué.

\subsubsection{Réduction au cas quasi-fini}
Considérons donc un morphisme fidèlement plat $S'\to S$ entre espaces strictement $k$-affinoïdes. Par le théorème sur l'existence de multisection plate 9.1.3 de \autocite{ducros2017families}, il existe un espace strictement $k$-affinoïde $X$, un morphisme quasi-fini (c'est à dire topologiquement propre et de dimension relative nulle), fidèlement plat $X\to S$ et un $S$-morphisme $X\to S'$. Maintenant, le morphisme $X'=S'\times _S X\to X$ ainsi que tout changement de base de celui-ci admet alors une section, c'est donc un morphisme de descente effective , donc par le lemme $\ref{my}$, on peut se ramener à démontrer le théorème pour un morphisme $S'\to S$ quasi-fini, plat et surjectif  puisque ces propriétés sont stables par changement de base (pour quasi-fini, cela résulte de l'invariance par changement de base de la dimension, qui est établie dans \autocite{ducros_2007}).  

\subsubsection{Réduction au cas quasi-étale} \label{redquaset}

Considérons un morphisme plat, surjectif, quasi-fini $S'\to S$ entre espaces affinoïdes. Alors, comme le morphisme est de dimension relative nulle et plat, par la remarque 8.4.3 de \autocite{ducros2017families}, $S'$ est de Cohen-Macaulay au dessus de $S$. Maintenant, on applique le théorème 8.4.6 de \autocite{ducros2017families} : pour tout $s'\in S'$, il existe un domaine affinoïde $V_i$ de $S'$ ($i$ dépend de $s'$) qui est un voisinage de $s'$, un espace affinoïde $W_i$, un $S$-espace $T_i$, étale au dessus de $S$, un morphismes $V_i\to W_i$ plat, fini, un morphisme $W_i\to T_i$ qui identie $W_i$ à un domaine affinoïde de $T_i$ et tels que la composée $V_i\to W_i\to T_i\to S$ soit aussi la restriction de $f$ à $V_i$. De plus, par la proposition 4.3.1 de \autocite{ducros2017families}, il existe un voisinage affinoïde $V_i'$ de $s'$ dans $V_i$ et un voisinage affinoïde $W_i'$ de l'image de $s'$ dans $W_i$ tels que la flèche $V_i'\to W_i$ se factorise en $\pi:V_i'\to W_i'$ finie et tels que $\pi_*(O_{V_i '})$ est un $\mathcal{O}_{W_i '}$ module libre. La flèche $W_i'\to S$ reste quasi-étale, et si l'on note $\mathcal{F}$ (resp $\mathcal{E}$) l'algèbre $k$-affinoïde associée à ${V}_i'$ (resp. ${W}_i'$), la flèche $\mathcal{E}\to \mathcal{F}$ est fidèlement plate, ($\mathcal{E}$ est un $\mathcal{F}$-module libre de rang $r>0$) donc elle est aussi injective, et la flèche $V_i'\to W_i'$ est bien surjective par le corollaire 2.1.16 de \autocite{berkovich2012spectral}.

Ainsi, on dispose d'une factorisation locale de $f$ en $V_i'\to W_i'$ finie, plate et surjective suivie de $W_i'\to S'$ quasi-étale. On obtient donc le diagramme commutatif suivant : 

$$\xymatrix{
    \amalg V_i' \ar[r] \ar[d]  & \amalg W_i' \ar[d] \\
    S' \ar[r] & S
  }$$

La flèche $\amalg V_i'\to S'$ est un $G$-recouvrement par des affinoïdes qui peut être choisi fini, donc c'est un morphisme de descente effective . Par le lemme $\ref{my}$, pour démontrer que $S'\to S$ est de descente effective, il suffit de le faire pour le morphisme $\amalg V_i'\to S$, mais celui-ci se décompose en $\amalg V_i'\to \amalg W_i'$ qui est fini, plat et surjectif donc effectif et de descente universellement, suivi de $\amalg W_i'\to S$ qui est quasi-étale fidèlement plat. Par le lemme $\ref{my}$, il suffit donc de démontrer le théorème pour un morphisme $p:S'\to S$ quasi-étale fidèlement plat.

\subsubsection{La réduction au cas d'une flèche plus simple} \label{redplusimple}

Considérons donc un morphisme $p:Y\to X$ quasi-étale surjectif entre espaces affinoïdes. On va montrer que le morphisme $p$ est de descente effective. Soit $x\in X$, notons $y_1,..,y_r\in Y$ les antécédents de $x$ qui sont en nombre fini. Par définition d'un morphisme quasi-étale, pour tout $1\leqslant i\leqslant r$, il existe un voisinage affinoïde $Y_i'$ de $y_i$ qui est $X$-isomorphe à un domaine affinoïde d'un $X$-espace étale $Z_i'$. Maintenant, par le caractère étale du morphisme $Z'_i\to X$, il existe $Z_i''$ un ouvert de $Z_i'$ contenant l'image de $y$ et un voisinage $X_i''$ de $x$ dans $X$ tel que la restriction de la flèche structurale de $Z_i'$ à $Z_i''$ induise une flèche $Z_i''\to X_i''$ finie étale. Maintenant, par la proposition 4.3.1 de \autocite{ducros2017families}, il existe un voisinage affinoïde $Z_i$ de l'image de $y$ dans $Z_i''$, et un voisinage affinoïde $X_i$ de $x$ dans $X_i''$ tels que $Z_i''\to X_i''$ se factorise en $\pi:Z_i\to X_i$ fini étale et vérifiant que $\pi_*(\mathcal{O}_{Z_i})$ est un $\mathcal O _{X_i}$-module libre. Alors si l'on note $\mathcal{E}$ (resp $\mathcal{F}$) l'algèbre des fonctions sur $X_i$ (resp $Z_i$), alors la flèche $\mathcal{E}\to \mathcal{F}$ est fidèlement plate, donc injective, et alors $Z_i\to X_i$ est surjective par 2.1.16 de \autocite{berkovich2012spectral}. 

Maintenant, $Y_i:=Z_i\times_{Z_i'}Y_i'\to Y_i'$ est une inclusion de domaine analytique au voisinage de $y$, tout comme $Y_i\to Z_i$. Ainsi, $Y_i\to Y$ est une inclusion de domaine analytique, tout comme $Y_i\to Z_i$, et $Z_i\to X$ se décompose en $Z_i\to X_i$ étale, finie et surjective, et $X_i\to X$ qui est une inclusion de voisinage affinoïde de $x$.

La situation est résumée par le diagramme suivant : 

$$\label{diag2} 
\xymatrix{ 
    Y_i\ar[d] \ar[rr] &  &  Z_i \ar[d] &  \\
    Y \ar[r] & X  & \ar[l] X_i & 
  } $$

La partie $\cup_{i\in I} Y_i$ est un voisinage de la fibre $p^{-1}(x)$ donc puisque $p$ est fermée, il existe $V_x$ un voisinage affinoïde de $x$ dans $X$ contenu dans chaque $X_i$ tel que $p^{-1}(V)\subset \cup_{i\in I} Y_i$ (en effet, si $W$ est un voisinage ouvert de $p^{-1}(x)$ inclus dans $\cup_{i\in I}Y_i$, on prend un voisinage affinoïde $V$ de $x$ inclus dans l'ouvert $V':=X\setminus{p(Y\setminus{W})}$). Maintenant, par compacité de $X$, il existe un nombre fini de $V_x$ qui recouvrent $X$, et en appliquant le lemme $\ref{my}$ aux composée $\amalg p^{-1}(V_x)\to \amalg V_x\to X$ et $\amalg p^{-1}(V_x)\to Y\to X$ on voit qu'il suffit de démontrer que la flèche $p^{-1}(V_x)\to V_x$ reste de descente effective universellement pour chaque $x\in X$ pour montrer que $p:Y\to X$ est de descente effective.

\subsubsection{Fin de la démonstration} \label{redgtop}
On fixe donc $x\in X$ et on note $V=V_x$ le voisinage affinoïde associé du paragraphe précédent. Pour montrer que $p^{-1}(V)\to V$ est de descente effective, on va appliquer le lemme $\ref{my}$ à certaines factorisations qui sont stables par changement de base, donc cela montrera aussi que $p^{-1}(V)\to V$ est de descente effective universellement.
Maintenant, par changement de base, pour chaque $1\leqslant i \leqslant r$, le diagramme $\ref{diag2}$ fournit le diagramme suivant : 

\[ \xymatrix{ 
    Y_i\times_X V \ar[d] \ar[r	] &    Z_i\times_X V \ar[d] &  \\
    Y\times_X V \ar[r] & V  &   
  } \]
et puisque chaque $Z_i\times_X V\to V$ est fini étale, il existe un revêtement fini étale galoisien $T\to V$ de l'espace $V$ qui domine tous ces revêtements finis étales. On a le diagramme suivant : 
\[ \xymatrix{ 
    \amalg_{i\in I} Y_i\times_X V \ar[d] \ar[r	] &    \amalg_{i\in I} Z_i\times_X V \ar[d] &  \\
    Y\times_X V \ar[r] & V  &   
  } \]

Puisque $T\to V$ est fini étale surjectif, il est de descente effectif universellement par hypothèse, et par le lemme $\ref{my}$, il suffit de démontrer que la flèche $Y\times_X V\times_V T\to T$ est universellement de descente effective pour montrer que la flèche $Y\times_X V\to V$ est de descente effective. On effectue donc le changement de base $T\to V$ au diagramme précédent pour obtenir : 

\[ \xymatrix{ 
    \amalg_{i\in I} Y_i\times_X V\times_V T \ar[d] \ar[r	] &    \amalg_{i\in I} Z_i\times_X V\times_V T \ar[d] &  \\
    Y\times_X V\times_V T \ar[r] & T  &   
  } \]
  
Et la flèche $\amalg_{i\in I} Y_i\times_X V\times_V T\to  Y\times_X V\times_V T$ est un $G$-recouvrement surjectif (car $\amalg Y_i\times_X V\to Y\times_X V$ l'est par définition de $V$), tout comme $\amalg_{i\in I} Y_i\times_X V\times_V T\to T$ (puisque $Y_i\to Z_i$ est une inclusion de domaine analytique, et que $Z_i\times_X V \times T\to T$ est de la forme $\amalg_{i\in H_i}T\to T$ pour un certain ensemble fini $H_i$ par définition d'un revêtement galoisien). Par une dernière application du lemme $\ref{my}$, on voit que $Y\times_X T\to T$ est de descente effective, ce qui démontre que $Y\times_X V \to V$ l'est, et comme cela vaut sur un $G$-recouvrement fini par des domaines affinoïdes $V$ de $X$, cela le démontre pour $Y\to X$ quasi-étale surjective, et le théorème $\ref{fibfi}$ est démontré. 

\subsection{La pleine fidélité}

On a un analogue du lemme $\ref{my}$ pour les morphismes de descente, qui est démontré dans \autocite{girauddescente1964}, à la proposition 10.10 et 10.11.

\begin{lemm}\label{mybis} Considérons $\Psi$ un pseudo-foncteur au dessus au dessus d'une catégorie possédant des produits fibrés ($\emph{e.g}$ la catégorie des espaces affinoïdes). Considérons $R\stackrel{v}{\to} S\stackrel{u}{\to} T$ des morphismes dans $\mathsf{C}$. Alors : 
\begin{enumerate}
\item Supposons que $u$ et $v$ sont des morphismes de descente, et que le morphisme canonique $k:R\times_T R \to S\times_T S$ induit un foncteur fidèle $\Psi(S\times_T S) \to \Psi ( R\times_T R)$. Alors le morphisme $u\circ v$ est un morphisme de descente.
\item Supposons que $u\circ v$ est un morphisme de descente et que le morphisme $v$ induit un foncteur fidèle $\Psi S\to \Psi R$. Alors $u$ est un morphisme de descente.
\end{enumerate}
\end{lemm}

En redéroulant la démonstration du théorème $\ref{fibfi}$ étape par étape, on obtient le théorème suivant : 

\begin{theo} \label{fibfibis} Considérons $\Psi$ un pseudo-foncteur au dessus de $\mathsf{C}$, avec $\mathsf{C}$ la catégorie des espaces analytiques $k$-affinoïdes. On suppose que :
\begin{enumerate}
\item Pour toute algèbre $k$-affinoïde $\mathcal{A}$, et tout polyrayon $k$-libre $r\in (\mathbb{R}_+ ^*)^n$, le morphisme $\mathcal{M}(\mathcal{A}_r)\to \mathcal{M(A)}$ est un morphisme de descente relativement à $\Psi$.
\item Tout morphisme plat, fini et surjectif $\mathcal{M(B)}\to \mathcal{M(A)}$ est un morphisme de descente relativement à $\Psi$.
\item Tout $G$-recouvrement fini $\amalg_{i=1}^n S_i\to S$ de $S$ par des domaines affinoïdes en nombre fini est un morphisme de descente relativement à $\Psi$.

\end{enumerate}

Alors tout morphisme fidèlement plat entre espaces affinoïdes est un morphisme de descente relativement à $\Psi$.
\end{theo}

\section{Application à certains pseudo-foncteurs particuliers}

L'idée de cette section est d'appliquer les deux théorèmes généraux de la section précédente à des pseudo-foncteurs particuliers pour obtenir des résultats de descente.
\subsection{Équivalence dans le cas des modules cohérents, pleine fidélité pour les \texorpdfstring{$S$}   -- espaces analytiques quelconques}

Les résultat suivant sont dûs à Berkovich dans \autocite{PMIHES_1993__78__5_0} en 1.2.0. et 1.3.2. Il indiquent juste qu'on peut toujours recoller des modules cohérents ou bien des morphismes d'espaces analytiques au dessus de $G$-recouvrements finis.

\begin{prop} \label{hyp3} Considérons un $G$-recouvrement $p:\amalg_{i\in I} S_i\to S$ d'un espace affinoïde $S$ par un nombre fini de domaines affinoïdes $S_i$. 

\begin{enumerate} \label{base}
\item Soit $\Psi$ le pseudo-foncteur qui à un espace affinoïde $S$ associe la catégorie des modules cohérents au dessus de $S$. Alors $p$ est un morphisme de descente effective .
\item Soit $\Phi$ le pseudo-foncteur qui à un espace affinoïde $S$ associe la catégorie des espaces affinoïdes $X$ au dessus de $S$. Alors le morphisme $p$ est un morphisme de descente.

%\item Soit $\Theta$ le pseudo-foncteur qui à un espace analytique $S$ associe la catégorie des espaces analytiques $X$ au dessus de $S$. Alors toute donnée de descente affinoïde $f:X'\to S'$ est effective, c'est à dire qu'il existe un espace analytique $X$ au dessus de $S$ dont le morphisme structural $f:X\to S$ vérifie $f_{\vert f^{-1}(S_i)}=f'$.
\end{enumerate} 
\end{prop}
\begin{lemm} $\label{equiv}$ Le pseudo-foncteur $\Psi$ qui à un espace affinoïde $S=\mathcal{M(A)}$ associe l'ensemble des modules cohérents au dessus de $S$ vérifie la propriété suivante : si $r\in (\mathbb{R}_+^*)^n$ est un polyrayon $k$-libre, un morphisme entre espaces affinoïdes $p:S'=\mathcal{M(B)}\to S=\mathcal{M(A)}$ est de descente effective pour $\Psi$ si le morphisme $S'_r\to S_r$ l'est pour le pseudo-foncteur obtenu par restriction de $\Psi$ à la catégorie des espaces $k_r$-affinoïdes.
\end{lemm}

\begin{proof}[Preuve]
Supposons que $S_r'\to S_r$ est de descente effective. On se donne $M$ et $N$ deux modules cohérents au dessus de $S$. Le morphisme $S'\to S$ est de descente, si la suite suivante est exacte : 
$$0\to M\to M\hat{\otimes}_\mathcal{A} \mathcal{B}\to M\hat{\otimes}_\mathcal{A} \mathcal{B}\hat{\otimes}_\mathcal{A} \mathcal{B}$$
et cette suite est exacte si et seulement si elle le reste après chapeau-tensorisation par $k_r$, ce qui est vrai puisque $S'_r\to S_r$ est de descente.

Pour l'effectivité, si on se donne un module cohérent $M'$ au dessus de $S'$ muni de données de descente $\varphi:M'\hat{\otimes}_\mathcal{A} \mathcal{B}\to \mathcal{B}\hat{\otimes}_\mathcal{A }M'$, on voudrait montrer que si l'on désigne par $M$ le $\mathcal{A}$-module définit par $M=\{x\in M',\varphi(x\hat{\otimes}1)=1\hat{\otimes}x\}$, on a un isomorphisme $\lambda:M\hat{\otimes}_\mathcal{A} \mathcal{B}\to M'$, mais la flèche sus-citée est un isomorphisme après chapeau-tensorisation avec $k_r$, donc c'est un isomorphisme. En effet, on peut considérer la suite exacte à gauche évidente : 
$$0\to M\to M'\to M'\hat{\otimes}_\mathcal{A} \mathcal{B}$$ On munit $M$ de la topologie induite, et on dispose donc d'une suite exacte admissible de modules de Banach, d'où par la proposition 2.1.2 de \autocite{berkovich2012spectral} une suite exacte 
$$0\to M_r\to M_r'\to M_r'\hat{\otimes}_\mathcal{{A}_r} \mathcal{B}_r$$
Puisque $S'_r\to S$ est de descente effective, on en déduit que l'on a un isomorphisme de $\mathcal{B}_r$-modules finis munis de données de descente $M_r\hat{\otimes}_{\mathcal{A}_r}\mathcal{B}_r \to M'_r$ mais cette flèche n'est autre que $\lambda_r$ donc par la proposition 2.1.2 de \autocite{berkovich2012spectral}, $\lambda $ est bien un isomorphisme.

De plus, par la proposition 2.1.11 de \autocite{berkovich2012spectral}, le $\mathcal{A}$-module $M$ est bien fini, puisque $M_r$ est fini au dessus de $\mathcal{A}_r$, d'où l'effectivité de $p$.
\end{proof}

La proposition suivante repose sur les calculs dus à Ducros dans \autocite{ducros2020devisser}.
\begin{prop} Considérons le pseudo-foncteur qui à un espace affinoïde $S=\mathcal{M(A)}$ associe la catégorie des $\mathcal{A}$-modules de Banach au dessus de $S$. Alors pour tout polyrayon $k$-libre $r\in (\mathbb{R}_+^*)^n$, la flèche canonique $p:\mathcal{M}(\mathcal{A}_r)\to \mathcal{M(A)}$ est une flèche de descente effective.
\end{prop} 

\begin{proof}[Démonstration]

Le fait que $p$ soit un morphisme de descente est évident puisque si $M$ un $\mathcal{A}$-module de Banach, on a une suite exacte admissible $0\to M\to M\hat{\otimes}_k k_r\to M\hat{\otimes}_k k_r\hat{\otimes}_k k_r$ puisque l'inclusion de $M$ dans $M_r$ est une isométrie. On montre donc l'effectivité. On se donne donc $(M,\varphi)$ un $\mathcal{A}_r$-module de Banach muni de données de descente.

Notons $T$ la famille des fonctions coordonnées sur $k_r$, et $T_1$ et $T_2$ les deux familles de fonctions coordonnées de l'anneau $k_r\hat{\otimes}k_r$. Posons $B:=\mathcal{A}_r\hat{\otimes}\mathcal{A}_r$. On dispose d'identifications :

$$\mathcal{A}_r=\mathcal{A}\{r^{-1}T,rT^{-1}\}=\mathcal{A}\{r^{-1}T_1,rT_1^{-1}\}=\mathcal{A}\{r^{-1}T_2,rT_2^{-1}\}$$	

Lorsqu'un objet mathématique sera défini au dessus de $\mathcal{A}\{r^{-1}T_i,rT_i^{-1}\}$, pour $i=1$, $2$ ou $3$, nous l'indiquerons en indice. On a ainsi la relation : 

$$B=\mathcal{A}_{r,1}\{r^{-1}T_2,rT_2^{-1}\}=\mathcal{A}_{r,2}\{r^{-1}T_1,rT_1^{-1}\}.$$

La donnée de descente $\varphi$ est donc un isomorphisme de $B$-modules de $M_1\{r^{-1}T_2,rT_2^{-1}\}$ vers $M_2\{r^{-1}T_1,rT_1^{-1}\}.$

L'isomorphisme $\varphi$ vérifie la condition de cocycle $\varphi_{13}=\varphi_{23}\circ\varphi_{12}$, comme sur le diagramme suivant :

\[\xymatrix{
    {M_1\{r^{-1}T_2,rT_2^{-1},r^{-1}T_3,rT_3^{-1}\}} \ar[r]^{\varphi_{12}}  \ar[rd]_{\varphi_{13}} & {M_2\{r^{-1}T_1,rT_1^{-1},r^{-1}T_3,rT_3^{-1}\}} \ar[d]_{\varphi_{23}} \\
    & {M_3\{r^{-1}T_1,rT_1^{-1},r^{-1}T_2,rT_2^{-1}\}}
  }
  \]

On va maintenant appliquer cette relation sur un élément générique $\sum_{I,J} m_{I,J,1}T_2^I T_3^J$ dans ${M_1\{r^{-1}T_2,rT_2^{-1},r^{-1}T_3,rT_3^{-1}\}}$. 
Pour tout $I$, écrivons $\varphi(\sum_{J}m_{I,J,1}T_2^J)=\sum_{J}n_{I,J,2}T_1^J$. Pour tout $J$, écrivons $\varphi(\sum_{I} m_{I,J,1}T_2^I)=\sum_I l_{I,J,2}T_1^I$. Écrivons enfin pour tout $I$, $\varphi(\sum_J l_{I,J,1}T_2^J)=\sum_J \lambda_{I,J,2}T_1^J$. 

Alors par définition, $\varphi_{13}(\sum_{I,J} m_{I,J,1}T_2^I T_3^J)=\sum_{I,J}n_{I,J,3}T_1^J T_2^I$. On dispose aussi de la relation $\varphi_{12}(\sum_{I,J} m_{I,J,1}T_2^I T_3^J)=\sum_{I,J}l_{I,J,2}T_1^I T_3^J$, et $\varphi_{23}(\sum_{I,J}l_{I,J,2}T_1^I T_3^J)=\sum_{I,J}\lambda_{I,J,3}T_1^I T_2^J$, et l'on en déduit ainsi la relation $n_{I,J}=\lambda_{J,I}$.

Maintenant, posons $M_0:=\{m\in M,\varphi(m_1)=m_2\}$. Alors $M_0$ est un $\mathcal{A}$-module de Banach complet comme fermé d'un espace complet. Considérons $m\in M$, et écrivons $\varphi(m_1)=\sum_J \mu_{J,2}T_1^J$. Appliquons ce qui précède avec $m_{I,J}=m$ si $(I,J)=(0,0)$ et $0$ sinon. Alors en gardant les mêmes notations, puisque $\varphi_{13}(m)=\sum_J \mu_{J,3}T_1^J$, on en déduit que $n_{I,J}=0$ si $I\neq 0$, et $n_{0,J}=\mu_J$. De même, puisque $\varphi_{12}(m)=\sum_J \mu_{J,2}T_1^J$, on obtient $l_{I,J}=0$ si $J\neq 0$, et $l_{I,0}=\mu_I$. Enfin, par définition des $\lambda_{I,J}$, on a $\varphi(l_{I,0,1})=\sum_{J} \lambda_{I,J,2}T_1^J$ soit encore grâce à la relation obtenue à la fin du paragraphe précédent $\varphi(\mu_I)=\sum_J n_{J,I,2} T_1^J=n_{0,I,2}=\mu_{I,2}$. Cela démontre que $\mu_I$ est dans $M_0$ pour tout $I$. 

Maintenant, le reste découle de ce résultat. En effet, par injectivité du morphisme $\varphi$, puisque $\varphi(\sum_J \mu_{J,1}T_1^J)=\sum_J \mu_{J,2}T_1^J=\varphi(m_1)$, on en déduit que $m_1=\sum_J \mu_{J,1}T_1^J$, ce qui montre que le morphisme naturel $\bar{\omega}:M_0\{r^{-1}T,rT^{-1}\}\to M_1$, qui envoie $\sum_J \alpha_J T^J$ sur $\sum_J \alpha_{J,1}T_1^J$ est surjectif. De plus, $\varphi(\sum_J \alpha_{J,1}T_1^J)=\sum_J \alpha_{J,2} T_1^J$, ce qui montre que ce morphisme est aussi injectif, et admissible (car $\varphi$ est admissible). De plus, modulo l'isomorphisme $M\simeq M_0\hat{\otimes}\mathcal{A}_r$, l'isomorphisme $\varphi$ est simplement l'égalité : 

$$M_0 \{ r^{-1} T_1,r T_1^{-1} \} \{ {r^{-1} T_2, rT_2^{-1}\}} \simeq M_0 \{ {r^{-1}T_2,rT_2^{-1}\} \{r^{-1}T_1,rT_1^{-1}\} }$$
Donc l'isomorphisme précédent est bien un isomorphisme de données de descente, et cela démontre que la flèche canonique $\mathcal{M}(\mathcal{A}_r)\to \mathcal{M(A)}$ est une flèche de descente effective. \qedhere

\end{proof}

\begin{rema} Ainsi, si l'on se donne un module $M'$ fini au dessus de $A_r$ muni de données de descente, on peut descendre le module de Banach sous jacent en un module de Banach $M$, et par la proposition 2.1.11 de \autocite{berkovich2012spectral}, c'est un module fini.

De même, si l'on se donne $\mathcal{D}$ une $\mathcal{A}_r$-algèbre de Banach munie de données de descente, on peut descendre le module de Banach sous-jacent grâce au théorème précédent, qui fournit $\mathcal{D}_0$ un $\mathcal{A}$-module de Banach, qui est en fait une $\mathcal{A}$-algèbre de Banach puisque la multiplication est une application bilinéaire bornée. De plus, la même démonstration que la proposition 2.1.8 de \autocite{berkovich2012spectral} montre que $\mathcal{D}$ est une algèbre $\mathcal{A}_r$-affinoïde si et seulement si $\mathcal{D}_0$ est une algèbre $\mathcal{A}$-affinoïde.

\end{rema}

Ces remarques montrent les deux propositions suivantes : 

\begin{prop} \label{oridu} Considérons le pseudo-foncteur qui à un espace $k$-affinoïde $S=\mathcal{M(A)}$ associe la catégorie des $\mathcal{A}$-modules cohérents au dessus de $S$. Alors pour tout polyrayon $k$-libre $r\in (\mathbb{R}_+^*)^n$, la flèche canonique $p:\mathcal{M}(\mathcal{A}_r)\to \mathcal{M(A)}$ est une flèche de descente effective pour ce pseudo-foncteur.
\end{prop}

\begin{prop} \label{hyp1} Considérons le pseudo-foncteur qui à un espace $k$-affinoïde $S=\mathcal{M(A)}$ associe la catégorie des espaces analytiques $k$-affinoïdes au dessus de $S$ (resp. la catégorie des algèbres $\mathcal{A}$-affinoïdes). Alors pour tout polyrayon $k$-libre $r\in (\mathbb{R}_+^*)^n$, la flèche canonique $p:\mathcal{M}(\mathcal{A}_r)\to \mathcal{M(A)}$ est une flèche de descente effective pour ce pseudo-foncteur. 
\end{prop}

\begin{rema}\label{polycardef}

Considérons un morphisme de schémas affines $p:\Spec B\to \Spec A$ fini et localement libre. Considérons $b\in B$ et $(f_i)_{i\in I}$ une famille finie d'éléments de $A$ qui génère l'idéal unité, et tel que $B_{f_i}$ est libre au dessus de $A_{f_i}$. Alors, la multiplication par $b$ induit un endomorphisme du $A_{f_i}$-module $B_{f_i}$, et possède donc un polynôme caractéristique $\chi_i\in A_{f_i}[X]$. Puisque le polynôme caractéristique d'un endomorphisme d'un module libre et fini est invariant par changement de base, on en déduit que les polynômes $\chi_i$ et $\chi_j$ sont égaux dans $A_{f_i f_j}[X]$, et par définition du faisceau structural d'un schéma affine, on en déduit l'existence d'un polynôme $\chi_b\in A[X]$ dont l'image dans chaque $A_{f_i}[X]$ est $\chi_i$. Ce polynôme est indépendant du recouvrement de $\Spec A$ choisi, et vérifie encore la relation $\chi_b(b)=0$, que l'on peut vérifier dans chaque $A_{f_i}[X]$. 
\end{rema}
On énonce maintenant deux lemme concernant les polynômes caractéristiques qui serviront dans la suite. 

\begin{lemm} \label{polycarbc}Considérons des anneaux $A$, $B$, et $C$ munis de morphismes d'anneaux $\psi:A\to B$, $h:A\to C$, tels que $\psi$ soit fini et localement libre. Notons $h':B\to B\otimes_A C$ et $\psi':C \to B\otimes_A C$ les deux morphismes canoniques.

Notons encore $h$ le morphisme induit par $h$ de $A[X]$ vers $C[X]$. Soit $b\in B$. Le morphisme $\Spec B\otimes_A C \to \Spec C$ est fini et localement libre, donc $h'(b)$ possède un polynôme caractéristique relativement à $\psi'$, que l'on note $\chi_{h'(f)}\in C[X]$. Notons $\chi_f\in A[X]$ le polynôme caractéristique de $f$ relativement à $\psi$. Alors, on dispose de l'égalité $h(\chi_f)=\chi_{h'(f)}$.
\end{lemm}

\begin{proof} La situation est résumée par le diagramme suivant : 
$$\xymatrix{
    B\otimes_A C    & C \ar[l]^{\psi'} \\
   B \ar[u]^{h'}  & A\ar[u]^h \ar[l]^{\psi}
  }$$ 

Soit $x\in \Spec A$. Par les propriétés des morphismes finis et localement libres, il existe $V$ un voisinage affine de $x$ dans $\Spec A$, tel que l'image inverse de $V$ sur $\Spec B$ est un ouvert affine $W$ de $\Spec B$ tel que $\mathcal{O}(W)$ est un $\mathcal{O}(V)$-module libre. Considérons l'image inverse $U$ de $V$ sur $\Spec C$. Alors, l'image inverse de $U$ sur $\Spec D$ s'identifie à $\Spec (\mathcal{O}(U)\otimes_{\mathcal{O}(V)} \mathcal{O}(W)) $ qui est un $\mathcal{O}(U)$-module libre. Maintenant, puisque la définition des polynôme caractéristiques est locale, il suffit de montrer que l'égalité du lemme vaut dans $\mathcal{O}(V)[X]$, et on peut donc supposer que $B$ est un $A$-module libre.

Maintenant, on écrit $B$ comme une somme directe finie $B=\bigoplus_{i\in I} {A b_i}$ avec $b_i\in B$, et $I=\llbracket 1,n\rrbracket$. Notons $F=(f_{ij})_{(i,j)\in I\times I}$ la matrice de la multiplication par $f$ dans la base $(b_i)_{i\in I}$. Alors, on dispose de l'égalité $B\otimes_A C=\bigoplus_{i\in I} {C h'(b_i)}$. On a aussi la relation $h'(f)h'(b_j)=h'(fb_j)=h'(\sum_{i\in I} f_{ij} d_i)=\sum_{i\in I} h(f_{ij}) d_i$ par les propriétés du produit tensoriel, puisque pour tout $i,j\in I$, $f_{ij}\in A$. Cela montre que la matrice de la multiplication par $h'(f)$ dans la base $(h'(b_i))_{i\in I}$ est la matrice $(h(f_{ij}))_{(i,j)\in I\times I}$.

Notons $\chi_f=\sum_{i=1}^n a_i X^i$. Alors, on a $a_i=(-1)^{n-k} \sum_{\vert J \vert =n-k} F[J]$ avec $F[J]$ le mineur principal de la matrice $F$ obtenu en enlevant les colonnes et les lignes dont les indices sont dans $J$. On en déduit que $h(a_i)=(-1)^{n-k} \sum_{\vert J \vert =n-k} h(F[J])$, et puisque $h$ est un morphisme d'algèbre, et que le déterminant est un polynôme en les coefficients, on en déduit que pour tout mineur principal $F[J]$, on dispose de l'égalité $h(F[J])=G[J]$ avec $G$ la matrice de taille $I\times I$ à coefficient dans $C$ définie par $G=(h(f_{ij})_{(i,j)\in I\times I})$. Par le paragraphe précédent, c'est exactement la matrice de la multiplication par $h'(f)$ dans la base $(h'(b_i))_{i\in I}$, et puisque la formule exprimant les coefficients du polynôme caractéristique en fonction des mineurs principaux reste valable, on en déduit que les coefficients de $\chi_{h'(f)}$ sont exactement $h(a_i)$, ce qui prouve que l'on a l'égalité $h(\chi_f)=\chi_{h'(f)}$. 
\end{proof}

Le lemme suivant exprime la compatibilité du polynôme caractéristique avec la structure produit d'une $k$-algèbre.

\begin{lemm}\label{polycarprod}Considérons $k$ un corps, et une $k$-algèbre finie de la forme $A=\prod_{i=1}^n A_i$ avec $A_i$ une $k$-algèbre finie pour tout $i\in \{1,...,n\}$. Soit $f=(f_1,...,f_n)\in A$. Alors, on dispose de l'égalité $\chi_f=\prod_{i=1}^n\chi_{f_i}$, où $\chi_{f_i}$ est le polynôme caractéristique de $f_i$ vu comme élément de $A_i$.
\end{lemm}

\begin{proof}
Il suffit de se donner pour tout $i\in \{1,...,n\}$ une base $\mathcal{B}_i$ de $A_i$. Alors, si l'on note $M_i$ la matrice de la multiplication par $f_i$ dans la base $\mathcal{B}_i$, la matrice de la multiplication par $f$ dans la base de $A$ obtenue à partir de chaque base $\mathcal{B}_i$ est une matrice diagonale par bloc de $M_i$. Le résultat s'en déduit en prenant le polynôme caractéristique, puisque le déterminant d'une matrice diagonale par bloc est le produit des déterminants de chacun de ses blocs.
\end{proof}

On donne un lemme qui relie la norme des coefficients de $\chi_f$ à la norme de $f$.

\begin{lemm}\label{polycarmajor}
Considérons $k\to L$ une extension finie de corps ultramétriques complets. Soit $f\in L$, tel que $\vert f \vert \leqslant  1$. Alors, si $\overline{k}$ est une clôture algébrique de $k$ qui contient $L$, les conjugués de $f$ dans $\overline{k}$ sont de norme plus petite que 1. En particulier, le polynôme $\chi_f$ est à coefficient dans $\overset{\circ}{k}$, c'est à dire que tous ses coefficients sont de norme inférieure à 1. 
\end{lemm}

\begin{proof}
Considérons $\overline{k}$ une clôture algébrique de $k$ contenant $L$. Les racines de $\chi_f$ sont exactement les racines du polynôme minimal de $f$. Puisque le polynôme minimal $P_f$ de $f$ est irréductible sur $k$, le groupe de Galois $\mathrm{Aut}(\overline{k}/k)$ agit transitivement sur les racines de $P_f$, donc pour toute racine de $\alpha\in \overline{k}$, il existe $\gamma\in\mathrm{Aut}(\overline{k}/k) $ tel que $\gamma \alpha=f$. Par le lemme 3.8.1.4 de \autocite{bosch1984non}, on en déduit que $\vert \alpha \vert \leqslant 1$, et puisque $\chi_f$ est unitaire, par inégalité triangulaire ultramétrique, en développant une écriture $\chi_f=\prod_\alpha (X-\alpha)^{n_\alpha} $, où le produit porte sur l'ensemble des racines de $P_f$, on en déduit que $\chi_f$ est bien à coefficients dans $\overset{\circ}{k}$.
\end{proof}

Ce dernier lemme nous permettra de relier le polynôme caractéristique d'un élément  d'une $k$-algèbre avec le polynôme caractéristique de l'image de cet élément dans l'algèbre réduite associée. 

\begin{lemm} Considérons un corps $k$, et une $k$-algèbre finie locale $A$, dont nous noterons $m$ l'idéal maximal. Notons $\pi:A\to A/m$ le morphisme quotient. Soit $f\in A$. Alors, le polynôme $\chi_f$ possède les mêmes facteurs irréductibles que le polynôme $\chi_{\pi(f)}$.

\end{lemm}
\begin{proof} Les polynômes caractéristiques et minimum d'un endomorphisme d'un $k$-espace vectoriels possèdent les mêmes facteurs irréductibles. 

Considérons $R$ un polynôme irréductible à coefficient dans $k$. Alors, $R$ divise $\chi_f$ (resp. $\chi_{\pi(f)}$) si et seulement si $R(f)$ (resp. $R(\pi(f))$) est non inversible. La propriété est évidente pour $\chi_{\pi(f)}$ puisque $A/m$ est un corps, on montre la propriété pour $\chi_{\pi(f)}$.

Si $R$ divise $\chi_f$, alors $R$ divise $P_f$ le polynôme minimal de $f$, donc $P_f=RQ$ avec $Q$ un polynôme à coefficient dans $k$ de degré strictement plus petit que $P_f$. Maintenant, $P_f(f)=0=R(f)Q(f)$ et si $R(f)$ était inversible, alors $Q$ serait un polynôme à coefficients dan $k$ annulant $f$ de degré strictement plus petit que $P_f$ ; c'est absurde, donc $R(f)$ est non inversible. 

Réciproquement, si $R(f)$ est non inversible, alors la multiplication par $R(f)$ est non surjective, donc non injective, et il existe $v\in A$ non nul tel que $R(f)v=0$. Si l'on prend $L$ une cloture algébrique de $k$, et que l'on décompose $R=\prod_{i=1}^n (X-\alpha_i)^n_i$, avec $\alpha\in L$, il existe donc $\i\in \{1,...,n\}$ tel que $(m_f-\alpha_i Id)$ est non-inversible (sinon $v=0$), où l'on a noté $m_f$ l'endomorphisme de $A$ égal à la multiplication par $f$. Cela montre que $\alpha_i$ est une valeur propre de $m_f$, donc une racine de $\chi_f$. Maintenant, le pgcd de $R$ et $\chi_f$ divise $R$, c'est donc $R$ ou $1$. Ce même pgcd est invariant par extension des scalaires (calculé par l'algorithme d'Euclide, toutes les opérations restent dans le corps de base), donc on peut le calculer dans $L[X]$. Or puisque $(X-\alpha_i)$ divise $R$ et divise $\chi_f$, ce pgcd ne peut être égal à 1, donc il vaut $R$, et $R$ divise $\chi_f$.

La démonstration du lemme est alors immédiate : prenons $R$ un polynôme irréductible de $k[X]$. Alors, $R$ divise $\chi_f$ si et seulement si $R(f)$ est non inversible dans $A$ si et seulement si $R(f)$ n'est pas dans $m$ l'idéal maximal de $A$ (car $A$ est local) si et seulement si $\pi(R(f))$ est non nul si et seulement si $\pi(R(f))=R(\pi(f))$ est non inversible si et seulement si $R$ divise $\chi_{\pi(f)}$. Cela montre que $\chi_f$ et $\chi_{\pi(f)}$ ont les mêmes facteurs irréductibles. \qedhere 

\end{proof}

\begin{rema} \label{redfinale}
Supposons que le corps $k$ est ultramétrique et complet, et que l'algèbre $A$ est une $k$-algèbre de Banach finie et locale, d'idéal maximal $m$, et de morphisme canonique $\pi:A\to A/m$. Soit $f\in A$. Fixons une clôture algébrique $L$ de $k$ qui contient $A/m$. Supposons que l'on ait montré que toute racine de $\chi_{\pi(f)}$ dans $L$ est de norme inférieure à 1. Alors, par le lemme précédent, si l'on note $\alpha$ une racine de $\chi_f$ dans $L$, alors $\alpha$ est aussi une racine de $\chi_{\pi(f)}$, donc $\alpha $ est de norme inférieure à 1. En développant l'écriture de $\chi_f$ comme produit de facteur $(X-\alpha)$ comptés avec multiplicité, avec $\vert \alpha\vert \leqslant 1$, puis en utilisant l'inégalité triangulaire ultramétrique, on en déduit que $\chi_f$ est aussi à coefficients dans $\overset{\circ}{k}$.
\end{rema}

Les lemmes qui précèdent vont maintenant servir dans la preuve de la proposition suivante.

\begin{prop}\label{reparation}
Considérons un morphisme d'espaces $k$-affinoïdes $p:\mathcal{M(B)\to \mathcal{M(A)}}$ fidèlement plat et fini. Soit $\mathcal{D}$ une $\mathcal{B}$-algèbre de Banach, et $\mathcal{C}$ une sous $\mathcal{A}$-algèbre fermée de $\mathcal{D}$. On suppose que le morphisme borné naturel de $\mathcal{B}$-algèbres de  $\varphi:\mathcal{C}\otimes_\mathcal{A} \mathcal{B} \to \mathcal{D}$ est un isomorphisme, ainsi que le morphisme obtenu par changement de base $\varphi\hat{\otimes}_k k_r$, pour tout polyrayon libre non trivial.

Alors, l'algèbre $\mathcal{C}$ est $k$-affinoïde si et seulement si l'algèbre $\mathcal{D}$ est $k$-affinoïde, et alors $\varphi$ est un isomorphisme d'algèbres $k$-affinoïdes lorsque l'on munit $\mathcal{C}\otimes_\mathcal{A} \mathcal{B}$ de sa norme naturelle d'algèbre tensorielle complétée. \end{prop}
\begin{proof} Notons $r$ un polyrayon libre. Par la proposition 2.1.8 de \autocite{berkovich2012spectral}, une $k$-algèbre $\mathcal{C}$ est $k$-affinoïde si et seulement si la $k_r$-algèbre $C\hat{\otimes_k}k_r$ est $k_r$-affinoïde. On peut donc supposer que les algèbres $\mathcal{A}$ et $\mathcal{B}$ sont strictement $k$-affinoïdes, et que le corps $k$ est non trivialement valué, et montrer que sous les mêmes hypothèses, la $k$-algèbre $\mathcal{C}$ est strictement $k$-affinoïde si et seulement si l'algèbre $\mathcal{D}$ est strictement $k$-affinoïde. En effet, les hypothèses sont invariantes par extension des scalaires à $k_r$, et si cette assertion est démontrée, supposons que $C$ est $k$-affinoïde. On choisit $r$ un polyrayon libre tel que $\mathcal{A}\hat{\otimes}_k k_r$ et $\mathcal{B}\hat{\otimes}_k k_r$ et $\mathcal{C}\hat{\otimes}_k k_r$ soient strictement $k_r$-affinoïdes, et $k_r$ non trivialement valué. Alors, $\mathcal{D}\hat{\otimes}_k k_r$ est strictement $k_r$-affinoïde, donc $\mathcal{D}$ est $k$-affinoïde. Si maintenant c'est $\mathcal{D}$ qui est $k$-affinoïde, on choisit $r$ un polyrayon libre tel que $\mathcal{A}\hat{\otimes}_k k_r$ et $\mathcal{B}\hat{\otimes}_k k_r$ et $\mathcal{D}\hat{\otimes}_k k_r$ soient strictement $k_r$-affinoïdes, et $k_r$ non trivialement valué, et alors $\mathcal{C}\hat{\otimes}_k k_r$ est strictement $k_r$-affinoïde, donc $\mathcal{C}$ est $k$-affinoïde.

Supposons que $\mathcal{C}$ est strictement $k$-affinoïde. Alors, la proposition 6.1.3.4 de \autocite{bosch1984non} appliquée au morphisme fini de $k$-algèbres $\mathcal{C}\to \mathcal{D}$ montre que $\mathcal{D}$ est strictement  $k$-affinoïde, et par une nouvelle application de 6.1.3.4 de \autocite{bosch1984non} au morphisme $\varphi$, on en déduit que $\varphi$ est un isomorphisme d'algèbres $k$-affinoïdes de $\mathcal{C}\otimes_\mathcal{A} \mathcal{B}$ muni de sa norme naturelle d'algèbre tensorielle complétée vers $\mathcal{D}$.

Supposons maintenant que $\mathcal{D}$ est strictement $k$-affinoïde. Alors, par la proposition 6.1.3.4 de \autocite{bosch1984non} appliquée à l'inverse de $\varphi$, l'algèbre $\mathcal{C}\otimes_\mathcal{A} \mathcal{B}$ munie de sa norme tensorielle complété est strictement $k$-affinoïde, et $\varphi$ est un isomorphisme d'algèbres $k$-affinoïdes de $\mathcal{C}\otimes_\mathcal{A} \mathcal{B}$ muni de sa norme naturelle d'algèbre tensorielle complétée vers $D$. On veut maintenant montrer que l'algèbre $C$ est strictement $k$-affinoïde. L'idée est d'appliquer la proposition 6.3.3.2 de \autocite{bosch1984non} au morphisme $\mathcal{C}\to \mathcal{D}$. Pour cela, il suffit de démontrer que l'anneau $\overset{\circ}{\mathcal{D}}$ est entier sur $\overset{\circ}{\mathcal{C}}$, où $\overset{\circ}{\mathcal{D}}$ (resp. $\overset{\circ}{\mathcal{C}}$) désigne l'ensemble des éléments dont les puissances successives sont bornées sur $\mathcal{D}$ (resp. $\mathcal{C}$).

Puisque $\mathcal{D}$ est strictement $k$-affinoïde, par la proposition 6.2.3.1 de \autocite{bosch1984non}, on a l'égalité $\overset{\circ}{\mathcal{D}}=\{f\in D, \rho_D (f)\leqslant 1\}$, où $\rho_D$ désigne la norme spectrale sur $\mathcal{D}$. On va montrer l'égalité analogue pour $\mathcal{C}$. L'inclusion $\overset{\circ}{\mathcal{C}}\subset\{f\in C, \rho_C (f)\leqslant 1\}$ est vraie pour n'importe quelle algèbre de Banach, on va donc montrer l'inclusion inverse. Par la proposition 1.3.1 de \autocite{berkovich2012spectral}, pour toute algèbre de Banach ${E}$, on a pour tout $f\in E$, l'égalité $\rho_E(f)=\max_{x\in \mathcal{M}(E)}\vert f(x)\vert$. Soit donc $f\in \mathcal{C}$ tel que $\rho_\mathcal{C} (f)\leqslant 1$. Alors, pour tout $y\in \mathcal{M(C)}$, on dispose de l'inégalité $\vert f (y)\vert \leqslant 1$, donc pour tout $x\in \mathcal{M(D)}$, si l'on note $\psi:\mathcal{C}\to \mathcal{D}$ l'inclusion canonique, alors $\vert \psi (f) \vert (x)\leqslant 1$, donc $\rho_\mathcal{D}(\psi(f))\leqslant 1$, donc $\psi(f)$ est de puissance bornée dans $\mathcal{D}$. Comme $\psi$ préserve la norme ($\mathcal{C}$ est une sous-$\mathcal{A}$ algèbre de $\mathcal{D}$), on en déduit que $f$ est de puissance bornée dans $\mathcal{C}$.

Soit donc $f\in \mathcal{D}$, tel que $\rho_\mathcal{D}(f)\leqslant 1$. Puisque le morphisme $\Spec \mathcal{B}\to \Spec \mathcal{A}$ est fidèlement plat et fini, c'est aussi le cas du morphisme $\Spec \mathcal{D}\to \Spec \mathcal{C}$, qui est donc fini et localement libre, et le polynôme caractéristique $\chi_f\in C[X]$ de la remarque $\ref{polycardef}$ est bien défini. Ce polynôme vérifie la relation $\chi_f(f)=0$, et on va montrer que $\chi_f$ est en fait à coefficients dans $\overset{\circ}{\mathcal{C}}$. Soit $x\in \mathcal{M(C)}$, notons $\chi_f=\sum_{i=0}^n a_n X^i$. Il suffit de montrer que $\vert a_i (x) \vert \leqslant 1$. On dispose du diagramme commutatif suivant : 

$$\xymatrix{
   \mathcal{D}\otimes_\mathcal{C} \mathcal{H}(x)     & \mathcal{H}(x)\ar[l]  \\
   \mathcal{D} \ar[u]& \mathcal{C}\ar[l] \ar[u]
  }$$

Notons $\psi'$ le morphisme canonique de $\mathcal{H}(x)$ dans $\mathcal{D}\otimes_\mathcal{C}\mathcal{H}(x)$ et $h'$ le morphisme canonique de $\mathcal{D}$ dans $\mathcal{D} \otimes_\mathcal{C} \mathcal{H}(x)$. Par le lemme $\ref{polycarbc}$, on dispose de l'égalité $\sum_{i=0}^n a_i (x)X^i=\chi_{h'(f)}$, où $a_i(x)$ désigne l'image de $a_i\in C$ dans $\mathcal{H}(x)$. Pour montrer que $\vert a_i (x) \vert \leqslant 1$, on va montrer que $\chi_{h'(f)}\in \overset{\circ}{\mathcal{H}(x)}[X]$. La $\mathcal{H}(x)$-algèbre $\mathcal{D} \otimes_\mathcal{C} \mathcal{H}(x)$ est finie, elle est donc artinienne, et on peut lui appliquer le théorème de structure des anneaux artiniens, qui est le lemme 10.53.6 de  \cite[\href{https://stacks.math.columbia.edu/tag/00JB}{Lemma 00JB}]{stacks-project} : on a $\mathcal{D} \otimes_\mathcal{C} \mathcal{H}(x)=\prod_{i=1}^M \mathcal{D}_i$ avec $\mathcal{D}_i$ des anneaux artiniens locaux pour tout $i\in \{1,...,M\}$. Notons $h'(f)=(f_1,...,f_M)$ avec $f_i\in \mathcal{D}_i$ pour tout $i\in \{1,...,M\}$. Alors, par le lemme $\ref{polycarprod}$, pour montrer que $\chi_{h'(f)}$ est dans $\overset{\circ}{\mathcal{H}(x)}[X]$, il suffit de montrer que chacun des $\chi_{f_i}$ est dans $\overset{\circ}{\mathcal{H}(x)}[X]$. Par la remarque $\ref{redfinale}$, si l'on note $m_i$ l'idéal maximal de $\mathcal{D}_i$, et $\pi_i:\mathcal{D}_i\to \mathcal{D}_i /m_i$ la projection canonique, il suffit de montrer que toute racine de $\chi_{\pi_i(f_i)}$ dans une extension algébriquement close de $\mathcal{H}(x)$ est de norme inférieure à 1. Maintenant, on dispose d'un morphisme continu $z:\mathcal{D}\to \mathcal{D}_i/ m_i$, et comme $\rho_\mathcal{D}(f)\leqslant 1$, on en déduit que $\vert z(f) \vert =\vert \pi_i (f_i )\vert\leqslant 1$, où $\vert .\vert$ désigne ici l'unique norme sur l'extension finie $D_i/m_i$ de $\mathcal{H}(x)$ qui prolonge la norme de $\mathcal{H}(x)$. On est maintenant en mesure d'appliquer le lemme $\ref{polycarmajor}$, qui montre que toute racine de $\chi_{\pi_i (f_i) }$ dans une cloture algébrique de $\mathcal{H}(x)$ est de norme plus petite que 1. On en déduit par la remarque $\ref{redfinale}$ que $\chi_{f_i}$ est à coefficient dans $\overset{\circ}{\mathcal{H}(x)}$, et c'est donc le cas du produit $\chi_{h'(f)}=\sum_{i=0}^n a_i(x) X^i$, et comme cela vaut pout tout $x\in \mathcal{M(C)}$, le polynôme $\chi_f=\sum_{i=0}^n a_i X^i $ est lui même bien à coefficients dans $\overset{\circ}{\mathcal{C}}$ et comme c'est un polynôme annulateur de $f$, on en déduit que $f\in \overset{\circ}{\mathcal{D}}$ est bien entier sur $\overset{\circ}{\mathcal{C}}$, et comme cela vaut pour tout $f\in \overset{\circ}{\mathcal{D}}$, on est en mesure d'appliquer la proposition 6.3.3.2 de \autocite{bosch1984non} pour obtenir que l'algèbre $\mathcal{C}$ est bien strictement $k$-affinoïde. \qedhere 

\end{proof}

\begin{prop} \label{hyp2} Considérons le pseudo-foncteur au dessus de la catégorie des espaces $k$-affinoïdes qui à un espace affinoïde $S=\mathcal{M(A)}$ associe la catégorie des espaces $k$-affinoïdes au dessus de $S$. Alors tout morphisme fini, plat et surjectif entre espaces affinoïdes $S'=\mathcal{M(B)}\to S=\mathcal{M(A)}$ est un morphisme de descente effective . 
\end{prop}

\begin{proof}[Démonstration]

Pour l'effectivité, considérons un morphisme plat, fini et surjectif d'espaces $k$-affinoïde $p:\mathcal{M}(\mathcal{B})\to\mathcal{M}(\mathcal{A})$ induit par un morphisme d'algèbre $k$-affinoïde $\epsilon:\mathcal{A}\to \mathcal{B}$. Considérons maintenant $\mathcal D$ une algèbre $k$-affinoïde au dessus de $\mathcal{B}$ munie de données de descente relativement au morphisme $p$, c'est à dire d'un isomorphisme de $\mathcal{B}\hat{\otimes}_\mathcal{A} \mathcal{B}$-algèbres de Banach borné $\varphi:\mathcal{D}\hat{\otimes}_{\mathcal{A}} \mathcal{B \to \mathcal{B}\hat{\otimes}_\mathcal{A} \mathcal{D}}$ vérifiant la condition de cocycle. Puisque $\mathcal{D}$ et $\mathcal{B}$ sont $k$-affinoïdes, elles sont noethériennes, et on en déduit par la proposition 3.7.2.6 de \autocite{bosch1984non} que l'on a les isomorphismes de $\mathcal{B}\otimes_\mathcal{A} \mathcal{B}$-algèbres suivants : 

$$\mathcal{B}\otimes_\mathcal{A} \mathcal{B} \simeq\mathcal{B}\hat{\otimes}_\mathcal{A} \mathcal{B},~~\mathcal{D}\otimes_\mathcal{A} \mathcal{B} \simeq \mathcal{D}\hat{\otimes}_\mathcal{A} \mathcal{B},~~\mathcal{B}\otimes_\mathcal{A} \mathcal{D} \simeq \mathcal{B}\hat{\otimes}_\mathcal{A} \mathcal{D}$$

On a aussi des isomorphismes de $\mathcal{B}\otimes_\mathcal{A}\mathcal{B}\otimes_\mathcal{A}\mathcal{B}\simeq \mathcal{B}\hat{\otimes}_\mathcal{A} \mathcal{B}\hat{\otimes}_\mathcal{A} \mathcal{B}$ algèbres similaires puisque $\mathcal{B}\otimes_\mathcal{A} \mathcal{B}$ est fini sur $A$. Cela montre que la donnée de descente $\varphi$ munit naturellement la $\mathcal{B}$-algèbre $\mathcal{D}$ d'une donnée de descente d'algèbre relativement au morphisme de schémas $\mathrm{Spec} ~\mathcal{B}\to\mathrm{Spec} ~\mathcal{A}$ dont l'isomorphisme est noté $\varphi_s$. Puisque $p$ est plate et surjective, par les propriétés rappelées en introduction, on en déduit que le morphisme d'anneau $\epsilon:\mathcal{A}\to \mathcal{B}$ est fidèlement plat. Par descente schématique des algèbres, il existe une $\mathcal{A}$-algèbre $\mathcal{C}$ telle que l'on ait un isomorphisme de $\mathcal{B}$-algèbres munies de données de descente schématiques $\psi:\mathcal{B}\otimes_\mathcal{A}\mathcal{C}\to \mathcal{D}$. De plus, l'algèbre $\mathcal{C}$ est la sous-algèbre de $\mathcal{D}$ définie par $\mathcal{C}=\{c\in \mathcal{D},\varphi(c\otimes 1)=1\otimes c\}$. C'est donc un fermé d'un espace de Banach, c'est donc une $\mathcal{A}$-algèbre de Banach. De plus, par le lemme  \cite[\href{https://stacks.math.columbia.edu/tag/033E}{Tag 033E}]{stacks-project}, puisque $\mathcal{B}\otimes_\mathcal{A}\mathcal{C}$ est noethérien, $\mathcal{C}$ est bien noethérienne, et l'on a donc encore par \autocite{bosch1984non} 3.7.2.6 un isomorphisme de $\mathcal{B}$-modules de Banach de $\mathcal{B}\otimes_\mathcal{A}\mathcal{C}$ vers $\mathcal{B}\hat{\otimes}_\mathcal{A}\mathcal{C}$. 

Le morphisme $\psi$ est un isomorphisme de $\mathcal{B}$-algèbres borné. On dispose aussi des données de descente $\varphi_r$ sur la $\mathcal{B}_r$-algèbre $\mathcal{D}_r$ relativement au morphisme $p_r$. Notons $h:\mathcal{D}\to \mathcal{B}\otimes_\mathcal{A}\mathcal{D}$ le morphisme de $\mathcal{A}$-algèbres défini par $h(d)=\varphi(d\otimes1)-1\otimes d$. Alors, par définition, $\mathcal{C}=\ker h$, et comme $h$ est un morphisme borné de $k$-algèbres de Banach, on peut vérifier que pour tout polyrayon libre $r$ non trivial, on a $\mathcal{C}\hat{\otimes}k_r=\ker h_r$, égalité qui est vérifiée pour tout morphisme de $k$-algèbres borné. Cela montre que $\mathcal{C}_r$ est bien solution du problème de descente comme $\mathcal{A}_r$-algèbre, donc on a bien un isomorphisme d'algèbres $\mathcal{C}_r{\otimes}_{\mathcal{A}_r}\mathcal{B}_r\to \mathcal{D}_r$. On est exactement dans la situation de la proposition précédente $\ref{reparation}$, qui montre que l'algèbre $\mathcal{C}$ est $k$-affinoïde, et que l'isomorphisme de $\mathcal{B}$-algèbres $\psi:\mathcal{B}\otimes_\mathcal{A}\mathcal{C}\to \mathcal{D}$ est un isomorphisme d'algèbres $k$-affinoïdes, lorsque l'on munit $\mathcal{B}\otimes_\mathcal{A} \mathcal{C}$ de sa norme de produit tensoriel complété. Cet isomorphisme de $k$-algèbre est même un isomorphisme de données de descente analytiques de $\mathcal{B}\otimes_\mathcal{A} \mathcal{C}$ muni de ses données de descente canoniques vers $\mathcal{D}$ muni de $\varphi$ puisque si l'on note $p_i:\mathcal{M}(\mathcal{B}\otimes_\mathcal{A} \mathcal{B})\to \mathcal{M}(B)$ les deux projections canoniques, il suffit de vérifier que $ \varphi\circ p_1^*\psi=p_2^*\psi$, relation vérifiée par définition du morphisme $\psi$, et puisque tous les produits tensoriels sont isomorphes commes algèbres aux produits tensoriels complétés. Cela achève de démontrer  l’effectivité du morphisme $S'\to S$ pour le pseudo-foncteur qui à un espace analytique associe l’ensemble des espaces $k$-affinoïdes au dessus de cet espace.

%ce qui montre que $\mathcal{B}\hat{\otimes}_\mathcal{A} \mathcal{C}$ muni des données de descente analytiques canoniques est isomorphe à $\mathcal{D}$ en tant que $\mathcal{B}$-algèbre muni de données de descente analytiques. 

%De plus, puisque $\mathcal{D}$ est $k$-affinoïde, par le lemme $\ref{aca1}$, la $\mathcal{A}$-algèbre $\mathcal{C}$ est $k$-affinoïde. Ceci achève de démontrer l'effectivité du morphisme $S'\to S$ pour le pseudo-foncteur qui à un espace analytique associe l'ensemble des objets $k$-affinoïdes au dessus de cet espace. Maintenant, par le lemme $\ref{aca}$, si $\mathcal{D}$ est $\mathcal{B}$-affinoïde, alors $\mathcal{C}$ est $\mathcal{A}$-affinoïde, donc on a aussi le résultat d'effectivité pour la catégorie des algèbres $S$-affinoïdes.

Pour montrer que le morphisme $S'\to S$ est un morphisme de descente pour ce pseudo-foncteur, par le lemme $\ref{mybis}$ il suffit de le démontrer pour le morphisme obtenu par changement de base $S'_r\to S_r$ avec $r$ un polyrayon rendant les espaces en jeux strictement affinoïdes, et la valeur absolue sur $k$ non triviale. On peut donc supposer $S'$ et $S$ strictement $k$-affinoïdes. Supposons donc qu'on se donne $\mathcal{C}$ et $\mathcal{D}$ deux algèbres $k$-affinoïdes au dessus de $\mathcal{A}$ munies d'un morphisme de données de descente analytique $\mathcal{D}\hat{\otimes}_A \mathcal{B}\to \mathcal{C}\hat{\otimes}_A \mathcal{B}$, alors on obtient un morphisme de données de descente schématiques relatives à $\varepsilon:\mathcal{A}\to \mathcal{B}$ (les algèbres $\mathcal{C}$ et $\mathcal{D}$ sont noethériennes, et $\varepsilon$ est fini), donc il existe un morphisme de $\mathcal{A}$-algèbres de $\mathcal{D}$ vers $\mathcal{C}$ dont le changement de base schématique induit $f$ par descente schématique. Ce morphisme est continu par le théorème 6.1.3.1 de \autocite{bosch1984non} et donc borné par 2.2.3 de \autocite{temkin2011introduction} donc c'est bien un morphisme d'algèbres ${k}$-affinoïde, solution du problème de descente pour les morphismes, et ce morphisme est unique car déterminé ensemblistement par fidèle platitude du morphisme $\varepsilon$, ce qui achève de montrer que le morphisme $S'\to S$ est de descente effectif. \qedhere

\end{proof}

Cette démonstration s'adapte immédiatement dans la situation ou l'on cherche à descendre un module cohérent au dessus de $\mathcal{M(B)}$ muni de données de descente, grâce à la proposition $\ref{equiv}$ qui permet de se ramener au cas strict non trivialement valué, et à la proposition 3.7.2.6 de \autocite{bosch1984non} qui permet immédiatement de se ramener à la situation schématique. On utilise ensuite l'équivalence de catégories entre $\mathcal{A}$-modules finis et $\mathcal{A}$-modules de Banach finis. En bref, on a démontré : 

\begin{prop} Considérons le pseudo-foncteur au dessus de la catégorie des espaces $k$-affinoïdes qui à un espace $k$-affinoïde $S=\mathcal{M(A)}$ associe la catégorie des modules cohérents au dessus de $\mathcal{M(A)}$. Alors tout morphisme fini, plat et surjectif entre espace affinoïde $S'\to S$ est un morphisme de descente effective pour ce pseudo-foncteur.
\end{prop}

Maintenant, puisqu'on peut toujours recoller des modules cohérents selon la $G$-topologie par $\ref{base}$, par la proposition précédente et la proposition $\ref{oridu}$ on en déduit que l'on est en situation d'appliquer le théorème $\ref{fibfi}$ pour obtenir : 

\begin{prop} Considérons le pseudo-foncteur au dessus de la catégorie des espaces $k$-affinoïdes qui à un espace $k$-affinoïde $S=\mathcal{M(A)}$ associe la catégorie des modules cohérents au dessus de $\mathcal{M(A)}$. Alors tout morphisme fidèlement plat entre espaces affinoïdes est un morphisme de descente effective pour ce pseudo-foncteur.
\end{prop}
%Par analogie avec la théorie des schémas, on introduit la définition suivante : 
%\begin{defi} Un morphisme d'espaces $k$-analytiques $f:Y\to X$ est proprement surjectif s'il existe un $G$-recouvrement $X=\bigcup_{i\in I} X_i$ par des domaines analytiques quasi-compacts $X_i$ qui soient chacun l'image d'un domaine analytique quasi-compact $Y_i\subset Y$.
%\end{defi}
%\begin{rema}
%Un morphisme plat, surjectif et sans bord $f:Y\to X$ est proprement surjectif. En effet, si l'on se donne $Y=\bigcup_{i\in I} Y_i$ un $G$-recouvrement de $Y$ par des domaines affinoïdes, alors par \autocite{ducros2017families}, 9.2.3, le morphisme $f$ est ouvert, et $X=\bigcup_{i\in I} f(Y_i)$ est un $G$-recouvrement de $X$ par des domaines analytiques quasi-compacts. 
%\end{rema}
\begin{rema} En globalisant, on obtient l'énoncé suivant, qui concerne les morphismes plats proprement surjectifs. Cette hypothèse couvre en particulier : 
\begin{enumerate}
    \item Les morphismes plats, surjectifs et topologiquement propres, et en particulier les morphismes plats et surjectifs entre espaces compacts.
    \item Les morphismes plats, surjectifs et sans bord. En effet, si $f:Y\to X$ est plat, surjectif et sans bord, alors par \autocite{ducros2017families}, 9.2.3, le morphisme $f$ est ouvert, et si l'on se donne $Y=\bigcup_{i\in I} Y_i$ un $G$-recouvrement de $Y$ par des domaines affinoïdes alors $X=\bigcup_{i\in I} f(Y_i)$ est un $G$-recouvrement de $X$ par des domaines analytiques quasi-compacts.
\end{enumerate}
Précisons que le théorème $\ref{pf}$ sera aussi obtenu sous ces mêmes hypothèses.
\end{rema}

\begin{theo} \label{Coh} Considérons le pseudo-foncteur au dessus de la catégorie des espaces $k$-analytiques qui à un espace $k$-analytique $S$ associe la catégorie des modules cohérents au dessus de $S$. Alors tout morphisme plat et proprement surjectif $p:S'\to S$ est un morphisme de descente effective . 
\end{theo}
\begin{proof}[Démonstration]
On se place d'abord dans le cas particulier où le morphisme $p$ est topologiquement propre. L'assertion étant $G$-locale sur $S$, on peut donc supposer $S$ affinoïde et $S'$ quasi-compact, $G$-recouvert par un nombre fini de domaines affinoïde $S'=\bigcup S'_i$. On a un diagramme :

$$\xymatrix{
    S'  \ar[r]  & S  \\
    \amalg_{i\in I} S'_{i} \ar[ru] \ar[u] & 
  }$$ 

Maintenant, la flèche $\amalg_{i\in I} S'_i\to S$ est plate surjective entre espaces $k$-affinoïdes donc de descente effective par la proposition précédente. Maintenant, puisque $\amalg_{i\in I} S'_i\to S'$ est un $G$-recouvrement fini, il est de descente effectif, donc par le lemme $\ref{my}$, le morphisme $p:S'\to S$ est de descente effectif.

Maintenant, si le morphisme $p$ est juste supposé plat et proprement surjectif, alors par hypothèse il existe un $G$-recouvrement $S=\bigcup_{i\in I} S_i$ de $S$ par des domaines analytiques quasi-compacts et des domaines analytiques $S'_i\subset S'$ quasi-compacts tels que $p(S'_i)=S_i$. On se donne un $G$-recouvrement de $p^{-1}(S_i)$ par des domaines affinoïdes $R'_{ij}\subset S'$. Alors on dispose du diagramme commutatif suivant : 

$$\xymatrix{
    \coprod_{ij} (S'_i\amalg R'_{ij}) \ar[d]  \ar[r]  & \coprod_{ij} S_i \ar[d]  \\
    S' \ar[r]  & S
  }$$ 
  
Maintenant, par le lemme $\ref{my}$, puisque les deux flèches verticales sont de descente effective universellement en tant que $G$-recouvrement, pour montrer que $p$ est de descente effective, il suffit de monter que la composée $\coprod_{ij} (S'_i\amalg R'_{ij})\to S$ est de descente effective, puis de monter que la flèche $\coprod_{ij} (S'_i\amalg R'_{ij})\to \coprod S_i$ est de descente effective universellement par une nouvelle application de ce lemme. La flèche $S'_i\amalg R'_{ij}\to S_i$ étant plate, surjective et topologiquement propre, et ces faits étant invariant par changement de base, par le paragraphe précédent, elle est de descente effective universellement, et donc $p$ est aussi de descente effective, ce qui permet de conclure.
\end{proof}

On peut aussi appliquer le théorème $\ref{fibfibis}$ grâce aux proposition $\ref{hyp1}$, $\ref{hyp2}$ et $\ref{hyp3}$ pour obtenir l'énoncé suivant :  
\begin{prop}  \label{gen}
Considérons le pseudo-foncteur $\Phi$ au dessus de la catégorie des espaces $k$-affinoïdes qui à un espace affinoïde $S=\mathcal{M(A)}$ associe la catégorie des espaces $k$-affinoïdes au dessus de $S$. Alors tout morphisme fidèlement plat entre espaces $k$-affinoïdes est un morphisme de descente pour $\Phi$, c'est à dire que tout morphisme fidèlement plat entre espaces $k$-affinoïde $p:S'\to S$ induit un foncteur tiré en arrière qui est pleinement fidèle de la catégorie des espaces $k$-affinoïdes au dessus de $S$ vers la catégorie des espaces $k$-affinoïdes au dessus de $S'$ munis de données de descente relativement au morphisme $p$.
\end{prop}

On rappelle qu'en gardant les notations de la proposition précédente, il est équivalent d'être un morphisme de descente pour le pseudo-foncteur $\Phi$ et d'être un épimorphisme effectif universel dans la catégorie des espaces affinoïdes.
On va maintenant globaliser l'énoncé précédent grâce aux deux lemmes suivants :

\begin{lemm} \label{autore} Considérons un morphisme $S'\to S$ entre deux espaces affinoïdes qui est un épimorphisme effectif universel dans la catégorie des espaces $k$-affinoïdes. Alors c'est un épimorphisme effectif dans la catégorie des espaces $k$-analytiques.
\end{lemm}
\begin{proof}[Démonstration]
Considérons un épimorphisme effectif universel $p:S'=\mathcal{M(B)\to S=M(A)}$ dans la catégorie des espace affinoïdes. On se donne un morphisme $g:\mathcal{M(B)}\to X$ vers un espace analytique $X$ vérifiant $g\circ p_1=g\circ p_2$, avec $p_i$ les projections canoniques $p_i:\mathcal{M(B\hat{\otimes}_A B})\to \mathcal{M(B)} $.  Alors on veut montrer qu'il existe un unique morphisme $f:\mathcal{M(A)}\to X$ vérifiant $f\circ p=g$.

Pour cela, on peut supposer que l'espace $X$ est quasi-compact. En effet, si ce n'est pas le cas, pour tout pour $s'\in S'$, on se donne $(X_{ij})_{j\in J_i}$ des domaines affinoïdes en nombre fini dans $X$ tel que $g(s')$ soit dans chacun des $X_{ij}$ et que l'union  $\bigcup_{j\in J}X_{ij}$ contienne un voisinage $O_i$ de $g(s')$. Alors l'union des $O_i$ est un recouvrement ouvert du compact $g(S')$. On en extrait un sous-recouvrement fini. Alors il existe un nombre fini de $X_{ij}$ qui recouvre l'image $g(S')$. Notons $X'$ cette union finie. C'est une partie quasi-compacte de $X$ qui est un domaine analytique de $X$, $G$-recouvert par les $X_{ij}$ en nombre fini, et le morphisme $g:S'\to X$ se factorise par $g':S'\to X'$, et le morphisme $g'$ commute encore aux projections. On s'est donc ramené au cas où l'espace $X$ est quasi-compact car si l'on a l'existence et l'unicité du morphisme de $S$ vers $X'$ pour tout $X'$ quasi-compact et commutant aux projections, on a l'existence et l'unicité du morphisme de $S$ vers $X$ car celui-ci se factorise par un domaine analytique quasi-compact.

Maintenant, on se donne un $G$-recouvrement $X_i$ de $X$ par un nombre fini de domaines affinoïdes, et l'on note $j:\amalg_i X_i\to X$ le morphisme induit.
Alors $p(g^{-1}(X_i))$ est l'image par un morphisme plat surjectif d'un domaine analytique compact, c'est donc un domaine analytique compact de $S$ par la proposition 9.2.1 de \autocite{ducros2017families}. On se donne donc un $G$-recouvrement fini de ce domaine analytique $p(g^{-1}(X_i))=\cup_w S_{iw}$. Notons $S'_{iw}=p^{-1}(S_{iw})$. Alors on dispose d'un épimorphisme effectif universel dans la catégorie des espaces $k$-affinoïdes $p':\coprod_{iw}S'_{iw}\to \coprod_{iw} S_{iw}$, ainsi que d'un morphisme $g':\coprod_{iw} S'_{iw}\to\coprod_i X_i$ induit par $g$ (si on se donne $x\in S'_{iw}$, il existe $s'\in g^{-1}(X_i)$ tel que $p(x)=p(s')$, mais alors le couple $(x,s')$ nous fournit un point de $S'\times_S S'$, et puisque $g$ fait commuter le diagramme idoine, on obtient $g(s')=g(x)\in X_i$), à destination d'un espace affinoïde.  On résume la situation par le diagramme qui suit.

$$\xymatrix{
     & & \amalg_{iwj\tau} S'_{iw}\cap S'_{j\tau} \ar[rr]^{p^{\prime \prime}}\ar@<-.5ex>[d] \ar@<.5ex>[d] & & \amalg_{iwj\tau} S_{iw}\cap S_{j\tau}\ar@<-.5ex>[d]|-{r_1} \ar@<.5ex>[d]|-{r_2} \\
    \amalg_{iw}{S'}_{iw}\times_{S_{iw}} {S'}_{iw}\ar[dd] \ar@<-.5ex>[rr] \ar@<.5ex>[rr]  &  & \amalg_{iw}{S'}_{iw} \ar[dd]^{r'}\ar[rr]^{p'}\ar[rd]^{g'} & &  \amalg_{iw}{S}_{iw}\ar[dd]^r  \ar@{.>}[dl]|-{f'}\\
     & & &  \amalg_i X_i\ar[ldd]_j & \\
     S''  \ar@<-.5ex>[rr] \ar@<.5ex>[rr] & & S' \ar[d] \ar[rr]^p& & S  \ar@{.>}[dll]|-{f} \\
      & &  X & & 
  }$$

On vérifie que le morphisme $g'$ commute aux deux projections $p'_i:\coprod_{iw} S'_{iw}\times_{S_{iw}} S'_{iw}\to \coprod_{iw} S'_{iw}$ (en vérifiant que ça commute sur chacun des ouverts fermés $S'_{iw}\times_{S_{iw}} S'_{iw}$ de $\coprod S'_{iw}\times_{S_{iw}} S'_{iw}=\coprod S'_{iw}\times_{\coprod S_{iw}} \coprod S'_{iw}$) et puisque $p'$ est un épimorphisme effectif universel dans la catégorie des espaces $k$-affinoïdes, on obtient l'existence d'un unique morphisme $f':\coprod_{iw} S_{iw}\to \coprod_i X_i$ qui vérifie la relation $f'\circ p'=g'$. Maintenant, on vérifie que le morphisme obtenu par composition $j\circ f':\coprod S_{iw}\to X$ provient d'un morphisme de $S\to X$. Il faut pour cela remarquer que le morphisme $p^{\prime \prime}:\coprod_{iw j\tau} S'_{iw}\cap S'_{j\tau}\to \coprod_{iw j\tau} S_{iw}\cap S_{j\tau}$ est encore un épimorphisme effectif universel dans la catégorie des espaces $k$-affinoïdes comme changement de base d'un tel morphisme donc c'est un épimorphisme, et vérifier que si $r:\amalg_{iw} S_{iw}\to S$ désigne le $G$-recouvrement, et $r_i:\coprod_{iw j\tau} S_{iw}\cap S_{j\tau} \to \amalg_{iw}S_{iw}$ chacune des deux projections canoniques, alors on a ${j\circ f' \circ r_1 \circ p^{\prime \prime }=j\circ f' \circ r_2 \circ p^{\prime \prime}}$, donc $j\circ f' \circ r_1=j\circ f' \circ r_2$, ce qui fournit alors l'existence d'un unique morphisme $f:S\to X$ vérifiant $f\circ r=j\circ f'$. Maintenant, la flèche $r':\amalg_{iw} S'_{iw}\to S'$ est un épimorphisme d'espaces $k$-analytiques, et l'on a l'égalité $f\circ p \circ r'=f\circ r \circ p'=j\circ f'\circ p'=j\circ g'=g\circ r'$, ce qui montre que $f$ vérifie bien la relation $f\circ p=g$, et $p$ est bien un épimorphisme effectif dans la catégorie des espaces $k$-analytiques, ce qui conclut la démonstration. \qedhere

\end{proof}

\begin{lemm} \label{protop} Soit $P$ une propriété de morphismes d'espaces analytiques (eg. plat) stable par changement de base et stable par restriction à un domaine affinoïde à droite et à gauche (si on se donne $f:Y\to X$ vérifiant $P$, et un domaine affinoïde $V\subset X$ resp. $U\subset Y$ vérifiant $f(U)\subset V$, alors le morphisme induit $f:U\to V$ possède $P$). On suppose que les morphismes entre espaces affinoïdes vérifiant P et surjectifs sont des épimorphismes effectifs dans la catégorie des espaces analytiques. Alors tout morphisme entre espaces analytiques vérifiant P, surjectif et propre topologiquement est un épimorphisme effectif dans la catégorie des espaces $k$-analytiques.
\end{lemm}
\begin{proof}[Démonstration]
Soit donc $p:S'\to S$ surjectif, propre topologiquement et vérifiant la propriété $P$. On va montrer que $p$ est un épimorphisme effectif. Soit donc $X$ un espace $k$-analytique général. Notons $F$ le foncteur représenté par $X$. On veut montrer que l'on a une suite exacte $0\to F(S)\to F(S')\rightrightarrows F(S'\times_S S')$, où dans la suite du texte la notation $0\to F(S)\to F(S')$ signifie que la flèche $F(S)\to F(S')$ est injective. 

On se donne d'abord $S=\cup_{i\in I} S_i$ un recouvrement de $S$ par des affinoïdes. On se donne un recouvrement fini $p^{-1}(S_i)=\cup_{l\in L_i} S'_{ij}$ de $p^{-1}(S_i)$ par un nombre fini d'affinoïdes. On en déduit un diagramme :

$$\xymatrix{
    S'  \ar[r]  & S  \\
    \amalg_{i\in I, l\in L_i} S'_{il} \ar[u] \ar[r] & \amalg_{i\in I} S_i \ar[u]
  }$$
  
Comme $U\mapsto F(U)$ est un faisceau pour la $G$-topologie par la proposition 1.3.2 de \autocite{PMIHES_1993__78__5_0}, on a des suites exactes : $$0\to F(S)\to \prod_{i\in I} F(S_i)\to \prod_{i,j\in I^2} F(S_i\cap S_j)$$ 
 $$0\to F(S')\to \prod_{i\in I,l\in L_i} F(S'_{il})\to \prod_{i,j\in I^2 l,k\in L_i\times L_j} F(S'_{il}\cap S'_{jk})$$

On montre d'abord l'unicité du recollement, c'est à dire l'exactitude de la partie gauche. On a un diagramme :
$$\xymatrix{& 0\ar[d] & \\
     0\ar[r]& F(S) \ar[r] \ar[d]  & F(S') \ar[d] \\
   0 \ar[r]& \prod_{i\in I} F(S_i)  \ar[r] & \prod_{i\in I,l\in L_i} F(S'_{il} )
  }$$
Par le cas affinoïde, la flèche horizontale en bas est injective. Par le paragraphe précédent, la flèche verticale gauche est injective. Par chasse au diagramme, on en déduit que la première ligne est injective.

Pour l'exactitude au milieu, on va utiliser le cas affinoïde, l'unicité démontrée pour un morphisme surjectif vérifiant $P$, et le recollement selon des domaines analytiques. 

On le diagramme suivant : 
$$\xymatrix{
     &0\ar[d] & 0\ar[d]&  \\
     0\ar[r]&F(S)\ar[r]\ar[d] & F(S')\ar@<-.5ex>[r] \ar@<.5ex>[r] \ar[d] & F(S'\times_S S')\ar[d] \\
    0 \ar[r]&  \prod_{i\in I} F(S_i) \ar[r]\ar@<-.5ex>[d] \ar@<.5ex>[d]&\prod_{i\in I}( F(\amalg_{l\in L_i}S'_{il}))\ar@<-.5ex>[r] \ar@<.5ex>[r] \ar@<-.5ex>[d] \ar@<.5ex>[d]& \prod_{i\in I}(F(\amalg_{l\in L_i}S'_{il}\times_{S_i}\amalg_{k\in L_i}S'_{ik}))) \\
    0\ar[r]&\prod_{i,j\in I^2} F(S_i\cap S_j)\ar[r]  & \prod_{i,j\in I^2}\prod_{l,k\in L_j\times L_k}F(S'_{il}\cap S'_{jk})
  }$$
  
On veut montrer que la première ligne est exacte. Les deux colonnes du milieu sont exactes parce que $U\mapsto F(U)$ est un faisceau pour la $G$-topologie. La deuxième ligne est exacte par le cas affinoïde, puisque les $L_i$ sont en nombre finis. La dernière ligne est exacte par ce qui précède puisque le morphisme $h:\amalg_{i,j\in I^2}\amalg_{k,l\in L_i\times L_j}S'_{il}\cap S'_{jk}\to \amalg_{i,j\in I^2}S_i\cap S_j$ est topologiquement propre, vérifie $P$ et est surjectif.

Par chasse au diagramme, la première ligne est exacte, ce qui achève la preuve.\end{proof}

On peut maintenant énoncer et  démontrer la globalisation de $\ref{gen}$. Remarquons que ce théorème couvre le cas des morphismes plats, surjectifs et topologiquement propres, ainsi que le cas des morphismes plats, surjectifs et sans bord.

\begin{theo} \label{pf} \label{final} Considérons le pseudo-foncteur $\Psi$ qui à un espace analytique $S$ associe la catégorie des espaces $k$-analytiques au dessus de $S$. Alors les morphismes plats et proprement surjectifs sont des morphismes de descente.

Autrement dit, les morphismes plats et proprement surjectifs sont des épimorphismes effectifs universels de la catégorie des espaces analytiques ; autrement dit, si l'on se donne un morphisme plat et proprement surjectif $p:S'\to S$ entre espaces analytiques, alors le foncteur de la catégorie des $S$-espaces analytiques vers la catégorie des $S'$-espaces analytiques munis de données de descente relativement au morphisme $p$ est pleinement fidèle.
\end{theo}

\begin{proof}[Preuve] Par $\ref{gen}$, les morphismes plats surjectifs entre espaces affinoïdes sont des épimorphismes effectifs universels dans la catégorie des espaces affinoïdes. Par $\ref{autore}$, ce sont des épimorphismes effectifs dans la catégorie des espaces analytiques, et on peut appliquer $\ref{protop}$ pour obtenir que les morphismes plats, propres topologiquement et surjectifs sont des épimorphismes effectifs, et même effectifs universels car les hypothèses sont stables par changement de base. De plus, les épimorphismes effectifs universels de la catégorie des espaces $k$-analytiques sont exactements les morphismes de descente pour le pseudo-foncteur $\Psi$ par les résultats de \autocite{grothendieck1963schemas}, donc les morphismes plats, surjectifs et topologiquement propres sont des morphismes de descente universellement pour $\Psi$ puisque ces hypothèses sont stables par changement de base.

Maintenant, si l'on se donne $p:S'\to S$ un morphisme plat et proprement surjectif, alors on procède comme en $\ref{Coh}$ : il existe un $G$-recouvrement $S=\bigcup_{i\in I} S_i$ de $S$ par des domaines analytiques quasi-compacts $S_i$ qui soient chacun l'image par $p$ de domaines analytiques quasi-compacts $S'_i\subset S'$. On se donne maintenant un $G$-recouvrement de $p^{-1}(S_i)$ par des domaines affinoïdes $R'_{ij}\subset S'$. Alors on dispose du diagramme commutatif suivant : 

$$\xymatrix{
    \coprod_{ij} (S'_i\amalg R'_{ij}) \ar[d]  \ar[r]  & \coprod_{ij} S_i \ar[d]  \\
    S' \ar[r]  & S
  }$$ 

Les deux flèches verticales sont de descente universellement parce que ce sont des $G$-recouvrement surjectifs, et puisque chacune des flèche $(S'_i\amalg R'_{ij})\to S_i$ est plate, surjective et topologiquement propre, la flèche $\coprod_{ij} (S'_i\amalg R'_{ij})\to \coprod_{ij}S_i$ est de descente universellement par le paragraphe précedent, donc la composée $\coprod_{ij} (S'_i\amalg R'_{ij})\to S$ est un morphisme de descente par le lemme $\ref{mybis}$, et par une seconde application du lemme $\ref{mybis}$, on voit que le morphisme $p:S'\to S$ est de descente pour $\Psi$, ce qui permet de conclure.
\end{proof}

\subsection{La question de l'effectivité } 
On va maintenant s'intéresser à la question de l'effectivité d'une donnée de descente fixée, pour le pseudo-foncteur au dessus de la catégorie des espaces $k$-affinoïdes qui à un espace affinoïde $S$ associe la catégorie des espaces analytiques au dessus de $S$. On verra qu'on peut toujours rendre effectif une donnée de descente affinoïde, mais que l'espace obtenu n'est pas forcément affinoïde.

La définition qui suit est motivée par le fait qu'il n'existe pas en géométrie rigide (et donc en géométrie de Berkovich) de notion satisfaisante de morphismes affinoïdes (par analogie avec les morphismes affines entre deux schémas). En effet, Liu montre dans \autocite{liu1990}, proposition 3.3 que pour tout espace de Stein rigide $X$ quasi-compact, il existe un espace affinoïde $Y$, un morphisme d'espaces rigides $f:X\to Y$ et un recouvrement affinoïde admissible $(Y_i)$ de $Y$ tel que $f^{-1}(Y_i)$ est un recouvrement affinoïde admissible de $X$, et il exhibe (théorème 4) de tels espaces de Stein quasi-compacts mais non-affinoïdes. La notion qui suit ne se teste donc pas sur un recouvrement affinoïde quelconque de la base.
\begin{defi} On dit qu'un morphisme d'espaces $k$-analytiques $p:Y\to X$ est presque affinoïde s'il existe un $G$-recouvrement $X=\cup_{i\in I} X_i$ de $X$ par des domaines affinoïdes tel que pour tout $i\in I$, l'image inverse par $p$ de $X_i$ soit un domaine affinoïde de $Y$.
\end{defi}
\begin{prop}La propriété d'être presque-affinoïde est stable par changement de base.
\end{prop}
\begin{proof}[Preuve]
Si l'on se donne $f:X\to S$ et $p:S'\to S$ deux morphismes d'espaces analytiques, avec $f:X\to S$ presque affinoïde, alors $f':X'\to S'$ le changement de base de $f$ le long de $p$ est presque affinoïde. En effet, si l'on se donne $s'\in S'$, $U$ un domaine affinoïde de $S$ contenant $s=p(s')$ tel que $f^{-1}(U)$ est encore affinoïde et $V$ un domaine affinoïde de $S'$ contenant $s'$ et vérifiant $p(V)\subset U$, alors $f'^{-1}(V)=f^{-1}(U)\times_U V$ est bien affinoïde.

Maintenant, il suffit de se donner un nombre fini  d'affinoïdes $U_i$ contenant $s$ et dont l'union contient un voisinage de $s$ dans $S$. Alors on peut trouver des affinoïdes $V_i$ qui contiennent $s'$ et dont l'union contient un voisinage de $s'$ dans $S'$. Alors l'image inverse par $f'$ des $V_i$ est affinoïde. Comme on peut faire ce raisonnement pour chaque $s'\in S'$, le morphisme $f'$ est bien presque-affinoïde.\end{proof}

La proposition suivante est une reformulation de la proposition 1.3.3 de \autocite{berkovich1993etale}.

\begin{prop} Considérons un $G$-recouvrement $p:S':=\coprod_{i\in I} S_i \to S$ d'un espace affinoïde $S$ par un nombre fini de domaines affinoïdes $S_i$, et soit $X'$ un espace affinoïde au dessus de $S'$ muni de données de descente relativement à $p$. Alors il existe un espace analytique $X$ au dessus de $S$, dont le morphisme structural est presque affinoïde, et qui induit la donnée de descente de $X'$.

\end{prop}

\begin{proof}
Notons $f':X'\to S'$ le morphisme structural. Par la proposition 1.1.3 de \autocite{berkovich1993etale}, on peut toujours recoller les espaces selon un nombre finis de domaines affinoïdes, et il existe un espace $k$-analytique $X$ qui rend effectif la donnée de descente. Par définition, cet espace est muni d'un morphisme $f:X\to S$ tel que $f^{-1}(S_i)=f'^{-1}(S_i)$, et chaque $f'^{-1}(S_i)$ étant un espace $k$-affinoïde, on en déduit que $f$ est un morphisme presque affinoïde. 
\end{proof}

Le résultat phare de cette section est que la catégorie fibrée des espaces presque affinoïdes au dessus de la catégorie des espaces $k$-affinoïdes est un champ pour la topologie dont les flèches couvrantes sont les morphismes plats et surjectifs.

\begin{defi} On dit qu'un morphisme entre espaces $k$-affinoïdes $p:S'\to S$ vérifie la propriété $\mathrm{PA}$ si pour toute donnée de descente $X'\to S'$ relativement à $p$ d'un espace affinoïde $X'$, il existe un espace analytique $X$ au dessus de $S$ dont le morphisme structural est presque-affinoïde, et qui rend effectif la donnée de descente $X'$ dans la catégorie des espaces analytiques.
\end{defi}

On énonce deux lemmes qui sont les analogues de $\ref{mybis}$ et $\ref{my}$ mais pour ce qui concerne l'effectivité de données de descente affinoïdes.

\begin{lemm}\label{redx}Considérons $R\stackrel{v}{\to} S\stackrel{u}{\to} T$ des morphismes entre espaces $k$-affinoïdes, et $\Psi$ le pseudo-foncteur qui à un espace $k$-affinoïde $S$ associe la catégorie des espaces $k$-affinoïdes au dessus de $S$. Alors : 
\begin{enumerate}
\item Supposons que le morphisme $S\to T$ vérifie la propriété $\mathrm{PA}$ et que $R\to S$ est un $G$-recouvrement par un nombre fini de domaines affinoïdes. Alors $R\to T$ vérifie la propriété $\mathrm{PA}$. 
\item Supposons que le morphisme $S\to T$ vérifie la propriété $\mathrm{PA}$, et que $R\to S$ est un morphisme de descente effective pour $\Psi$ dans la catégorie des espaces affinoïdes.  Alors $R\to T$ possède la propriété $\mathrm{PA}$.
\item Supposons que $R\to T$ possède la propriété $\mathrm{PA}$ et que $R\to S$ est un épimorphisme effectif universel dans la catégorie des espaces $k$-analytiques, et possède la propriété $\mathrm{PA}$. Alors $S\to T$ possède la propriété $\mathrm{PA}$.
\end{enumerate}
\end{lemm}
\begin{proof}[Preuve]
Les démonstrations sont semblables à \autocite{girauddescente1964} 10.10 et 10.11. On démontre par exemple la deuxième propriété. On suppose donc que $S\to T$ vérifie la propriété $\mathrm{PA}$ et que $R\to S$ est de descente effective dans la catégorie des espaces $k$-affinoïdes. Alors on se donne une donnée de descente $X'\to R$ relativement au morphisme $R\to T$, avec $X'$ un espace $k$-affinoïde. Notons $p_i:S\times _T S \to S$, $q_i:R\times_T R\to R$ et $r_i:R\times_S R\to R$ les projections canoniques. On dispose aussi de morphismes canoniques $m:R\times_T R\to S\times_T S$ et $l:R\times_S R\to R\times_T R$. Alors $X'$ vient avec un isomorphisme $\varphi:q_1^* X'\to q_2 ^*X'$ vérifiant la condition de cocycle usuelle. En tirant la donnée de descente $\varphi$ par le morphisme $l$, on en déduit une donnée de descente sur $X'$ relativement à $v:R\to S$. Puisque ce morphisme est de descente effective dans la catégorie des espaces affinoïdes, on en déduit qu'il existe un espace affinoïde $X_1$ au dessus de $S$ muni d'un isomorphisme de données de descente $\lambda':v^*X_1\to X'$ relativement à $v$.

Maintenant, comme $v$ est de descente si et seulement si c'est un épimorphisme effectif universel dans la catégorie des espaces $k$-affinoïdes, que cette hypothèse est stable par changement de base (affinoïdes) et que le morphisme $m$ est obtenu à partir de $v$ par changement de base (affinoïdes) et composition, le morphisme $m$ est de descente, et on vérifie qu'on peut utiliser cette propriété pour trouver une donnée de descente canonique $\varphi_1$ au dessus de $X_1$ relativement au morphisme $u$. Maintenant, puisque $u$ vérifie la propriété $\textrm{PA}$, on en déduit qu'il existe un espace presque affinoïde $X_0$ au dessus de $T$ muni d'un isomorphisme de données de descente $u^*X_0\to X_1$, et on vérifie enfin que $X_0$ fournit un espace presque affinoïde au dessus de $T$ qui rend effectif la donnée de descente $X'$ ce  qui montre que $R\to T$ vérifie $\textrm{PA}$.
\end{proof}
\begin{rema} En particulier, par la proposition $\ref{hyp2}$ et la deuxième partie du lemme ci-dessus, si l'on précompose un morphisme vérifiant la propriété $\mathrm{PA}$ par un morphisme fini, plat et surjectif de source affinoïde, la composée vérifie encore la propriété $\mathrm{PA}$.
\end{rema}
 
\begin{prop} Considérons un morphisme fidèlement plat entre espaces $k$-affinoïdes $p:S'\to S$, et un $S'$-espace $k$-affinoïde $X'$ qui est muni de données de descente relativement au morphisme $p$. Alors il existe un espace analytique $X$ presque affinoïde au dessus de $S$ qui rend effectif la donnée descente, c'est à dire que l'on a un isomorphisme de données de descente entre $X'$ muni de ses données de descente vers $X\times_S S'$ muni des données de descente canoniques ; autrement dit, le morphisme $p$ vérifie la propriété $\mathrm{PA}$.
\end{prop}

\begin{proof}[Démonstration]
La démonstration suit le même chemin et les mêmes réductions que la preuve du théorème $\ref{fibfi}$, mais en utilisant à la fois  le lemme $\ref{redx}$ ainsi que la remarque qui suit. Toutes les réductions utilisent des factorisations par des morphismes plats surjectifs, donc par le théorème $\ref{pf}$, ce sont des morphismes de descente, et l'application du lemme est licite. Nous utiliserons donc les notations de la preuve du théorème $\ref{fibfi}$.

On se donne donc un morphisme fidèlement plat $S'\to S$ entre espaces $k$-affinoïdes. Si $S$ et $S'$ sont strictement affinoïdes et que le corps $k$ est non trivialement valué, on pose $S_r=S$ et $S'_r=S'$. Sinon, soit $r$ un polyrayon $k$-libre non trivial tel que $S'_r\to S_r$ est un morphisme entre espaces strictement $k_r$-affinoïdes. Par le théorème de multisection de Ducros, il existe un espace strictement $k_r$-affinoïde $H$, et un morphisme quasi-fini, plat et surjectif de $H\to S_r$ tel que le changement de base $H'=H\times_{S_r}S'_r\to H$ possède une section. Maintenant, puisque $H\to S_r$ est un morphisme quasi-fini, plat et surjectif, par le théorème 8.4.6 de \autocite{ducros2017families} (que l'on est en mesure d'appliquer grâce à la remarque 8.4.3 de \autocite{ducros2017families}), il existe un $G$-recouvrement fini de $H$ par des affinoïdes $V'_i$, des espaces affinoïdes $W'_i$, et des morphismes $\amalg_{i\in I} V'_i \to \amalg_{i\in I}W'_i$  et $R:=\amalg_{i\in I}W'_i\to S_r$ tels que le morphisme $\amalg_{i\in I} V'_i\to R$ soit fini, plat et surjectif, tel que le morphisme $R=\amalg_{i\in I} W'_i\to S_r$ soit quasi-étale et surjectif et tel que la composée $V'_i\to W'_i\to S'_r$ soit simplement la restriction de la flèche initiale $H\to S_r$. Maintenant, en utilisant les mêmes factorisation et notations que dans la preuve $\ref{fibfi}$ (sauf pour le $Y\to X$ du cas quasi-étale, qui ici est remplacé par le morphisme $R=\amalg_{i\in I} W'_i\to S_r$, et les $V$, resp $T$, resp $Y_i$ du cas quasi-étale, qui sont remplacés par $V''_x$, resp $T_x$, resp $Y_{jw}$), il existe un $G$-recouvrement fini de $S_r$ par des domaines affinoïdes $S_r=\cup_{x\in E} V''_x$, pour chaque $x\in E$, il existe un $V''_x$-espace $T_x$ dont le morphisme structural est fini, plat et surjectif, et il existe des $R$-espaces affinoïdes $\amalg_{x\in E, j\in J_x} Y_{jx}$ tel que la flèche $\amalg_{x\in E} \amalg_{j\in J_x} Y_{jx}\times_{S_r} V''_x\times_{V''_x} T_x\to \amalg_{x\in E} T_x$ ainsi que le morphisme $\amalg_{x\in E} \amalg_{j\in J_x} Y_{jx}\times_{S_r} V''_x\times_{V''_x} T_x\to \amalg_{x\in E}V''_x\times_{S_r}R\times_{V''_x}T_x$ soient des $G$-recouvrement finis et surjectifs.

On résume la factorisation introduite ici, et détaillée dans la démonstration de $\ref{fibfi}$ par le diagramme commutatif suivant :

$$\xymatrix{
    H' \ar[d] \ar[r]  & H \ar[d]  & \amalg_{i\in I} V'_i \ar[d] \ar[l] & & & \\
   S'_r  \ar[r] \ar[d] & S_r \ar[d] & R\ar[l] & \amalg_{x\in E}V''_x\times_{S_r} R \ar[ld]\ar[l]& \amalg_{x\in E}R\times_{S_r} T_x \ar[l] \ar[dl] & \amalg_{x\in E,j\in J_x}  Y_{jx}\times_{S_r} V''_x\times_{V''_x} T_x \ar[l]\ar[lld]\\
   S' \ar[r] & S &  \amalg_{x\in E} V''_x \ar[lu] & \amalg_{x\in E} T_x\ar[l] &  &
  }$$

Et les flèches du diagramme sont toutes de sources et but affinoïdes comme suit :  
\begin{enumerate}
\item $S'\to S$ est plate surjective entre espaces affinoïdes.
\item $S'_r\to S_r$ est plate surjective entre espaces strictement affinoïdes sur $k_r$.
\item $H'\to H$ est fidèlement plat entre espaces strictement affinoïdes et possède une section, et $H'\to S_r$ est quasi-fini, plat et surjectif.
\item $\amalg_{i\in I} V'_i\to H$ est un $G$-recouvrement fini, $\amalg_{i\in I} V'_i\to R$ est fini, plat et surjectif, et $R\to S_r$ est quasi-étale surjectif.
\item Les morphismes $\amalg_{x\in E,j\in J_x}  Y_{jx}\times_{S_r} V''_x\times_{V''_x} T_x \to \amalg_{x\in E}R\times_{S_r} T_x$ ainsi que $\amalg_{x\in E,j\in J_x}  Y_{jx}\times_{S_r} V''_x\times_{V''_x} T_x\to \amalg_{x\in E}T_x$, $\amalg_{x\in E} V''_x\times_{S_r}R\to R$ et $\amalg_{x\in E} V''_x\to S_r$ sont des $G$-recouvrements surjectifs par un nombre fini de domaines affinoïdes.
\item Les morphismes $\amalg_{x\in E} T_x\to \amalg_{x\in E}V''_x$ ainsi que $\amalg_{x\in E}R\times_{S_r} T_x\to \amalg_{x\in E}V''_x\times_{S_r} R$ sont finis, plat et surjectifs.
\end{enumerate}

A partir de là, on va appliquer plusieurs fois le lemme $\ref{redx}$ et la remarque qui suit immédiatement après : on applique d'abord les deux première parties du lemme précédent pour avoir que $\amalg_{x\in E,j\in J_x}  Y_{jx}\times_{S_r} V''_x\times_{V''_x} T_x \to S$ vérifie $\mathrm{PA}$. On applique ensuite la troisième partie du lemme plusieurs fois pour avoir que $R\to S$ vérifie $\mathrm{PA}$. On applique la deuxième partie du lemme pour avoir que $\amalg_{i\in I}V'_i\to S$ vérifie la propriété PA. On applique la troisième partie pour avoir que $H\to S$ vérifie la propriété PA. La deuxième partie pour avoir que $H'\to S$ vérifie la propriété PA et enfin deux fois la troisième partie du lemme pour avoir que $S'_r\to S$ puis $S'\to S$ possède la propriété PA, ce qui achève la preuve de la proposition. \qedhere

\end{proof}

Grâce à la proposition d'effectivité précédente, on en déduit le théorème suivant : 

\begin{theo} \label{cool} Considérons le pseudo-foncteur $\Psi$ qui à un espace $k$-affinoïde $S$ associe la catégorie des espaces $k$-analytiques au dessus de $S$ dont le morphisme structural est presque affinoïde. Alors les morphismes fidèlement plats sont des morphismes de descente effectif pour ce pseudo-foncteur. Autrement dit, le pseudo-foncteur qui à un espace $k$-affinoïde $S$ associe $\Psi S$ est un champ pour la topologie dont les flèches couvrantes sont plates et surjectives.
\end{theo}

\begin{proof}[Démonstration]
Par le théorème $\ref{final}$, c'est un morphisme de descente. Reste à montrer l'effectivité d'une donnée de descente. 

Considérons donc un morphisme $p:S'\to S$ fidèlement plat entre espaces $k$-affinoïdes et un $S'$-espace analytique $X'$ muni de données de descente $\varphi:p_1^*X'\to p_2 ^* X'$ avec $p_i:S'':=S'\times_S S'\to S'$ les deux projections canoniques et dont le morphisme structural $f:X'\to S'$ est presque affinoïde. On souhaite montrer que la donnée de descente est effective.
Soient donc $S'_i$ des domaines affinoïdes en nombre fini qui recouvrent $S'$ et tels que $f^{-1}(S'_i)$ est affinoïde. Soit $X''=X'\times_{S'}\amalg_{i\in I}S'_i$. Par définition c'est un espace affinoïde, et en tirant en arrière la donnée de descente $\varphi$ par le morphisme $k:\amalg_{i\in I} S'_i\times_S' \amalg_{i\in I} S'_i$, on voit que $X''$ est naturellement fourni avec ${k^*\varphi:q_1^* X'\to q_2 ^* X'}$, donnée de recollement relativement au morphisme composé $\amalg_{i\in I} S'_i\to S$, et où l'on a noté $q_i:{\amalg_{i\in I} S'_i}\times_S \amalg_{i\in I} S'_i\to \amalg_{i\in I} S'_i$ les deux projections canoniques. Maintenant, avec des notations évidentes, si l'on tire la relation de cocycle $p_{13}^*\varphi=p_{23}^* \varphi\circ p_{12}^* \varphi$ par le morphisme canonique $\amalg_{i\in I} S'_i\times_S \amalg_{i\in I} S'_i\times_S \amalg_{i\in I} S'_i\to S'\times_S S'\times_S S'$, grâce au théorème $\ref{final}$, on voit que $k^*\varphi$ vérifie bien la condition de cocycle, et donc que c'est bien une donnée de descente pour $X''$ relativement à $\amalg_{i\in I} S'_i\to S$.

Puisque $h:\amalg_{i\in I} S'_i\to S$ est un morphisme fidèlement plat, par la proposition précédente sur l'effectivité, cette donnée de descente est effective, et l'on obtient un $S$-espace analytique $X$ dont le morphisme structural est presque affinoïde, et tel que l'on ait un isomorphisme de données de descente $h^*X\to X''$. Maintenant, en utilisant que $\amalg_{i\in I} S'_i\to S$ ainsi que ${\amalg_{i\in I} S'_i\times_S \amalg_{i\in I} S'_i\to S'\times_S S'}$ sont des morphismes de descente, on voit que l'on peut descendre l'isomorphisme de données	 de descente $h^*X\to X''$ en un isomorphisme de données de descente de $p^*X$ vers $X'$, et cela achève la preuve de l'effectivité.
\end{proof}

On donne un dernier critère d'effectivité analogue au cas des schémas pour une donnée de descente générale pour les $S$-objets.
\begin{defi} On se donne un morphisme entre espaces $k$-affinoïdes $p:S'\to S$, et un $S'$-espace analytique $X'$ muni de données de descente, c'est à dire d'un isomorphisme de $S''$-objets $\varphi:X'\times_S S'\to S'\times_S X'$. Un domaine affinoïde $X'_i$ de $X'$ est $\varphi$-stable si $\varphi$ se restreint en un isomorphisme $X_i '\times_S S'\to S'\times_S X_i '$.
\end{defi}
\begin{rema} \label{equiv2} On peut reformuler la définition précédente en introduisant $q_1:X'\times_S S'\to X'$ la première projection et $q_2:X'\times_S S'\to X'$ la composée de $\varphi$ avec la seconde projection. Alors un domaine affinoïde $X'_i$ de $X'$ est $\varphi$-stable si l'on a l'égalité $q_2^{-1}(X'_i)=q_1^{-1}(X'_i)$ ou bien $q_2(q_1^{-1}(X_i'))= X'_i$.
\end{rema}

\begin{lemm} \label{submersive}Considérons un morphisme $f:Y\to X$ surjectif et topologiquement propre entre espaces analytiques. Alors $f$ est submersive, c'est à dire que la topologie sur $X$ est la topologie quotient relativement à $f$ : une partie $E\subset X$ est ouverte (resp. fermée) si et seulement si $f^{-1}(E)$ l'est.

\end{lemm}

\begin{proof}[Preuve]
 Par la proposition 1.1.1 de \autocite{berkovich1993etale}, l'assertion est $G$-locale sur la base, on peut donc supposer $X$ affinoïde, et $Y$ est alors quasi-compact, donc si l'on se donne $E$ dont l'image inverse est fermée, alors $f^{-1}(E)$ est quasi-compact donc par surjectivité $E$ est quasi-compact donc fermé puisque $X$ est séparé.\end{proof}
\begin{lemm} \label{partie} Considérons un morphisme fidèlement plat (resp. fidèlement plat et topologiquement propre) entre espaces analytiques  $p:S'\to S$. Notons $p_i:S'\times_S S'\to S'$ les deux projections canoniques. On se donne un domaine analytique compact $W'\subset S'$ (resp. un ouvert) vérifiant la relation ${p_1^{-1}(W')=p_2^{-1}(W')}$. 

Alors, il existe un domaine analytique compact (resp. un ouvert) $W\subset S$ tel que l'on ait $W'=p^{-1}(W)$.
\end{lemm}

\begin{proof}[Preuve]

On a l'égalité $p^{-1}(p(W'))=W'$ par surjectivité du produit fibré analytique dans le produit fibré topologique et l'hypothèse vérifiée par $W'$.
On pose donc $W=p(W')$ qui est un domaine analytique compact de $S$ par le theorème 9.2.1 de \autocite{ducros2017families} dans le cas où $W'$ est un domaine analytique compact.

Dans le cas où $W'$ est un ouvert et $p$ est de plus supposé topologiquement propre, $p$ est une application submersive par le lemme précédent, et l'égalité précédente montre que $W$ est ouvert.
\end{proof}

\begin{prop} Considérons un morphisme fidèlement plat $p:S'\to S$ entre espaces $k$-analytiques.
On suppose que $(X',\varphi)$ est un espace analytique Hausdorff muni de données de descente et d'un $G$-recouvrement fini par des domaines affinoïdes $X'=\cup X_i'$ tel que chaque $X_i'$ est stable pour $\varphi$. 

Alors la donnée de descente $(X',\varphi)$ est effective dans la catégorie des espaces analytiques si et seulement si chaque donnée de descente induite $X_i '$ est effective.
\end{prop}

%Hausdorff pour avoir l'intersection compacte fini pour pouvoir recoller

\begin{proof}[Preuve]

Le sens direct est évident. On montre le sens réciproque.
Par hypothèse, il existe des $S$-espaces $X_i$ et des isomorphismes de données de descente $\lambda_i:X_i\times_S S'\to X_i'$.

On pose maintenant $W'_{ij}=\lambda_i^{-1}(X'_{ij})$ avec $X'_{ij}=X'_i\cap X'_j$. Alors puisque $\lambda_i$ est un morphisme de données de descente et fait commuter le diagramme idoine, on en déduit que $W'_{ij}$ vérifie les hypothèse du lemme précédent, donc il existe un domaine analytique $X_{ij}$ de $X_i$ dont le changement de base par $p$ est égal à $W'_{ij}$. 

Maintenant, la collection $(X_i',X'_{ij})$ se recolle en un espace $X'$, on en déduit par pleine fidélité du foncteur tiré en arrière que la collection $(X_i,X_{ij})$ se recolle en un espace $X\to S$, et que les isomorphismes $\lambda_i$ se recollent en un isomorphisme de données de descente $\lambda:X\times_S S'\to X'$. \qedhere 

\end{proof}

On donne une dernière proposition d'effectivité, qui possède son analogue en corollaire 7.3 de l'exposé sur la descente  \autocite{grothendieck2002rev} et se démontre exactement de la même manière.
\begin{prop} \label{recouv} Considérons un morphisme $p:S'\to S$ entre espaces analytiques, et $(S_i)$ un $G$-recouvrement de $S$ par un nombre fini de domaines affinoïdes. Considérons aussi $X'$ un $S'$-espace analytique compact munis de données de descente relativement à $p$.
Notons $S_i'$ et $X_i'$ les espaces déduits de $S'$ et $X'$ par le changement de base $S_i\to S$. 

Alors la donnée de descente sur $X'$ est effective si et seulement si pour tout $i$, la donnée de descente sur $X_i '$ relativement à $S_i '\to S_i$ est effective.
\end{prop}
%Hausdorff pour avoir l'intersection compacte fiin pour ouvoir recoller

Comme dans l'exposé sur la descente fidèlemente plate de \autocite{grothendieck2002rev}, on peut donner un théorème d'effectivité pour une donnée de descente dont le morphisme sous-jacent est universellement injectif, et la démonstration est exactement la même que dans le cas des schémas : on vérifie que tout domaine affinoïde est stable pour la donnée de descente. Avant, on donne quelques propriétés des morphismes universellement injectifs : 

\begin{prop}
Considérons un morphisme d'espaces $k$-analytiques $f:Y\to X$. Alors les propriétés suivantes sont équivalentes : 
\begin{enumerate}
\item le morphisme $f$ est universellement injectif, c'est à dire que le morphisme obtenu à partir de $f$ par extension du corps de base, ou bien par changement de base par un espace $k$-analytique est injectif.
\item pour toute extension de corps valué complet $k\hookrightarrow L$, l'application induite $\Hom(\mathcal{M}(L),Y)\to \Hom(\mathcal{M}(L),X)$ sur les morphismes de $k$-espaces analytiques est injective,
%\item le morphisme diagonal $Y\to Y \times_X Y$ est surjectif.
\end{enumerate}
\end{prop}

\begin{proof} Supposons que le morphisme $f$ est universellement injectif. Soit $L$ une extension complète de $k$. La donnée d'un élément de $\Hom(\mathcal{M}(L),Y)$ est équivalente à la donnée d'un point $y\in Y$ ainsi qu'une extension $\mathcal{H}(y)\to L$, c'est aussi équivalent à la donnée d'une section $s:\mathcal{M}(L)\to Y_L$ au morphisme structural $Y_L\to \mathcal{M}(L)$, et c'est aussi équivalent à la donnée d'un point de $Y_L$ dont le corps résiduel complété est $L$. Puisque le morphisme $Y_L\to X_L$ est par hypothèse supposé injectif, on en déduit que $\Hom(\mathcal{M}(L),Y)\to \Hom(\mathcal{M}(L),X)$ est injectif.

Maintenant, supposons que $\Hom(\mathcal{M}(L),Y)\to \Hom(\mathcal{M}(L),X)$ est injective. Soit $S$ un espace $k$-analytique ou alors le spectre d'une extension complète de $k$. On veut montrer que $f_S:Y_S\to X_S$ est universellement injectif. Pour cela, soient $y,y'\in Y_S$ qui sont envoyés sur le même point $f_S(y)=f_S(y')=s\in S$. Alors par surjectivité du produit fibré analytique dans le produit fibré topologique, on peut choisir un point $z\in Y_S\times_{S} Y_S$ dont l'image par chacune des deux projections canoniques est $y$ et $y'$. Maintenant, notons $L$ le corps résiduel complété de $z$. Alors, on dispose d'un carré commutatif suivant : 

$$\xymatrix{L&\mathcal{H}(y)\ar[l] \\
\mathcal{H}(y')\ar[u]&\mathcal{H}(s)\ar[u]\ar[l]}
$$

Notons $h_1:\mathcal{M}(L)\to \mathcal{M}(\mathcal{H}(y))\to Y$ et $h_2:\mathcal{M}(L)\to \mathcal{M}(\mathcal{H}(y'))\to Y$ les deux morphismes obtenus par composition. Si l'on note $r:Y_S\to Y$, et $t:Y_S\to S$, alors $r\circ h_1=r\circ h_2$ par hypothèse, car $f\circ r \circ h_1=f\circ r \circ h_2$, et $t\circ h_1=t\circ h_2$ par définition de $L$. Cela montre que les deux morphismes $h_1$ et $h_2$ sont égaux, et donc que $y=y'$.
\end{proof}

En particulier, les inclusions de domaines analytiques et les immersions fermées sont universellement injectives.

\begin{prop} Considérons un morphisme injectif $f:Y\to X$ tel que pour tout $y\in Y$, et $x=f(y)$, l'extension purement inséparable maximale de $\mathcal{H}(x)$ dans $\mathcal{H}(y)$ est dense dans celui-ci. Alors c'est un morphisme universellement injectif, qui vérifie donc les propriétés équivalentes de la proposition précédente.
\end{prop}
\begin{proof}[Preuve]
L'injectivité est claire. On se donne $L$ une extension valuée complète de $k$, et deux morphismes $g,g':\mathcal{M}(L)\to Y$ qui sont égaux après composition par $f$. L'espace $\mathcal{M}(L)$ est un point, et puisque $f$ est injective, l'image de $g$ est aussi l'image de $g'$, c'est un point dont on note $\mathcal{H}(y)$ le corps résiduel complété.

Maintenant, on a un diagramme commutatif $\mathcal{H}(x)\to \mathcal{H}(y)\rightrightarrows L$. Notons $D$ une extension purement inséparable de $\mathcal{H} (x)$ dense dans $\mathcal{H}(y)$. Il suffit de montrer que les deux flèches sont égales en restriction à $D$ par densité, mais par définition d'une extension inséparable, ces deux flèches sont égales, ce qui montre que $g=g'$.
\end{proof}

On peut relier l'injectivité universelle à la notion de quasi-immersion introduite par Berkovich. On rappelle qu'un morphisme $Y\to X$ est une quasi-immersion s'il induit un homéomorphisme sur son image et si pour tout $y\in Y$, l'extension purement inséparable maximale de $\mathcal{H}(x)$ dans $\mathcal{H}(y)$ est dense dans celui-ci. Ce qui précède montre qu'une quasi-immersion est bien universellement injective.

\begin{prop}
Considérons un morphisme plat, surjectif, universellement injectif et presque-affinoïde entre espace analytiques $p:S'\to S$ avec $S$ et $S'$ compacts. Alors toute donnée de descente $X'\to S'$ avec $X'$ compact est effective.
\end{prop}

\begin{proof}[Preuve]
On peut supposer que $S$ et $S'$ sont affinoïdes par le lemme $\ref{recouv}$.
On montre maintenant que tout domaine affinoïde de $X'$ est stable par la donnée de descente en exploitant la condition de cocycle. En effet, si l'on note $q_i:X'\times_S S'\to X'$ comme en $\ref{equiv2}$, alors la relation de cocycle fournit pour tout espace analytique $T$ et tout $T$-point $(y,v)\in (X'\times_S S')(T)$ la relation $q_2(y,v)=q_2(q_2(y,v),v))$, et puisque $q_2$ est obtenu par changement de base à partir de $p$, il est universellement injectif, et cela montre que tout domaine affinoide est stable pour la donnée de descente. Maintenant, on se ramène au cas où $X'$ est affinoide, et par $\ref{cool}$, la donnée de descente est effective. \qedhere

\end{proof}

\section{Application}

On utilise le théorème $\ref{pf}$ pour montrer un petit résultat qui généralise la proposition A-1 de \autocite{rémy2009bruhattits}.

\begin{prop} \label{genera} Considérons un morphisme fidèlement plat $p:\mathcal{M}(\mathcal{B})\to\mathcal{M(A)}$ entre deux espaces $k$-affinoïdes. Considérons aussi un morphisme $j:\mathcal{M(C)}\to\mathcal{M}(\mathcal{A})$ avec $\mathcal{C}$ une algèbre $k$-affinoïde. Alors $\mathcal{M(C)}\to\mathcal{M(A)}$ est une inclusion de domaine affinoïde si et seulement si son changement de base $\mathcal{M(\mathcal{C}\hat{\otimes}_\mathcal{A} \mathcal{B})}\to\mathcal{M}(\mathcal{B})$ l'est.
\end{prop}
\begin{proof}[Démonstration]
L'image réciproque d'un domaine affinoïde est fermée et vérifie clairement la propriété universelle d'un domaine affinoïde. On démontre l'autre sens.

Notons $F$ l'image de $j$ dans $\mathcal{M(A)}$.
Supposons donc que la flèche $\mathcal{M(\mathcal{C}\hat{\otimes}_\mathcal{A} \mathcal{B})}\to\mathcal{M(B)}$ est une inclusion de domaine affinoïde. Notons $D'$ le fermé de $\mathcal{M(B)}$ sous-jacent à $\mathcal{M(\mathcal{C}\hat{\otimes}_\mathcal{A} \mathcal{B})}$. Alors on a l'égalité $p^{-1}(F)=D'$ puisqu'on dispose d'une surjection du produit fibré $\mathcal{M(\mathcal{C}\hat{\otimes}_\mathcal{A} \mathcal{B}})$ vers le produit fibré topologique sous-jacent.
Alors par surjectivité, $F=p(p^{-1}(F))=p(D')$ est l'image d'un compact par une application continue à valeur dans un espace séparé donc compact donc fermé.

On vérifie maintenant que $\mathcal{M(C)}$ vérifie la propriété universelle des domaines affinoïdes. On se donne donc une algèbre $k$-affinoïde $\mathcal{D}$ munie d'un morphisme d'espace affinoïde $p:\mathcal{M(D)}\to\mathcal{M(A)}$ d'image incluse dans $F$. On veut montrer que $f$ se factorise par $j$.

Considérons la flèche naturelle $k:\mathcal{M(D\hat{\otimes}_A B)}\to \mathcal{M(B)}$. Alors la composée $p\circ k$ se factorise par $\mathcal{D}$ et est donc d'image inclue $F$, donc l'image de $k$ est incluse dans $D'$, donc on dispose par propriété universelle d'un morphisme d'espaces affinoïde $g:\mathcal{M(D\hat{\otimes}_A B)}\to \mathcal{M(\mathcal{C}\hat{\otimes}_\mathcal{A} \mathcal{B})}$. 

La situation est résumée par le diagramme commutatif suivant : 

$$\xymatrix{
       & \mathcal{M(B)} \ar[r] & \mathcal{M(A)} & \\
     & \mathcal{M(C\hat{\otimes}_A B)} \ar[r] \ar[u] & \mathcal{M(C)} \ar[u]  & \\
     \mathcal{M(D\hat{\otimes}_A B)} \ar[ur]\ar[rrr] \ar[ruu]& & & \mathcal{M(D)}\ar[uul]
  }$$
Maintenant, si l'on est en mesure d'appliquer le théorème $\ref{final}$, on en déduit que la flèche ainsi définie $g:\mathcal{M(D\hat{\otimes}_A B)}\to \mathcal{M(\mathcal{C}\hat{\otimes}_\mathcal{A} \mathcal{B})}$ est obtenue à partir d'un unique morphisme d'espaces affinoïdes $\mathcal{M(D)}\to\mathcal{M(C)}$ par changement de base, et on a terminé, car cette flèche factorise $\mathcal{M(D)}\to\mathcal{M(A)}$.

Il reste donc à vérifier que si l'on note $p_i:\mathcal{M(B\hat{\otimes}_A B)}\to\mathcal{M(B)}$ les projections canoniques, on a $p_1 ^* g=p_2 ^* g$. Mais par définition les deux flèches $p_i ^* g:\mathcal{M(D\hat{\otimes}_A B\hat{\otimes}_A B)}\to \mathcal{M(C\hat{\otimes}_A B\hat{\otimes}_A B)}$ font commuter le diagramme suivant : 

$$\xymatrix{
     & \mathcal{M(C\hat{\otimes}_A B\hat{\otimes}_A B)}\ar[dr] &  \\
     \mathcal{M(D\hat{\otimes}_A B\hat{\otimes}_A B)} \ar[rr] \ar[ru]& & \mathcal{M(B\hat{\otimes}_A B)}
  }$$
où les autres flèches sont les flèches structurales et la flèche $\mathcal{M(C\hat{\otimes}_A B\hat{\otimes}_A B)}\to \mathcal{M(B\hat{\otimes}_A B)}$ est une inclusion de domaine affinoïde comme changement de base d'une telle inclusion, ce qui garanti l'égalité $p_1 ^* g=p_2 ^* g$ et prouve la proposition. \qedhere 

\end{proof}

On dispose aussi comme corollaire de $\ref{pf}$ du théorème suivant, obtenu par Conrad et Temkin dans \autocite{conrad2019descent} en 3.4 et 4.7 par des méthodes différentes, qui utilisent de manière intensive la réduction à la Temkin. Ici, une immersion compacte $f:Y\to X$ sera un morphisme d'espaces analytiques tel que pour tout domaine analytique compact $U\subset X$, il existe un domaine analytique compact $V$ de $U$ tel que $f^{-1}(U)\to U$ se factorise par une immersion fermée $f^{-1}(U)\to V$. 
\begin{theo} \label{retrouv} Considérons $S$ un espace $k$-analytique et un morphisme de $S$-espaces $k$-analytique $f:Y\to X$. Considérons un morphisme plat et proprement surjectif $p:S'\to S$. Notons $f':Y'\to X'$ le morphisme obtenu à partir de $f$ par changement de base. Alors $f$ vérifie la propriété  suivante si et seulement si $f'$ vérifie la propriété suivante : 
\begin{enumerate}
\item être un isomorphisme
\item être un monomorphisme
\end{enumerate}
De plus, si le morphisme $p$ est topologiquement compact, alors $f$ est une immersion compacte (resp. une immersion ouverte) si et seulement si $f'$ l'est.
\end{theo}
\begin{proof}[Preuve]
Ces propriétés sont stables par changement de base, seule la descente est à montrer. La situation est résumée par le double diagramme cartésien suivant : 

$$\xymatrix{
    Y' \ar[d]^{g'}  \ar[r]^{f'}  & X' \ar[d]^{g}  \ar[r] &S'\ar[d]^p \\
    Y \ar[r]^f  & X \ar[r] & S
  }$$

Si $f'$ est un isomorphisme, $f'$ induit alors un isomorphisme de $S'$-espaces analytiques munis de données de descente relativement à $p$, donc par le théorème $\ref{pf}$, $f$ est un isomorphisme de $S$-espaces analytiques.

Maintenant, $f$ est un monomorphisme si et seulement le morphisme diagonal $\Delta_f:Y\to Y\times_X Y$ est un isomorphisme, et comme $\Delta_{f'}=\Delta_f\times_S S'$, par descente des isomorphismes, $f$ est un monomorphisme si et seulement si $f'$ l'est.

Supposons que le morphisme $p$ est topologiquement compact. Par la proposition 3.2.17 de \autocite{ducros2014structure}, une immersion compacte est simplement un monomorphisme compact. Il suffit donc de montrer que la propriété d'être un morphisme compact se descend, ce qui vient du fait que les morphisme $g:X'\to X$ (resp. $g':Y'\to Y$) obtenus à partir de $p$ en effectuant le changement de base par $X\to S$ (resp. $Y\to S$) restent surjectifs et topologiquement compacts. Ainsi, si l'on suppose $f'$ compact, et si l'on choisi un compact $K\subset X$ alors son image inverse par $f$ sur $Y$ est aussi $g'(g'^{-1}(f^{-1}(K)))=g'(f'^{-1}(g^{-1}(K)))$, qui est donc compact.

Maintenant, si $f'$ est une immersion ouverte et que $p$ est toujours supposé topologiquement compact, alors on a la relation $g^{-1}(f(Y))=f'(Y')$ qui est ouverte, donc puisque les surjections topologiquement propre entre espaces analytiques sont submersives par $\ref{submersive}$, on en déduit que l'image $f(Y)$ est ouverte, et par définition, $f$ induit par changement de base un isomorphisme $f'$ entre $Y'$ et $f'(Y')$ donc $f$ induit aussi un isomorphisme entre $Y$ et $f(Y)$ par ce qui précède, et $f$ est bien une immersion ouverte.
\end{proof}

\begin{rema} Ce théorème est loin d'être optimal : dans \autocite{conrad2019descent}, les auteurs montrent que la propriété d'être un monomorphisme est locale pour la topologie dont les flèches couvrantes sont juste les flèches surjectives.
\end{rema}

\section {Problèmes d'algébrisation}

On dit qu'un $\mathcal{A}$-espace analytique est algébrisable s'il appartient à l'image essentielle du foncteur d'analytification, et si $\mathcal{X}$ et $\mathcal{Y}$ sont deux $\mathcal{A}$-schémas localement de type fini, un morphisme $f:X:=\mathcal{X}^\textrm{an}\to Y:=\mathcal{Y}^\textrm{an}$ est dit algébrique s'il existe un morphisme de $\mathcal{A}$-schémas de $\mathcal{X}$ vers $\mathcal{Y}$ dont l'analytification est $f$. Si $\mathcal{X}$ est un $\mathcal{A}$-schéma localement de type fini, on dit qu'une partie localement constructible $E\subset \mathcal{X}^{\textrm{an}}$ est algébrique si c'est l'image réciproque ensembliste par la flèche d'analytification d'une partie localement constructible de $\mathcal{X}$. La fin de cette sous-section sera occupée par la démonstration du résultat suivant, qui fait l'objet du théorème $\ref{algmorph}$ : la propriété pour un morphisme d'être algébrique est locale pour la topologie sur la catégorie des espaces $k$-affinoïdes dont les flèches couvrantes sont les flèches plates et surjectives. On démontre d'abord quelques lemmes qui serviront au cours de la preuve de ce résultat.

\begin{lemm} \label{Zloc}
Considérons une algèbre $k$-affinoïde $\mathcal{A}$, un $\mathcal{A}$-schéma localement de type fini $\mathcal{X}$ et un recouvrement de $\mathcal{X}$ par des ouverts affines $(\mathcal{X}_i)_{i\in I}$. Alors une partie $E\subset \mathcal{X}^\mathrm{an}$ est localement constructible et algébrique si et seulement si pour tout $i\in I$, la partie $E\cap \mathcal{X}_i^\mathrm{an}$ est une partie algébrique constructible de $\mathcal{X}^\textrm{an}_i$.
\end{lemm}
\begin{proof}[Preuve]

Si $E\subset \mathcal{X}^\textrm{an}$ est localement constructible et algébrique, alors il existe $E'\subset \mathcal{X}$ localement constructible dont l'image inverse par la flèche d'analytification est $E$. Maintenant, on a $E\cap \mathcal{X}_i^\textrm{an}=(E'\cap \mathcal{X}_i)^\textrm{an}$ qui est donc une partie constructible algébrique de $\mathcal{X}_i^\textrm{an}$.

Réciproquement, si, pour tout $i\in I$, on dispose d'une partie constructible $E_i\subset \mathcal{X}_i$ dont l'analytification est $E\cap \mathcal{X}_i^\textrm{an}$, alors par surjectivité du morphisme d'analytification $m:\mathcal{X}^\textrm{an}\to \mathcal{X}$, on a l'égalité $E_i\cap \mathcal{X}_i\cap \mathcal{X}_j=m(m^{-1}(E_i\cap \mathcal{X}_i\cap \mathcal{X}_j))=m(E_i^\textrm{an}\cap \mathcal{X}_i^\textrm{an}\cap \mathcal{X}_j^\textrm{an})=m(E\cap\mathcal{X}_i^\textrm{an} \cap \mathcal{X}_j^\textrm{an})=E_j\cap \mathcal{X}_i \cap \mathcal{X}_j$, donc la collection ensembliste de partie $E_i\subset\mathcal{X}_i$ se recolle en une partie $E'\subset \mathcal{X}$ qui est localement constructible. De plus, on a pour tout $i\in I$ l'égalité ${E'}^\textrm{an}\cap \mathcal{X}_i^\textrm{an}=(E'\cap \mathcal{X}_i)^\textrm{an}=E_i^\textrm{an}=E\cap \mathcal{X}_i^\textrm{an}$, donc ${E'}^\textrm{an}=E$ et la partie $E$ est bien algébrique et localement constructible.
\end{proof}

\begin{lemm} Le foncteur d'analytification de la catégorie des $\mathcal{A}$-schémas localement de type fini vers la catégorie des espaces $\mathcal{A}$-analytiques est fidèle. \label{analfid}
\end{lemm}
\begin{proof}[Preuve] Par surjectivité ensembliste du morphisme d'espaces localement annelé $\mathcal{X}^{\textrm{an}}\to \mathcal{X}$, si deux morphismes de $\mathcal{A}$-schémas ont même analytification, alors ils coïncident topologiquement, et comme le foncteur d'analytification préserve les immersions ouvertes, montrer que ces deux morphismes coïncident est une assertion Zariski-locale sur la base et la source, et il suffit donc de montrer que l'on a une injection $\Hom_{\mathcal{A}-\textrm{Sch}}(\mathcal{X},\mathcal{Y})\to \Hom_\mathcal{A\textrm{-ann}}(X,Y)$, avec $\mathcal{X}$ et $\mathcal{Y}$ affines, et $X$, $Y$ les deux analytifiés respectifs.

Maintenant, quitte à plonger $\mathcal{Y}$ dans $\mathbb{A}^{n,alg}_\mathcal{A}$, on peut supposer que ces deux espaces sont égaux, et puisque se donner un morphisme dans la catégorie des espaces localement annelés de $\mathcal{X}$ vers $\mathbb{A}^{n,alg}_\mathcal{A}$ est équivalent à se donner $n$ fonctions dans $\mathcal O (X)$, on peut supposer que $\mathcal{Y}$ est la droite affine algébrique de dimension 1, et l'assertion à montrer est que la flèche $\Gamma (\mathcal{X},\mathcal{O}_\mathcal{X})\to \Gamma(X,\mathcal{O}_X)$ est une injection.

Cela découle maintenant uniquement de la platitude de la flèche $X\to \mathcal X$. En effet, $\Gamma (X,O_X)$ (resp. $\Gamma (\mathcal X,O_{\mathcal{X}})$) s'identifie canoniquement à $\Hom_{\mathcal O_X}(\mathcal O_X,\mathcal{O}_X )$ (resp. $\Hom_{\mathcal O_\mathcal X}(\mathcal O_\mathcal X,\mathcal{O}_\mathcal X )$), et par platitude, un homomorphisme $f:\mathcal O_\mathcal{X}\to \mathcal O_\mathcal{X}$ est nul si et seulement si son image $\mathcal O_X\to \mathcal{O}_X$ est nulle. \qedhere 

\end{proof}

\begin{rema} Dans la preuve du résultat précédent, on a montré au passage que la flèche canonique $\Gamma (\mathcal{X},\mathcal{O}_\mathcal{X})\to \Gamma(\mathcal{X}^{\textrm{an}},\mathcal{O}_{X^{\textrm{an}}})$ est injective pour tout $\mathcal{A}$-schémas affine, et cela implique l'énoncé pour tout $\mathcal{A}$-schéma localement de type fini $\mathcal{X}$. On dit qu'une fonction $f\in \Gamma(\mathcal{X}^{\textrm{an}},\mathcal{O}_\mathcal{X}^{\textrm{an}})$ est algébrique si elle est dans l'image de $\Gamma (\mathcal{X},\mathcal{O}_\mathcal{X})$.
\end{rema}

\begin{rema} Considérons $X$ un espace $k$-analytique, et $M$ un $\mathcal{O}(X)$-module de type fini. Alors on peut comme en définition 2.3 de \autocite{maculan2018notions} associer à $M$ un faisceau de modules cohérents $\tilde{M}$ au dessus de $X$ vérifiant $\tilde{M}(D)=M\otimes_{\mathcal{O}(X)} \mathcal{O}(D)$ pour tout domaine affinoïde $D\subset X$. De plus si $X$ est cohomologiquement de Stein (pour tout faisceau cohérent de $\mathcal{O}_X$-modules $F$ et tout $q\geqslant 1$, les groupes de cohomologie $\textrm{H}^q (X,F)$ sont nuls) la flèche canonique $M\to \tilde{M}(X)$ est surjective par la proposition 2.6 de \autocite{maculan2018notions}. Maintenant, si l'on se donne ${J}$ un idéal de type fini de $\mathcal{O}(X)$, comme la limite inverse est exacte à gauche, $\tilde{J}(X)\to \mathcal{O}(X)$ est injective, donc on dispose en fait d'un isomorphisme de $\mathcal{O}(X)$-modules $J\to \tilde{J}(X)$. En particulier, cela vaut pour $\mathbb{A}^{\textrm{an},n}_\mathcal{A}$ l'espace affine analytique de dimension $n\in \mathbb{N}$ au dessus d'un espace $k$-affinoïde $\mathcal{A}$ puisque celui-ci est bien cohomologiquement de Stein. On renvoie à \autocite{maculan2018notions} pour les résultats concernant les espaces de Stein, et les différentes définitions équivalentes d'espaces de Stein.

\end{rema}
La proposition suivante constitue le cœur de la preuve du théorème $\ref{algmorph}$, puisqu'elle permet de traiter le cas où l'espace $\mathcal{X}$ est affine, et $\mathcal{Y}$ est la droite algébrique affine.
\begin{prop} Considérons un morphisme fidèlement plat $p:\mathcal{M(B)}\to \mathcal{M(A)}$ entre espaces $k$-affinoïdes, $\mathcal{X}$ un $\mathcal{A}$-schéma affine de type fini et $f\in \Gamma(\mathcal{X}^{\mathrm{an}},\mathcal{O}_{\mathcal{X}^{\mathrm{an}}})$ une fonction sur l'analytifié de $\mathcal{X}$. Alors $f$ est algébrique sur $\mathcal{X}^{\mathrm{an}}$ si et seulement si son image dans $\Gamma(\mathcal{X}_{\mathcal{B}}^{\mathrm{an}},\mathcal{O}_{\mathcal{X}_{\mathcal{B}}^{\mathrm{an}}})$ est algébrique. \label{coeur}
\end{prop}

\begin{proof}[Démonstration]
On commence par traiter le cas où $\mathcal{X}=\mathbb{A}_\mathcal{A}^{n}$ est l'espace affine de dimension $n\in \mathbb{N}$. On dispose alors d'une description simple de $\Gamma(\mathcal{X}^{\textrm{an}},\mathcal{O}_\mathcal{X}^{\textrm{an}})$ en tant que sous ensemble de $\mathcal{A}[[T_1,..,T_n]]$ puisque l'on a $\Gamma(\mathcal{X}^{\textrm{an}},\mathcal{O}_\mathcal{X}^{\textrm{an}})=\{\sum a_i \underline{T}^i\mid a_i\in \mathcal{A},\textrm{max} \Vert a_i \Vert \underline{r}^i <+\infty, \forall \underline{r}\in (\mathbb{R}_{+}^*)\}$. Alors la flèche $\Gamma(\mathcal{X},\mathcal{O}_\mathcal{X})\to \Gamma(\mathcal{X}^{\textrm{an}},\mathcal{O}_\mathcal{X}^{\textrm{an}})$ est simplement l'inclusion de $\mathcal{A}[T_1,..T_n]$ dans cet espace. Maintenant, si l'on se donne $f\in \Gamma(\mathcal{X}^{\textrm{an}},\mathcal{O}(\mathcal{X}^{\textrm{an}}))$ dont l'image $f_\mathcal{B}$ dans $\Gamma(\mathcal{X}_{\mathcal{B}}^{\textrm{an}},\mathcal{O}_{\mathcal{X}_{\mathcal{B}}^{\textrm{an}}})$ est algébrique, alors l'image de $f_\mathcal{B}$ dans $\mathcal{B}[[T_1,..T_n]]$ est un polynôme, donc l'image de $f$ dans $\mathcal{A}[[T_1,..T_n]]$ est aussi un polynôme, ce qui montre que $f$ est algébrique car on a une inclusion $\mathcal{A}\hookrightarrow\mathcal{B}$ par fidèle platitude schématique de cette flèche.

Maintenant, on traite le cas général, $\mathcal{X}$ est juste supposé affine. Par définition, il existe une immersion fermée $\mathcal{X}\hookrightarrow \mathbb{A}_\mathcal{A}^{n}$. Notons $\mathcal{J}$ le module des sections globales du faisceau d'idéaux qui définit $\mathcal{X}$ en tant que sous-schéma fermé de $\mathbb{A}_\mathcal{A}^{n}$. Alors comme $\mathcal{A}[T_1,..,T_n]$ est noethérien, il existe un nombre fini de polynômes $P_1,..,P_m$ qui engendrent l'idéal $\mathcal{J}\subset \mathcal{A}[T_1,..,T_n]$. Montrons l'égalité préliminaire : 
\setcounter{subsection}{1}
\setcounter{subsubsection}{1}
\begin{equation}\mathcal{J} \mathcal{O}(\mathbb{A}_\mathcal{B}^{n,\textrm{an}})\cap \mathcal{O}(\mathbb{A}_\mathcal{A}^{n,\textrm{an}})=\mathcal{J} \mathcal{O}(\mathbb{A}_{\mathcal{A}}^{n,an}).\label{eg0}
\end{equation}

 Une inclusion est claire, soit donc $D\in \mathcal{J} \mathcal{O}(\mathbb{A}_\mathcal{B}^{n,\textrm{an}})\cap \mathcal{O}(\mathbb{A}_\mathcal{A}^{n,\textrm{an}}) $. Alors pour tout polyrayon $\underline{r}\in (\mathbb{R}_+^*)^n$, l'image de $D$ dans $\mathcal{B}\{\underline{r}^{-1}\underline{T}\}$ est dans $\mathcal{J}\mathcal{B}\{\underline{r}^{-1}\underline{T}\}\cap \mathcal{A}\{\underline{r}^{-1}\underline{T}\}$, et par fidèle platitude de la flèche $\mathcal{A}\{\underline{r}^{-1}\underline{T}\}\to \mathcal{B}\{\underline{r}^{-1}\underline{T}\}$, l'élément $D$ est en fait dans $\mathcal{J}\mathcal{A}\{\underline{r}^{-1}\underline{T}\}$. Comme cela vaut pour tout polyrayon positif, comme on a la platitude de la flèche $\Gamma (\mathbb{A}_{\mathcal{A}}^{n,an},\mathcal{O}_{\mathbb{A}_{\mathcal{A}}^{n,an}}) \to \mathcal{A}\{\underline{r}^{-1}\underline{T}\}$ par la proposition 2.6 de \autocite{maculan2018notions}, on en déduit que $D\in \widetilde{\mathcal{J}\mathcal{O}\mathbb{A}_\mathcal{A}^{n,\textrm{an}}}(\mathbb{A}_\mathcal{A}^{n,\textrm{an}})$, et par la remarque qui précède le lemme, on a en fait $D\in \mathcal{J}\mathcal{O}(\mathbb{A}_\mathcal{A}^{n,\textrm{an}})$ puisque $\mathcal{J}\mathcal{O}(\mathbb{A}_\mathcal{A}^{n,\textrm{an}})$ est bien un $\mathcal{O}(\mathbb{A}_{\mathcal{A}}^{n,an})$-module de type fini. 

Maintenant, on dispose du diagramme suivant : 

\begin{equation*}
    \begin{tikzpicture}
    
    %%%% Sommets du diagramme
    % Sommets de la premiere suite
    \node (A0) at (-4,1) {$0$};
    \node (A1) at (-2,1) {$\mathcal{J}$};
    \node (A2) at (0,1) {$\mathcal{O}(\mathbb{A}_\mathcal{A}^n)$};
    \node (A3) at (2,1) {$\mathcal{O}(\mathcal{X})$};
    \node (A4) at (4,1) {$0$};
    
        % Sommets de la deuxieme suite
    \node (B0) at (-4,-1) {$0$};
    \node (B1) at (-2,-1) {$\mathcal{J} \mathcal{O}(\mathbb{A}_\mathcal{A}^{n,\textrm{an}})$};
    \node (B2) at (0,-1) {$\mathcal{O}(\mathbb{A}_\mathcal{A}^{n,\textrm{an}})$};
    \node (B3) at (2,-1) {$\mathcal{O}(\mathcal{X}^{\textrm{an}})$};
    \node (B4) at (4,-1) {$0$};
    
            % Sommets de la troisieme suite
    \node (C0) at (-2.5,2) {$0$};
    \node (C1) at (-0.5,2) {$\mathcal{J}_\mathcal{B}$};
    \node (C2) at (1.5,2) {$\mathcal{O}(\mathbb{A}_\mathcal{B}^n)$};
    \node (C3) at (3.5,2) {$\mathcal{O}(\mathcal{X}_\mathcal{B})$};
    \node (C4) at (5.5,2) {$0$};
    
            % Sommets de la quatrieme suite
    \node[color=gray!95] (D0) at (-2.5,0) {$0$};
    \node[color=gray!95] (D1) at (-0.5,0) {$\mathcal{J}_\mathcal{B} \mathcal{O}(\mathbb{A}_\mathcal{B}^{n,\textrm{an}})$};
    \node[color=gray!95] (D2) at (1.5,0) {$\mathcal{O}(\mathbb{A}_\mathcal{B}^{n,\textrm{an}})$};
    \node[color=gray!95] (D3) at (3.5,0) {$\mathcal{O}(\mathcal{X}_\mathcal{B}^{\textrm{an}})$};
    \node[color=gray!95] (D4) at (5.5,0) {$0$};
    
    %%%% Fleches du diagramme
    %%%%%% Fleches de la premiere suite
    \draw[->] (A0) -- node[midway, above]{} (A1);
    \draw[->] (A1) -- node[midway, above]{$ {}$} (A2);
    \draw[->] (A2) -- node[midway, above]{$ $} (A3);
    \draw[->] (A3) -- node[midway, above]{} (A4);
        %%%%%% Fleches de la deuxieme suite
    \draw[->] (B0) -- node[midway, above]{} (B1);
    \draw[->] (B1) -- node[midway, above]{$ $} (B2);
    \draw[->] (B2) -- node[midway, above]{$ $} (B3);
    \draw[->] (B3) -- node[midway, above]{} (B4);
    
        %%%%%% Fleches de la troisieme suite
    \draw[->] (C0) -- node[midway, above]{} (C1);
    \draw[->] (C1) -- node[midway, above]{$ $} (C2);
    \draw[->] (C2) -- node[midway, above]{$ $} (C3);
    \draw[->] (C3) -- node[midway, above]{} (C4);
    
        %%%%%% Fleches de la quatrieme suite
    \draw[->,color=gray!95] (D0) -- node[midway, above]{} (D1);
    \draw[->,color=gray!95] (D1) -- node[midway, above]{$ $} (D2);
    \draw[->,color=gray!95] (D2) -- node[midway, above]{$ $} (D3);
    \draw[->,color=gray!95] (D3) -- node[midway, above]{} (D4);
    
    %%%%%% Fleches entre suites
    %%% A to B
    \draw[right hook->] (A1) -- node[midway, left]{} (B1);
        \draw[right hook->] (A2) -- node[midway, left]{} (B2);
            \draw[right hook->] (A3) -- node[midway, left]{} (B3);

    %%% A to C
    \draw[right hook->] (A1) -- node[midway, left]{} (C1);
        \draw[right hook->] (A2) -- node[midway, left]{} (C2);
            \draw[right hook->] (A3) -- node[midway, left]{} (C3);
            
                %%% C to D
    \draw[right hook->,color=gray!75] (C1) -- node[midway, left]{} (D1);
        \draw[right hook->,color=gray!75] (C2) -- node[midway, left]{} (D2);
            \draw[right hook->,color=gray!75] (C3) -- node[midway, left]{} (D3);
            
                %%% B to D
    \draw[right hook->] (B1) -- node[midway, left]{} (D1);
        \draw[right hook->] (B2) -- node[midway, left]{} (D2);
            \draw[right hook->] (B3) -- node[midway, left]{} (D3);
    
    \end{tikzpicture}
\end{equation*}
avec $\mathcal{J}_\mathcal{B}=\mathcal{J}\otimes_\mathcal{A}\mathcal{B}$. Chaque carré est commutatif, et les deux lignes du haut sont exactes par définition et par fidèle platitude schématique du morphisme $\Spec (\mathcal{B})\to \Spec\mathcal{(A)}$. Pour tout domaine affinoïde $D\subset \mathbb{A}_\mathcal{A}^{n,\textrm{an}}$, la flèche $\mathcal{O}({\mathbb{A}_\mathcal{A}^{n,\textrm{an}}})\to \mathcal{O}(D)$ est plate par la proposition 2.6 de \autocite{maculan2018notions} donc on a l'égalité $(\tilde{J})^\textrm{an}=\widetilde{J \mathcal{O}({\mathbb{A}_\mathcal{A}^{n,\textrm{an}}})}$ et on en déduit la suite exacte de $\mathcal{O}_{\mathbb{A}_\mathcal{A}^{n,\textrm{an}}}$-modules suivante : $0\to \widetilde{J \mathcal{O}({\mathbb{A}_\mathcal{A}^{n,\textrm{an}}})}\to \mathcal{O}_{\mathbb{A}_\mathcal{A}^{n,\textrm{an}}} \to \mathcal{O}_{\mathcal{X}^{\textrm{an}}}\to 0$. Comme $\mathcal{J} \mathcal{O}(\mathbb{A}_\mathcal{A}^{n,\textrm{an}})$ est un $\mathcal{O}( \mathbb{A}_\mathcal{A}^{n,\textrm{an}})$-module de type fini, par la remarque qui précède le lemme, on a l'égalité $\widetilde{\mathcal{J}\mathcal{O}( \mathbb{A}_\mathcal{A}^{n,\textrm{an}})}( \mathbb{A}_\mathcal{A}^{n,\textrm{an}})=\mathcal{J}\mathcal{O}( \mathbb{A}_\mathcal{A}^{n,\textrm{an}})$, ce qui fournit l'exactitude des deux lignes du bas du diagramme en utilisant la steinitude cohomologique de $\mathbb{A}_\mathcal{A}^{n,\textrm{an}}$.

On se donne donc un élément $f\in \mathcal{O} (\mathcal{X}^{\textrm{an}})$ dont l'image $f_\mathcal{B}$ dans $\mathcal{O}(\mathcal{X}^{\textrm{an}}_\mathcal{B})$ est dans l'image de $\mathcal{O}(\mathcal{X}_\mathcal{B})$. Montrons que $f$ est dans $\mathcal{O}(\mathcal{X})$. Par surjectivité des flèches idoines, il existe un élément $f_1\in \mathcal{O}(\mathbb{A}_{\mathcal{A}}^{n,an})$ dont l'image dans $\mathcal{O}(\mathcal{X}^{\textrm{an}})$ est $f$, et un élément algébrique $f_2\in \mathcal{O}(\mathbb{A}_{\mathcal{B}}^{n})$ dont l'image dans $\mathcal{O}(\mathcal{X}^{\textrm{an}}_\mathcal{B})$ est $f_\mathcal{B}$. Maintenant, puisque les flèches idoines sont des inclusions, $f_1$ et $f_2$ définissent deux éléments de $\mathcal{O}(\mathbb{A}_{\mathcal{B}}^{n,an})$ dont l'image dans $\mathcal{O}(\mathcal{X}^\textrm{an}_\mathcal{B})$ vaut $f_\mathcal{B}$, donc on peut écrire l'égalité suivante dans $\mathcal{O}(\mathbb{A}_{\mathcal{B}}^{n,an})$ :
\begin{equation}\label{eg}
f_1-f_2=h 
\end{equation} avec $h\in \mathcal{J}_{\mathcal{B}} \mathcal{O}(\mathbb{A}_{\mathcal{B}}^{n,an})$. L'élément $f_1$ est une série formelle à coefficients dans $\mathcal{A}$, et l'élément $f_2$ est un polynôme à coefficients dans $\mathcal{B}$, tandis que l'on peut écrire $h$ comme une somme finie $h=\sum_I P_I g_I$ avec $g_I\in \mathcal{O}(\mathbb{A}_{\mathcal{B}}^{n,an})$ et $P_I\in \mathcal{A}[T_1,..,T_n]$ les polynômes qui engendrent l'idéal $\mathcal{J}\subset \mathcal{A}[T_1,..,T_n]$.

Puisque tous les éléments de l'égalité $\ref{eg}$ sont des séries formelles, en identifiant les coefficients, on en déduit qu'il existe un indice à partir duquel tous les coefficients de $h$ sont dans $\mathcal{A}$, ce qui nous fournit l'existence d'une décomposition $g_I=g_{I_1}+g_{I_2}$ avec $g_{I_1}\in  \mathcal{O}(\mathbb{A}_{\mathcal{B}}^{n,an})$ et $g_{I_2}\in \mathcal{B}[T_1,..,T_n]$ vérifiant $\sum_I P_I g_{I_1}\in  \mathcal{O}(\mathbb{A}_{\mathcal{A}}^{n,an})$, et par l'égalité $\ref{eg0}$, on en déduit que l'on a même $\sum_I P_I g_{I_1}\in \mathcal{J}\mathcal{O}(\mathbb{A}_\mathcal{A}^{n,\textrm{an}})$, ce qui fournit un nombre fini d'éléments $g'_{I_1}\in \mathcal{O}(\mathbb{A}_\mathcal{A}^{n,\textrm{an}}) $ tels que $\sum_I P_I g_{I_1}=\sum_I P_I g'_{I_1}$. 

Maintenant, posons $f'_1=f_1-\sum_I P_I g'_{I_1}$ et $f'_2=f_2-\sum_{I} P_I g_{I_2}$. Alors par $\ref{eg}$, on a l'égalité $f'_1=f'_2$, donc $f'_1$ est dans $\mathcal{O}(\mathbb{A}_{\mathcal{A}}^{n,an})\cap \mathcal{O}(\mathbb{A}_{\mathcal{B}}^{n})=\mathcal{O}(\mathbb{A}_{\mathcal{A}}^{n})$ par le début de la démonstration, et cet élément $f'_1$ diffère de $f_1$ d'un élément de $\mathcal{J}\mathcal{O}(\mathbb{A}_{\mathcal{A}}^{n,an})$, donc l'image de $f'_1$ dans $\mathcal{O}(\mathcal{X}^{\textrm{an}})$ est aussi $f$, et comme $f'_1$ est un polynôme, on en déduit que $f$ appartient à $\mathcal{O}(\mathcal{X})$, et $f$ est algébrique ce que l'on voulait montrer.
\end{proof}

On énonce une dernière proposition qui nous permettra d'effectuer une réduction décisive de problème dans la preuve du théorème $\ref{algmorph}$. La preuve repose sur le schéma de démonstration du théorème $\ref{fibfi}$ : pour démontrer la proposition, on la démontre pour trois classes de morphismes plus simple, puis on en déduit le cas d'un morphisme fidèlement plat.

\begin{lemm} \label{lemmepartie} Considérons $p:\mathcal{M(B)}\to \mathcal{M(A)}$ un morphisme fidèlement plat entre espaces $k$-affinoïdes et $\mathcal{X}$ un $\mathcal{A}$-schéma localement de type fini. Alors une partie $E\subset \mathcal{X}^{\mathrm{an}}$ est localement constructible algébrique si et seulement si son image inverse $E_\mathcal{B}={(q^\mathrm{an}})^{-1}(E)$ est localement constructible, algébrique, où $q^\mathrm{an}$ est le morphisme $q^\textrm{an}:\mathcal{X}^\mathrm{an}_\mathcal{B}\to \mathcal{X}^{\mathrm{an}}$ obtenu à partir de $p$ par changement de base.
\end{lemm}

\begin{proof}[Démonstration] Un sens est évident, on va démontrer l'autre sens.

Si l'on se donne un recouvrement de $\mathcal{X}$ par des ouverts affines $(\mathcal{X}_i)$, et que $E_\mathcal{B}$ est localement constructible et algébrique alors $\mathcal{X}_{i_\mathcal{B}}\cap E_\mathcal{B}$ est encore algébrique en tant que fermé de $\mathcal{X}_{i_\mathcal{B}} $, donc si l'on suppose le résultat acquis pour l'analytifié de tout schéma affine, on en déduit que $\mathcal{X}^\textrm{an}_i\cap E$ est algébrique constructible pour tout $i$, et par le lemme $\ref{Zloc}$, on en déduit que $E$ est localement constructible et algébrique. Cela nous permet donc de réduire la démonstration du théorème au cas où $\mathcal{X}$ est un schéma affine $\mathcal{X}=\Spec C$. Sur $\mathcal{X}$ ainsi que sur $\mathcal{X}^\textrm{an}$, les notions de constructibilité et de locale constructibilité coïncident par noethérianité de $\mathcal{A}$ et parce que $\mathcal{X}^\mathrm{an}$ est de dimension finie.

On va d'abord montrer le résultat si $p$ est de la forme $\mathcal{M(A}_r)\to \mathcal{M(A)}$ pour $r\in \mathbb{R}^*_+$ un polyrayon $k$-libre, puis lorsque $p$ est plat fini et surjectif puis lorsque $p$ est un $G$-recouvrement par des domaines affinoïdes. Alors la technique de démonstration donnée au théorème $\ref{fibfi}$ permettra de conclure pour un morphisme fidèlement plat quelconque.

On se donne donc d'abord un polyrayon $k$-libre $r\in (\mathbb{R}^*_+)^n$, et on va montrer que le lemme est vrai dans le cas de la flèche $p:\mathcal{M}(\mathcal{A}_r)\to \mathcal{M(A)}$. Soit donc $E$ une partie de $\mathcal{X}^\textrm{an}$ telle que $E_{r}$ soit une partie constructible algébrique de  $\mathcal{X}^\textrm{an}_r$, on veut montrer que $E$ est en fait constructible algébrique. La partie $E_r$ s'écrit comme une union finie $E_r=\bigcup_{i\in I} {U_i\cap V_i}$ avec $U_i$ (resp. $V_i$) des ouverts (resp. fermés) algébriques de $\mathcal{X}_r^\textrm{an}$. De plus, si $\sigma:\mathcal{X}^\textrm{an}\to\mathcal{X}^\textrm{an}_r$ est la section de Shilov du morphisme $q^\textrm{an}$, alors on a l'égalité $E=\sigma^{-1}({q^\textrm{an}}^{-1}(E))=\sigma^{-1}(\bigcup_{i\in I} {U_i\cap V_i})=\bigcup_{i\in I} {\sigma^{-1}(U_i\cap V_i})$. Puisque la propriété d'être constructible et algébrique est stable par union finie, on peut supposer que $E_r=U\cap V$ avec $U$ ouvert algébrique et $V$ fermé algébrique de $\mathcal{X}_r^\textrm{an}$ et montrer que $\sigma^{-1}(U\cap V)$ est constructible algébrique. Par noethérianité de $C\otimes_\mathcal{A}\mathcal{A}_r$, il existe un nombre fini de fonctions $(g_j)_{j\in J}$ (resp. $(f_j)_{j\in J}$) de $C\otimes_\mathcal{A} \mathcal{A}_r$ dont $V$ (resp. $U^\mathrm{C}$) est le lieu d'annulation. De plus, on peut écrire chaque fonction $f_j$ sous la forme $f_j=\sum_{M\in \mathbb{Z}^n} f_{jM} \underline{T}^M$ avec $f_{jM}\in C$ et chaque fonction $g_j$ sous la forme  $g_j=\sum_{M\in \mathbb{Z}^n} g_{jM} \underline{T}^M$ avec $g_{jM}\in C$, et par définition de la section de Shilov, $\vert f_j \vert (\sigma(x))=\max_{M\in \mathbb{Z}^n} \vert f_{jM}(x)\vert r^M$, et donc $E=\{x\in\mathcal{X}^\textrm{an}, \forall j\in J,\forall M \in \mathbb{Z}^n, \vert g_{jM}(x)\vert =0 \}\bigcup \{x\in\mathcal{X}^\textrm{an},\exists j \in J, \exists M\in \mathbb{Z}^n, \vert f_{jM}(x)\vert \neq 0 \}.$ Notons maintenant $\mathcal{I}\subset C$ l'idéal engendré par les $g_{jM}$ dans $C$ et $\mathcal{J}\subset C$ l'idéal engendré par les $f_{jM}$ dans $C$. L'égalité précédente montre alors que $E=({V(\mathcal{I})\bigcup V(\mathcal{J})^\mathrm{C}})^\textrm{an}$, ce qui montre que $E$ est bien algébrique constructible.

Maintenant, on suppose que $p:\mathcal{M(B)}\to\mathcal{M(A)}$ est fini, plat et surjectif, et $E$ est comme dans le paragraphe précédent une partie de $\mathcal{X}^\textrm{an}$ dont l'image inverse $E_\mathcal{B}\subset \mathcal{X}_\mathcal{B}^\textrm{an}$  est constructible algébrique. Par définition, on peut écrire $E_B$ comme l'analytification d'une partie constructible $E'$. Comme $\mathcal{B}$ est noethérien, par finitude on a l'égalité $\mathcal{B}\hat{\otimes}_\mathcal{A}\mathcal{B}=\mathcal{B}\otimes_\mathcal{A}\mathcal{B}$, et le morphisme $k_3:\mathcal{X}^\textrm{an}_{\mathcal{B}\hat{\otimes}_\mathcal{A}\mathcal{B}}\to\mathcal{X}_{\mathcal{B\otimes_\mathcal{A}\mathcal{B}}}$ s'identifie avec la flèche d'analytification. Si ${q}^{\textrm{an}}_i:\mathcal{X}^\textrm{an}_{\mathcal{B}\hat{\otimes}_\mathcal{A}\mathcal{B}}\to \mathcal{X}^\textrm{an}_\mathcal{B}$ désigne chacune des deux projections, alors on a des données de descente ensemblistes canoniques sur $E_B$, donc on dispose de l'égalité ${q^\textrm{an}}^{-1}_1(E_\mathcal{B})={q^{\textrm{an}}}^{-1}_2(E_\mathcal{B})$, et si l'on note $q_i:\mathcal{X}_{\mathcal{B\otimes_\mathcal{A}\mathcal{B}}}\to \mathcal{X}_\mathcal{B}$ les deux projections, on a donc $({q}^{-1}_1((E'))^\textrm{an}=({q}^{-1}_2(E'))^\textrm{an}$, et par surjectivité du morphisme d'analytification $k_3$, on en déduit que ${q}^{-1}_1 (E')={q}^{-1}_2 (E')$, donc il existe une partie $E'_0\subset \mathcal{X}$ telle que ${q}^{-1}(E'_0)=E'$. De plus, on a par surjectivité du morphisme $q$ l'égalité $E'_0=q(E')$, et le morphisme $q$ étant fini, par le théorème de Chevalley, la partie $E'_0$ est constructible, et l'analytification de $E'_0$ est bien $E$ par surjectivité du morphisme d'analytification et du morphisme $q^\textrm{an}$, ce qui montre que $E$ est bien une partie constructible algébrique de $\mathcal{X}^\textrm{an}$.

On traite ensuite le cas où le morphisme $p:\mathcal{M(B)}\to\mathcal{M(A)}$ est un $G$-recouvrement $p:\coprod_i \mathcal{M}(\mathcal{A}_i)\to\mathcal{M(A)}$ de $S:=\mathcal{M(A)}$ par un nombre fini de domaines affinoïdes $S_i:=\mathcal{M}(\mathcal{A}_i)\subset \mathcal{M(A)}$. Considérons donc $E$ une partie de $\mathcal{X}^\textrm{an}=(\Spec C)^\textrm{an}$ telle que $E_\mathcal{B}$ est une partie constructible algébrique de $\mathcal{X}^\textrm{an}_\mathcal{B}$. Puisque $\mathcal{X}$ est affine, on peut trouver une immersion ouverte d'image dense $j:\mathcal{X}\to \bar{\mathcal{X}}$ de $\mathcal{A}$-schémas telle que $\bar{\mathcal{X}}$ soit propre et même projectif au dessus de $\mathcal{A}$. Si l'on note $\mathcal{X}_i=\mathcal{X}\times_\mathcal{A} \mathcal{A}_i$, $\bar{\mathcal{X}_i}=\bar{\mathcal{X}}\times_\mathcal{A} \mathcal{A}_i$, et encore $E$ l'image de $E$ dans $\bar{\mathcal{X}}^\textrm{an}$, alors par hypothèse $E\cap \mathcal{X}^\textrm{an}_i$ est une partie constructible algébrique de $\mathcal{X}^\textrm{an}_i$, c'est donc l'analytification d'une partie constructible $E'_i\subset \mathcal{X}_i$. Maintenant, $E'_i$ est constructible dans $\overline{\mathcal{X}_i}$, donc pour tout $i$, $E\cap \overline{\mathcal{X}_i}^\textrm{an}={E_i '}^\textrm{an}$ est une partie constructible de $\mathcal{X}_i^\textrm{an}$ et $E$ est une partie $G$-constructible de $\overline{\mathcal{X}}^\textrm{an}$. Par la proposition 10.1.12 de \autocite{ducros2017families}, puisque $\overline{\mathcal{X}}$ peut être choisi de dimension finie, $E$ est une partie constructible de $\overline{\mathcal{X}}^\textrm{an}$, et puisque $\overline{\mathcal{X}}$ est propre, par GAGA, il existe une partie constructible $E_0\subset \overline{\mathcal{X}}$ telle que ${E_0}^\textrm{an}=E$. Maintenant, on a l'égalité $(j^{-1}(E_0))^\textrm{an}=E$, et $j^{-1}(E_0)$ est encore une partie constructible de $\mathcal{X}$, ce qui montre que $E$ est bien une partie constructible algébrique de $\mathcal{X}^\textrm{an}$.

Maintenant, notons $\mathcal{P}$ la propriété pour un morphisme entre espaces $k$-affinoïdes de vérifier les conclusions du lemme présent. Alors on vient de montrer que les $G$-recouvrements, les morphismes plats finis et surjectifs et les extensions de polyrayon vérifiaient la propriété $\mathcal{P}$. On va traiter le cas d'un morphisme fidèlement plat $p:S':=\mathcal{M(B)}\to S=\mathcal{M(A)}$ quelconque en utilisant le même canevas de démonstration que le théorème $\ref{fibfi}$. 

On vérifie facilement que l'on a l'analogue suivant du lemme $\ref{my}$ : si l'on dispose de $f:S\to T$ et $g:R\to S$ des morphismes d'espaces $k$-affinoïdes alors si $f$ et $g$ vérifient la propriété $\mathcal{P}$, alors $f\circ g$ vérifie $\mathcal{P}$ et si $f\circ g$ vérifie $\mathcal{P}$, alors $f$ vérifie $\mathcal{P}$. Maintenant, en utilisant le diagramme $\ref{strict}$, on voit qu'on peut supposer que le corps $k$ est non trivialement valué et que $S$ et $S'$ sont strictement affinoïdes puisque $S_r\to S$ et $S'_r\to S'$ vérifient la propriété $\mathcal{P}$. Un morphisme avec une section vérifie facilement la propriété $\mathcal{P}$, donc en utilisant le théorème de Ducros sur l'existence de multisections plates et affinoïdes, on peut aussi supposer que $p$ est quasi-fini, plat et surjectif. En réutilisant le paragraphe $\ref{redquaset}$, puisque les $G$-recouvrements et les morphismes finis, plats et surjectifs vérifient $\mathcal{P}$, on en déduit qu'il suffit de montrer le théorème pour un morphisme $p:S'\to S$ quasi-étale, plat et surjectif. Maintenant, en raisonnant localement sur $S$ exactement comme dans $\ref{redplusimple}$ et $\ref{redgtop}$, on peut supposer qu'il existe un revêtement fini galoisien $T\to S$ tel que $S'$ est une union finie $\cup S'_i$ où chaque $S'_i$ est un domaine affinoïde d'un quotient $H_i$ de $T$. Par les propriété de compatibilité de la propriété $\mathcal{P}$ à la composition, et le fait que la propriété $\mathcal{P}$ soit vraie pour les morphismes finis, plat et surjectifs, on voit qu'il suffit de traiter le cas d'un $G$-recouvrement, et puisque ceux-ci vérifient la propriété $\mathcal{P}$, tout morphisme fidèlement plat entre espaces $k$-affinoïdes vérifie la propriété $\mathcal{P}$, ce qui conclut la démonstration.
\end{proof}
\begin{theo} \label{algmorph} Considérons $\mathcal{M(B)}\to \mathcal{M(A)}$ un morphisme fidèlement plat entre espaces $k$-affinoïdes. Considérons maintenant $\mathcal{X}$ et $\mathcal{Y}$ deux $\mathcal{A}$-schémas localement de type fini. Alors un morphisme $f:\mathcal{X}^\mathrm{an}\to\mathcal{Y}^\mathrm{an}$ est algébrique si et seulement si son changement de base ${f_\mathcal{B}:\mathcal{X}_\mathcal{B}^\mathrm{an}\to \mathcal{Y}_\mathcal{B}^\mathrm{an} }$ est algébrique.
\end{theo}

\begin{proof}[Démonstration] Si $\mathcal{X}$ est un $\mathcal{A}$-schéma affine de type fini et $\mathcal{Y}=\mathbb{A}_\mathcal{A}^n$ est l'espace affine de dimension $n\in\mathbb{N}$, alors se donner un morphisme de schémas de $\mathcal{X}$ vers $\mathcal{Y}$ revient à se donner un $n$-uplet de fonctions $(f_i)\in \mathcal{O}(\mathcal{X)}^n$, et se donner un morphisme de $\mathcal{X}^{\textrm{an}}\to \mathcal{Y}^{\textrm{an}}$ revient à se donner un $n$-uplet de fonctions de $\mathcal{O}(\mathcal{X}^{\textrm{an}})^n$, et l'analytification d'un morphisme donné par $(f_1,..,f_n)\in \mathcal{O}(\mathcal{X)}^n$ est le morphisme donné par $(f_1,..,f_n)\in \mathcal{O}(\mathcal{X}^{\textrm{an}})^n$ où $f_i\in \mathcal{O}(\mathcal{X}^{\textrm{an}})$ désigne aussi l'image de $f_i$ par l'injection canonique $\mathcal{O}(\mathcal{X})\to \mathcal{O}(\mathcal{X}^{\textrm{an}})$. Le problème se ramène donc à la proposition précédente $\ref{coeur}$.

Maintenant, si $\mathcal{X}$ est un $\mathcal{A}$-schéma localement de type fini quelconque et $\mathcal{Y}=\mathbb{A}_\mathcal{A}^n$ est l'espace affine de dimension $n\in\mathbb{N}$, et qu'on dispose d'un morphisme d'espaces $\mathcal{A}$-analytique $f:\mathcal{X}^{\textrm{an}}\to\mathcal{Y}^{\textrm{an}}$ dont le changement de base $f_\mathcal{B}$ est algébrique, alors si l'on se donne un Zariski-recouvrement $\mathcal{X}=\cup_i \mathcal{X}_i$ par des $\mathcal{A}$-schémas affines de type fini $\mathcal{X}_i$, par hypothèse, il existe un morphisme de $\mathcal{B}$-schémas $h:\mathcal{X}_\mathcal{B}\to \mathcal{Y}_\mathcal{B}$ dont l'analytification est $f_\mathcal{B}$, donc la restriction de $f$ à $\mathcal{X}^{\textrm{an}}_i$ fournit un morphisme d'espace analytique $f_i:\mathcal{X}^{\textrm{an}}_i\to \mathcal{Y}^{\textrm{an}}$ dont le changement de base $f_{i_\mathcal{B}}$ est algébrique car c'est l'analytification de la composée de $h$ avec l'inclusion $\mathcal{X}_{i_\mathcal{B}}\to \mathcal{X}_\mathcal{B}$. Par le paragraphe précédent, on dispose donc d'un morphisme $g_i:\mathcal{X}_i\to\mathcal{Y}$ dont l'analytification est $f_i$. Maintenant, puisqu'on a l'égalité $(g_i\vert_{\mathcal{X_i\cap \mathcal{X_j}}})^{\textrm{an}}=(g_j\vert_{\mathcal{X_i\cap \mathcal{X_j}}})^{\textrm{an}}=f\vert_{\mathcal{X}_i^{\textrm{an}}\cap \mathcal{X}_j^{\textrm{an}}}$, par le lemme $\ref{analfid}$, les $g_i$ se recollent sur les doubles intersection en un morphisme $g:\mathcal{X}\to \mathcal{Y}$ et comme $({g\vert_{\mathcal{X}_i}})^{\textrm{an}}=f_i$, le morphisme $f$ est bien l'analytification de $g$ et est donc bien algébrique.

Puisque qu'on peut toujours plonger un $\mathcal{A}$-schéma affine $\mathcal{Y}$ de type fini dans un espace affine de dimension finie, et que se donner un morphisme de schémas d'un espace $\mathcal{X}$ vers $\mathcal{Y}$ revient alors à se donner un certain nombre de fonctions globales sur $\mathcal{X}$ qui s'annulent sur un fermé, on en déduit que l'on a le résultat pour $\mathcal{X}$ un $\mathcal{A}$-schéma localement de type fini et $\mathcal{Y}$ affine de type fini au dessus de $\mathcal{A}$.

On se donne maintenant deux $\mathcal{A}$-schémas localement de type fini quelconque $\mathcal{X}$ et $\mathcal{Y}$ et ${f:\mathcal{X}^\textrm{an}\to \mathcal{Y}^\textrm{an}}$ un morphisme de $\mathcal{A}$-espaces analytiques dont le changement de base $f_\mathcal{B}$ est algébrique. Alors si $\mathcal{Y}_i$ est un ouvert affine inclus dans $\mathcal{Y}$, alors la partie $f^{-1}({\mathcal{Y}_i}^\textrm{an})$ est un ouvert de $\mathcal{X}^\textrm{an}$ dont l'image inverse par le morphisme $q^\textrm{an}:\mathcal{X}^\textrm{an}_\mathcal{B}\to \mathcal{X}^\textrm{an}$ est égale à ${f_\mathcal{B}}^\textrm{-1}({\mathcal{Y}_i}_\mathcal{B}^\textrm{an})$, et puisque $f_\mathcal{B}$ est algébrique, c'est une partie ouverte algébrique de $\mathcal{X}_\mathcal{B}$, donc par le lemme $\ref{lemmepartie}$, on en déduit que $f^{-1}({\mathcal{Y}_i}^\textrm{an})$ est une partie localement constructible algébrique de $\mathcal{X}^\textrm{an}$ qui est ouverte donc par le lemme 10.1.10 de \autocite{ducros2017families}, c'est un ouvert algébrique de $\mathcal{X}^\textrm{an}$ donc c'est l'analytification d'un ouvert $\mathcal{X}_i$ de $\mathcal{X}$. Par le paragraphe précédent, il existe un unique morphisme $f_i$ de $\mathcal{A}$-schémas de $\mathcal{X}_i$ vers $\mathcal{Y}_i$ dont l'analytification fournit la restriction de $f$ à $\mathcal{X}_i^\textrm{an}$. Par fidélité du foncteur d'analytification, la collection des $(f_i)$ se recolle en un morphisme de $\mathcal{A}$-schémas de $\mathcal{X}$ vers $\mathcal{Y}$ dont l'analytification fournit $f$, et cela montre bien que $f$ est algébrique, ce que l'on voulait. \qedhere

\end{proof}

On conclut ce papier en donnant une conséquence amusante du théorème précédent.
\begin{prop} Considérons une algèbre $k$-affinoïde $\mathcal{A}$, ainsi qu'un $\mathcal{A}$-schéma localement de type fini $\mathcal{X}$. Alors toute section $s:\mathcal{M(A)}\to \mathcal{X}^\mathrm{an}$ du morphisme canonique $\mathcal{X}^\mathrm{an}\to\mathcal{M(A)}$ est algébrique, c'est-à-dire que c'est l'analytification d'une section schématique de $\mathcal{X}\to \Spec \mathcal{A}$

\end{prop}

\begin{proof}
Si le schéma $\mathcal{X}$ est séparé et quasi-compact, il admet une compactification $\overline{\mathcal{X}}$ au dessus de $\Spec\mathcal{A}$, et $s$ fournit une section à la flèche canonique $\overline{\mathcal{X}}^\mathrm{an}\to \mathcal{M(A)}$, qui par GAGA est algébrique, et cela montre que la section $s$ est algébrique.

Maintenant, on se donne $x\in \mathcal{M(A)}$ et un ouvert affine $U\subset \mathcal{X}$ tel que $s(x)\in U^\mathrm{an}$. Alors il existe un voisinage affinoïde $V_x$ de $x$ dans $\mathcal{M(A)}$ tel que $s(V_x)\subset U^\mathrm{an}$. Par le paragraphe précédent, $s_{\vert V_x} :V_x\to (U\times_\mathcal{A} O(V_x))^\mathrm{an}$ est algébrique. On extrait maintenant un $G$-recouvrement de $\mathcal{M(A)}$ par un nombre fini de domaines affinoïdes $\mathcal{M(A)}=\coprod_{x\in I} V_x$. On a montré que $s\times_\mathcal{A} \coprod_{x\in I} V_x$ était algébrique, et par le théorème $\ref{algmorph}$, on en déduit que la section $s$ est algébrique et cela montre la proposition.
\end{proof}

%\printbibliography 

\end{document}